\documentclass[regno, a4paper,12pt]{amsart}

\usepackage{enumerate,graphicx}
\usepackage[mathscr]{eucal}
\usepackage{verbatim,url}
\usepackage{amsthm,amssymb,latexsym}
\usepackage{amsmath}

\usepackage{tikz}
\usepackage{pgf}
\usetikzlibrary{arrows,automata,shapes.geometric,calc}
\input xy
\xyoption{all}
\usepackage[section]{placeins}
\usepackage[font=small,skip=0pt]{caption}

\newtheorem{thm}{Theorem}[section]
\newtheorem{lem}[thm]{Lemma}
\newtheorem{cor}[thm]{Corollary}

\newtheorem{prob}[thm]{Problem}

\theoremstyle{definition}
\newtheorem{df}[thm]{Definition}

\newtheorem{obs}[thm]{Observation}

\DeclareMathOperator\dom{dom}
\DeclareMathOperator\con{Con}
\DeclareMathOperator\im{im}
\DeclareMathOperator\rank{rank}
\DeclareMathOperator\End{End}
\DeclareMathOperator\Aut{Aut}
\DeclareMathOperator\gl{GL}

\DeclareMathOperator\id{id}

\def\th{\theta}

\def\id{\operatorname{id}}

\newcommand{\PT}{\mathcal{PT}}
\newcommand{\In}{\mathcal{I}}
\newcommand{\T}{\mathcal{T}}

\newcommand{\alt}{\mathcal{A}}
\newcommand{\J}{\mathcal{J}}
\newcommand{\Rg}{\mathcal{R}}
\newcommand{\Lg}{\mathcal{L}}
\newcommand{\Hg}{\mathcal{H}}
\newcommand{\Dg}{\mathcal{D}}
\newcommand{\mP}{\mathcal{P}}

\begin{document}
\title[Congruences on Products of Monoids]{Congruences on Direct Products of Transformation and Matrix Monoids}
\author{Jo\~ao Ara\'ujo}
\author{Wolfram Bentz}
\author{Gracinda M.S. Gomes}
\address[Ara\'{u}jo]{{Universidade Aberta, R. Escola Polit\'{e}cnica, 147}\\
  {1269-001 Lisboa, Portugal \& CEMAT-CI\^{E}NCIAS, 
Departamento de Matem\'atica, Faculdade de Ci\^{e}ncias, Universidade de Lisboa,
 1749-016, Lisboa, Portugal},
  {jjaraujo@fc.ul.pt}}
\address[Bentz] {Department of Physics and Mathematics, University of Hull, Hull,
HU6 7RX, UK,
  {W.Bentz@hull.ac.uk}}
\address[Gomes]{{CEMAT-CI\^{E}NCIAS, 
Departamento de Matem\'atica, Faculdade de Ci\^{e}ncias, Universidade de Lisboa,
 1749-016, Lisboa, Portugal},
  {gmcunha@fc.ul.pt}}

\subjclass[2010]{
}
\maketitle

\begin{abstract}
Malcev described the congruences of the monoid $\T_n$ of all full transformations on a finite set $X_n=\{1, \dots,n\}$. Since then, congruences have been
characterized in various other monoids of (partial) transformations on $X_n$, such as the symmetric inverse monoid $\In_n$ of all injective partial transformations, or the monoid $\PT_n$ of all
partial transformations.

The first  aim of this paper is to describe the congruences of the direct products $Q_m\times P_n$, where $Q$ and $P$ belong to  $\{\T, \PT,\In\}$. 

Malcev also provided a similar description of the congruences on the multiplicative monoid $F_n$ of all $n\times n$ matrices with entries in a field $F$; our second aim is provide a description of the principal congruences of $F_m \times F_n$.

The paper finishes with some comments on the congruences of products of more than two transformation semigroups, and a fairly large number of open problems. 
\end{abstract}


%

\section{Introduction}\label{sectintro}

Let $\PT_n$ denote  the monoid of all partial transformations on the set $X_n=\{1,\ldots,n\}$. Let $\T_n$ be full transformation monoid, that is, the semigroup of  all transformations in $\PT_n$  with domain $X_n$; and let $\In_n$ be the symmetric inverse monoid, that is, the semigroup of all $1-1$ maps contained in $\PT_n$. The congruences of these semigroups have been described in the past: Malcev \cite{Malcev} for  $\T_n$, Sutov \cite{Sutov} for  $\PT_n$ and Liber \cite{Liber} for $\In_n$. 

 In this paper we provide a description of the  principal congruences of $Q_n\times Q_m$ (Theorem \ref{th:2h}), where $Q\in\{\PT,\T,\In\}$, and then use this result to provide the full description of all congruences of these semigroups (Theorem \ref{mainQ}).  

Similarly, for a field $F$, denote by $F_n$ the monoid of all $n\times n$  matrices with entries in $F$. The congruences of $F_n$ have been described by Malcev \cite{Malcev2} (see also \cite{MatWeb}). Here we provide a description of the principal congruences of $F_n\times F_m$ (Theorems \ref{th:0hmat}, \ref{th:1hmat}, \ref{lm:equalF2}, and \ref{lm:HHF}). 

It is worth pointing out that the descriptions of the congruences of the semigroups 
\[
S:=\prod_{i\in M} Q_i  \mbox{ and }  T:=\prod_{i\in M} F_i,
\] where $F$ is a field, $M$ is a finite multiset of natural numbers, and $Q\in\{\PT,\T,\In\}$, are in fact yielded by the results of this paper, modulo the use of heavy notation and very long, but not very informative, statements of theorems.  (For more details we refer the reader to Section \ref{triples}.)

It is well known that the description of the congruence classes of a semigroup, contrary to what happens in a group or in a ring,
poses special problems and usually requires very delicate considerations (see \cite[Section 5.3]{Ho95}).  Therefore it is no wonder that the study of congruences  is among the topics attracting more attention when researching semigroups, something well illustrated by the fact that the few  years of this  century already witnessed the publication of more than $250$ research papers  on the topic. 


%

Given the ubiquitous nature of the direct product construction, it comes quite as a surprise to realize that almost nothing is known about congruences on direct products of semigroups, even when the congruences on each factor of the product are known.  Here we start that study trusting that this will be the first contribution in a long sequence of papers describing the congruences of direct products of transformation semigroups. 
Before closing this introduction it is also worth to add that we have been led to this problem, not by the inner appeal of a natural idea (describing the congruences of direct products of very important classes of semigroups whose congruences were already known), but by considerations on the centralizer in $\T_n$ of idempotent transformations.  More about this will be said on the problems section at the end of the paper.


In order to outline the structure of the paper, we now introduce some notation. Let $S$ be a finite  monoid. We say that $a,b\in S$ are $\Hg$-related if 
\[
aS=bS \mbox{ and } Sa=Sb.
\]
 The elements $a,b\in S$ are said to be $\Dg$-related if $SaS=SbS$. Let $F$ be a field. By $F_n$ we denote the monoid of all $n\times n$ matrices with entries in $F$.  
  
In Section \ref{s:pre}, we recall the description of the congruences on $Q_i$, which we use in Section \ref{s:principal} to fully describe the principal congruences on $Q_m \times Q_n$, for $Q\in\{\PT,\T,\In\}$. 
In Section \ref{s:allcong},
we show that a congruence $\th$ on $Q_m \times Q_n$ is determined by those unions of its classes that are also unions of $\Dg$-classes, which we will call $\th$-dlocks.
After presenting the possible types of
$\th$-dlocks, whose properties are related to their $\Hg$-classes, we shall describe $\th$ within each such block, making use of the results obtained in Section \ref{s:principal}. In Section \ref{triples} we give an idea of how the congruences look like on a semigroup of the form $Q_m \times Q_n\times Q_r$. 
 Section \ref{s:matrix} is devoted to the description of the congruences on $F_n$ following \cite{MatWeb}, and we dedicate Section \ref{s:principalF} to characterizing the principal congruences on $F_m \times F_n$. We do so following the pattern of Section \ref{s:principal}, doing the necessary adaptations. Describing the general congruences of $F_m \times F_n$ is notationally heavy but we trust the reader will be convinced that to do so nothing but straightforward adaptions of Section~\ref{s:allcong} are needed. The paper finishes with a set of problems. 

\section{Preliminaries}\label{s:pre}

Let $\In_n$, $\T_n$ and $\PT_n$ be, respectively, the inverse symmetric monoid, the full transformation monoid and the partial transformation monoid, on the set $X_n:=\{1,\ldots ,n\}$. Let $S_n$ denote the symmetric group on $X_n$.
Let $Q$ stand for one of $\mathcal{I}$, $\mathcal{T}$, $\mathcal{PT}$.


For clarity, we start by recalling some well known facts on the Green relations as well as the description of the congruences on an arbitrary $Q_m$. The lattice of congruences of a semigroup $S$ will be denoted by $\con(S)$.
For further details see \cite{Marzo}.

Given $f\in Q_n$ we denote its domain by $\dom (f)$, its image by  $\im (f)$, its kernel by $\ker(f)$ and its rank (the size of the image of $f$) by $|f|$.

\begin{lem}
Let $f,g\in Q_n$. Then
\begin{enumerate}
\item $ f \mathcal{D} g$ iff $|f|= |g|$;
\item $f \mathcal{L} g$ iff $f$ and $g$ have the same image;
\item $f \mathcal{R} g$ iff $f$ and $g$ have the same domain and kernel;
\item $f \mathcal{H} g$ iff $f$ and $g$ have the same domain, kernel, and image.
\end{enumerate}
\end{lem}

Let $S$ be a monoid. A set $I\subseteq S$ is said to be an ideal of $S$ if $SIS\subseteq I$ and  an ideal $I$ is said to be principal if there exists an element $a\in S$ such that $I=SaS$.  It is well known that all ideals in $Q_n$ are principal; in fact, given any ideal $I\leq Q_n$, then $I=Q_naQ_n$, for all transformation $a\in I$ of maximum rank.  In addition we have
\[
Q_naQ_n=\{b\in Q_n\mid |b|\le |a|\}.
\]

The Green relation $\J$ is  defined on a monoid $S$ as follows: for $a,b\in S$,
\[
a\J b \mbox{ iff } SaS=SbS.
\] Thus two elements are $\J$-related if and only if they generate the same principal ideal. In a finite semigroup we have $\J={\mathcal D}$ (and this explains the definition of $\Dg$ used in Section \ref{sectintro}). It is easy to see that if $|f|<|g|$, then $SfS\subseteq SgS$ and in particular $f\in SgS$. That $f\in SgS$ implies $|f|\le |g|$ is also obvious. Therefore $ f \J g $ iff $|f|=|g|$.

If $f \Hg g$ then $f$ and $g$ have the same image, so we may speak about the set image $\im H $ of an $\Hg$-class $H$. Given an  $\Hg$-class $H$ of $Q_n$, we can fix an arbitrary linear order on
$\im H$, say that
$\im f=\{a_1 <\dots < a_{|f|}\}$ for all $f \in H$. We define a right action $\cdot$ of the group $S_{i}$ with $i=|f|$ on all elements in $Q_n$ of rank $i$: let $\omega \in S_i$ and $x \in \dom(f)$, then
$(x) f\cdot \omega = a_{j \omega}$ where $x f=a_j$ with regard to the fixed ordering associated with the $\Hg$-class of $f$. Note that the action $\cdot$ preserves $\Hg$-classes.
Hence for each $w \in S_i$ and $\Hg$-class $H$ with $\im H=i$, we may define $\bar \omega^H\in S_{\im H}$ by  $f \cdot \omega= f \bar{\omega}^H$, for all $f \in H$.

The description of the congruences of $Q_n$ can be found in \cite[sec. 6.3.15]{Marzo}, and goes as follows.
\begin{thm}\label{th:congrQn}
A non-universal congruence   of $Q_n$ is associated with a pair $(k,N)$, where $1\le k\le n$, and $N$ is a normal subgroup of $S_k$; and it is of the form $\th(k,N)$  defined as follows: for all $f,g\in Q_n$,
\[
f\theta(k,N) g \mbox{ \ iff \ }
\left\{
\begin{array}{l}
f=g \mbox{ and } |f|>k, \mbox{ or}\\
|f|,|g|<k, \mbox{ or}\\
|f|=|g|=k, f{\mathcal H}g \mbox{ and }  f= g\cdot \omega, \mbox{ where $\omega\in N$.}
\end{array}
\right.
\]
\end{thm}
We write $\th=\th(k,N)$ or just $\th$ if there is no ambiguity.
It follows from the normality of $N$ that the  definition of $\th(k,N)$ is independent of the ordering associated with each of the $\Hg$-classes of $Q_n$. A similar independence result will hold for a corresponding
construction in our main result.

 The following will be applied later without further reference. Let $g,g' \in Q_n$. It follows from Theorem \ref{th:congrQn}, that if $(g,g') \in \Hg$, then the principal congruence $\th$ generated by $(g,g')$ is $\th(|g|,N)$, where $N$ is the normal subgroup of $S_{|g|}$ generated by
$\sigma \in S_{|g|}$ with $g'=g\cdot \sigma$, with respect to a fixed ordering of the image of $g$. If $(g,g') \notin \Hg$ then $\th$ is $\th(k+1,\{\id_{S_{k+1}}\})$, where $k=\max\{|g|,|g'|\}$, which is the Rees
congruence defined by the ideal $I_k$ of all transformations of rank less or equal to $k$, i.e. $\th_{I_k}$.

From this description we see that if $\th=\theta(k,N)\in \con (Q_n)$ and there exist $f,g\in Q_n$, with $|f|<|g|$ and  $(f,g)\in \theta$, then $|g|<k$ and the ideal generated by $g$ is contained in a single $\theta$-class.


For each $n>1$ the congruences on each semigroup $Q_n$ form a chain \cite[sec.~6.5.1]{Marzo}. Let $\iota_S$ and $\omega_S$ be, respectively, the trivial and the universal congruences on $S$. For $k\in \{1,\ldots,n\}$,  denote by $\equiv_{\varepsilon_k}$, $\equiv_{\alt_k}$ and $\equiv_{S_k}$, the congruence associated to $k$ and to the trivial, the alternating and  the symmetric subgroup of $S _k$,  respectively. Finally, for $k=4$, let $\equiv_{V_4}$ be the congruence associated with the Klein $4$-group.
We have
\[
\iota_S=\equiv_{S _1}\subset \equiv_{\varepsilon_2}\subset \equiv_{S _2}\subset \equiv_{\varepsilon_3}\subset  \equiv_{\alt_3}\subset  \equiv_{S _3}\]
\[ \subset \equiv_{\varepsilon_4}
\subset \equiv_{V_4}\subset \equiv_{\alt_4} \subset  \equiv_{S _4} \subset \ldots \subset \equiv_{\varepsilon_n}\subset  \equiv_{\alt_n}\subset  \equiv_{S _n}\subset \omega_S.
\]
Let $0*$
stand for $1$ if $Q= \mathcal{T}$ and for $0$ in the other cases.
For $0* \le i \le m$, let $I_i^{(m)}$ stand for the ideal of $Q_m$  consisting of all functions
$f$ with $|f| \le i$. We will usually just write $I_i$ if $m$ is deducible from the context. Let $\theta_{I_i}$ stand for the Rees congruence on $Q_m$ defined by $I_i$.


\section{Principal congruences on $Q_m\times Q_n$} \label{s:principal}

The aim of this section is to describe the principal congruences of $Q_m\times Q_n$, when $Q\in \{\In,\T,\PT\}$. We will start by transferring our notations to the setting of this product semigroup.

For functions $f \in Q_m \cup Q_n$, let $|f|$ once again stand for the size of the image of $f$, and for $(f,g) \in Q_m\times Q_n$ let $|(f,g)| = (|f|,|g|)$, where we order these pairs according to the partial order  $\le \times \le$.
Throughout, $\pi_1$ and $\pi_2$ denote the projections to the first and second factor.

We will start with some general lemmas about congruences on $Q_m\times Q_n$. In case there is no danger of ambiguity, we will use the shorthand $P$ for $Q_m\times Q_n$, to 
simplify the writing. 

\begin{lem} \label{lem:single} Let  $\theta$ be a congruence of  $Q_m\times Q_n$ and fix  $f\in Q_m$; let
$$\theta_f:=\{ (g,g')\in Q_n\times Q_n\mid (f,g)\theta (f,g')\}.$$

Then
\begin{enumerate}
\item $\theta_f$ is a congruence on $Q_n$;
\item if $f'\in Q_m$ and $|f'|\le |f|$, then $\theta_f\subseteq \theta_{f'}$;
\item if $|f|= |f'|$ in $Q_m$, then $\theta_f=\theta_{f'}$.
\end{enumerate}
\end{lem}
\begin{proof}
(1) That $\theta_f$ is an equivalence on $Q_n$ is clear. The compatibility follows from the fact that $Q_m$ has  an identity; indeed
$$(g,g')\in\theta_f\Rightarrow (f,g)\theta (f,g')\Rightarrow (f,g)(1,h)\theta (f,g')(1,h)$$
$$\Rightarrow  (f,gh)\theta (f,g'h)\Rightarrow gh\theta_fg'h.$$
Similarly we prove the left compatibility.

(2) If $|f'| \le |f|$ then $f'\in Q_mfQ_m$ and hence we have  $f'=hfh'$, for some $h,h'\in Q_m$;
in addition $(f,g) \theta (f,g')$ implies $$(f',g)=(h,1)(f,g)(h',1) \theta (h,1)(f,g')(h',1)=(f',g')$$
 so that $\theta_f \subseteq \theta_{f'}$. Condition (3) follows from (2) and its symmetric.

\end{proof}
In a similar way, given a congruence $\th$ on $Q_m \times Q_n$ and fixed $g \in Q_n$, we define $\theta_g:=\{ (f,f')\in Q_m\times Q_m\mid (f,g)\theta (f',g)\}.$

The next result describes the ideals of $Q_n\times Q_m$.

\begin{lem}\label{lem:ideals1}
Let $Q\in \{\In,\T,\PT\}$. The ideals of $P:=Q_m\times Q_n$ are exactly the unions of sets of the form $I^{(m)}_i\times I^{(n)}_j$, where $I^{(m)}_i$ and $I^{(n)}_j$ are ideals of $Q_m$ and $Q_n$, respectively.
\end{lem}
\begin{proof}
That any union of sets of the form $I^{(m)}_i\times I^{(n)}_j$ is an ideal of $P$ is obvious.

Conversely, let $I$ be an ideal of $P$, and $(f,g)\in I$. Then, by the definition of an ideal of $P$, we have $P(f,g)P \subseteq I$,  for every $(f,g)\in I$.  Let $(f',g')\in I_{|f|}\times I_{|g|}$. Then $f'\in Q_m fQ_m$ and $g' \in Q_ngQ_n$ so that $(f',g')\in P(f,g)P\subseteq I$. It follows that $\cup_{(f,g)\in I}I_{|f|}\times I_{|g|} \subseteq I$. Regarding the reverse inclusion, let $(f,g)\in I$; it is self-evident that $(f,g)\in I_{|f|}\times I_{|g|}$ and hence   $\cup_{(f,g)\in I}I_{|f|}\times I_{|g|} \supseteq I$. The result follows.
\end{proof}

\begin{lem} \label{lem:ideals}
Let $\theta$ be a congruence of $P:=Q_m\times Q_n$.
\begin{enumerate}
\item If $Q\in \{\mathcal{PT},\In\}$, then  $\theta $  contains a  class $I_\theta$ which is an ideal;
\item If $Q=\mathcal{T}$
and  both $\pi_1(\th)$ and $\pi_2(\th)$ are non-trivial, then $\th$ contains a class $I_\theta$ which is an ideal;
\item $\th$ contains at most one  ideal class.
\end{enumerate}
\end{lem}
\begin{proof}(1)
If $Q$ is $\mathcal{PT}$ or $\mathcal{I}$, then $P$ has a zero, whose congruence class is easily seen to be a unique ideal of $P$.

(2) If $Q$ is $\mathcal{T}$, let $c_a$ denote the constant map whose  image is $\{a\}$, for some $a\in \{1,\ldots,n\}$. Since $\pi_2(\th)$ is non-trivial, it follows that there exist  $f, f' \in Q_m$
and distinct  $g,g' \in Q_n$ such that $(f,g) \th (f',g')$. Thus  $(f,g)(c_a,1) \th (f',g')(c_a,1)$, that is, $(c_a, g)  \th  (c_a, g')$,
so that $\th_{c_a}$ is non-trivial. By Theorem \ref{th:congrQn}, the ideal $I^{(n)}_1$ is contained in a class of $\th_{c_a}$. Similarly, we pick $b\in \{1,\ldots,n\}$ concluding that  $I^{(m)}_1 \times\{c_b\}$ lies in a $\theta$-class.

We claim that $I^{(m)}_1 \times I^{(n)}_1$ is contained in one $\theta$-class. In fact, let $(c_a,c_b),(c_d,c_e)\in I^{(m)}_1 \times I^{(n)}_1$. Then $(c_a,c_d)\in \theta_{c_e}$, since   $I^{(m)}_1 \times\{c_e\}$ lies in a $\theta$-class, and similarly $(c_b,c_e)\in \theta_{c_a}$. Thus $(c_a,c_e)\theta(c_a,c_b)$ and $(c_d,c_e)\theta(c_a,c_e)$.  Therefore  by transitivity  we get $(c_a,c_b)\theta(c_d,c_e)$. We have proved that $I^{(m)}_1 \times I^{(n)}_1$ is contained in the $\theta$-class of $(c_a,c_b)$.

Conversely, given any $(f,g)$ in the $\th$-class of $ (c_a,c_b)$ and any $(f',g')\in P$, we have $$(ff',gg')\theta(c_af,c_bg)\mbox{ and }(c_af,c_b) \in I^{(m)}_1 \times I^{(n)}_1\subseteq [(c_a,c_b)]_\theta,$$
so $(ff',gg') \in [(c_a,c_b)]_\th$;
similarly we prove that
$(f'f,g'g)\in [(c_a,c_b)]_\theta$. Thus $[(c_a,c_b)]_\theta$ is an ideal.

(3) The last assertion holds in all semigroups, as any ideal class is a (necessarily unique) zero element of the quotient semigroup $P/\th$.
%
\end{proof}

For the remains of this section we fix the following notations.
Let $Q\in  \{\T,\mathcal{PT},\In\}$, $f,f'\in Q_m$ and $g,g'\in Q_n$. Let $\theta$ be a principal congruence on $Q_m\times Q_n$ generated by $((f,g),(f',g'))$. Let $\theta_1$ be the principal congruence generated by $(f,f')$ in $Q_m$ and $\theta_2$ be the principal congruence generated by $(g,g')$ in $Q_n$.


\begin{lem}\label{lm:equal} Let $\theta$ be a principal congruence on $Q_m\times Q_n$ generated by $((f,g),(f',g'))$.
If $f=f'$ and $g\ne g'$, we have $(a,b)\theta(c,d)$ if and only if  $(a,b)= (c,d)$ or $a=c$, $|a|\le |f|$ and $b\,\theta_2 \,d$.
\end{lem}
\begin{proof}
Let $\th'$ be the binary relation defined by the statement of the lemma. If $|a|\le |f|$,  then $a=ufv$, for some $u,v\in Q_m$ and hence
$$((a,g), (a,g') )= ((ufv,1g1), (uf'v,1g'1)) \in \th.$$ It follows that
$(g,g') \in \th_a$ whence $\th_2 \subseteq \th_a$. Now if $(b,d) \in \th_2$, then $(b,d) \in \th_a$, and therefore $(a,b) \th (a,d)$.
 Hence $\th' \subseteq \th$.

Conversely, it is straightforward to check that $\th'$ is a congruence containing $((f,g),(f,g'))$, thus
$\th \subseteq \th'$. The result follows.
\end{proof}
The following corollary gives a more direct description of the congruences covered by Lemma \ref{lm:equal} by incorporating the structure of the congruence $\th_2$ on $Q_n$. By applying Theorem \ref{th:congrQn}, we obtain
\begin{cor}\label{c:equal}
Let $\theta$ be a principal congruence on $Q_m\times Q_n$ generated by $((f,g),(f,g'))$.  If $(g,g') \notin \Hg$, let $k=\max\{|g|,|g'|\}$. Then $\th_2$ is the Rees congruence
$\th_{I_k}$ and $(a,b) \th (c,d)$ if and only if one of the following holds:
\begin{enumerate}
  \item $(a,b)=(c,d)$, $|a|=|c| > |f|$ or $|b|=|d| > k$;
  \item $a=c$ and $|a|=|c| \le |f|$, $|b|, |d| \le k$.
\end{enumerate}
If $(g,g') \in \Hg$ and $g \ne g'$ then $\th_2= \th(k,N)$, for $k =|g|$ and $N= \langle \sigma \rangle$, where $g'=g \cdot \sigma$ with $\sigma \in S_{k}$.
Moreover, $(a,b) \th (c,d)$ if and only if one of the following holds:
\begin{enumerate}
  \item $(a,b)=(c,d)$, $|a|=|c| > |f|$ or $|b|=|d|>k$;
  \item $a=c$, $|b|=|d|=k$, $b \Hg d$ and $d=b \cdot \omega$ for some $\omega \in N$;
  \item $a=c$, $|a|=|c| \le |f|$, $|b|, |d| < |g|$.
\end{enumerate}
\end{cor}

\begin{lem} \label{lem:Icont}
Let $\theta$ be a principal congruence on $Q_m\times Q_n$ generated by $((f,g),(f',g'))$, and let $j = \max \{|f|,|f'|\}$.
If $g\ne g'$ and $(f,f')\not\in \mathcal{H} $,  then $\th_g$ or $\theta_{g'}$ contains the Rees congruence $\theta_{I_j}$ of $Q_m$.
\end{lem}
\begin{proof}
We will show that for $j=|f'|$, $\th_{I_j} \subseteq \th_{g'}$. An anolog result for $j=|f|$ follows symmetrically. So let us assume that $|f| \le |f'|=j$.
As $(f,f')  \not\in  \mathcal{H}$, $f$ and $f'$ must differ in either image or kernel. We consider
 two cases.

 {First case}:  $\im f \ne \im f'$.

As $|f'|=j\ge |f|$, $\im f' \not \subseteq \im f$. As $Q_m$ is regular, there exists an idempotent $h\in Q_m$ such that $h \Lg f$. Hence $\im h= \im f$, and as $h$ is idempotent,
$fh=f$.

We have that $$(f'h,g') = (f',g') (h, 1)\th  (f,g) (h, 1)=(f h,g) =(f,g) \th (f',g')$$ and
so $(f'h, f') \in \theta_{g'}$. As $\im f' \not \subseteq \im f$, and $\im h=\im f$, the maps  $f'h$ and $f'$ have different images.
It follows that $(f'h,f') \notin\Hg$. Now, the congruence $\th'$ generated by $(f'h,f')$ is contained in $\th_{g'}$ and by Theorem \ref{th:congrQn}, we have $\th'=\th_{I_j}$. We get $\th_{I_j} \subseteq \theta_{g'}$.


%

{Second case}:  $\ker f \ne \ker f'$.

Now $|f'|=j \ge |f|$, implies that $\ker f \not\subseteq \ker f'$. As above, the regularity of $Q_m$ implies that there exists an idempotent $h$ that is $\Rg$-related to $f$; thus $h$ and $f$ have the same kernel.
 Hence
$$(hf',g')= (h,1) (f',g') \theta (h,1)(f,g)=(hf,g)=(f,g) \theta (f',g')$$
and so $(hf',f')\in \theta_{g'}$.
Now $\ker f = \ker h \subseteq \ker(hf')$. As $\ker f \not \subseteq \ker f'$, $(hf',f')\not\in \Hg$.  As above, by Theorem \ref{th:congrQn}, we get $\th_{I_{|f'|}}=\th_{I_j} \subseteq \theta_{g'}$.
\end{proof}

\begin{thm}\label{theorem2.5}\label{th:0h}
Let $\theta$ be the congruence on $Q_m \times Q_n$ generated by $((f,g), (f',g'))$, and assume that $(f,f') \not\in \mathcal{H} $, $(g,g') \not\in \mathcal{H} $, $|f|=i, |f'|=j, |g|=k, |g'|=l$.
Then $\theta$ is the Rees congruence on $Q_m\times Q_n$ defined by the ideal $I=I_i\times I_k
\cup  I_j \times I_l$.
\end{thm}

\begin{proof}
If  $Q=\T$, then  pick two arbitrary constants $z=c_a$ and $z'=c_b$ in  $Q_m$ and $Q_n$, respectively. If $Q\in \{\PT,\In\}$, let  $z, z'$ be the empty maps in $Q_m$ and $Q_n$.

As $f\ne f'$ and $g\ne g'$,
by Lemma \ref{lem:ideals}, the congruence $\theta$ contains an ideal class $K$. As $(z,z')$ lies in the smallest ideal $I_{0*} \times I_{0*}$ of $P$,  $(z,z') \in K$. We claim that $(f,g) \in K$.

To show this, note that by Lemma \ref{lem:Icont}, either  $\theta_g$ or $\theta_{g'}$  contains the Rees congruence $\theta_{I_{\max\{i,j\}}}$. The dual of Lemma \ref{lem:Icont} guarantees that either  $\theta_{I_{\max\{k,l\}}} \subseteq \theta_f$ or $\theta_{I_{\max\{k,l\}}} \subseteq \theta_f'$.
Up to symmetry, there are two cases.

{First case}:   $\theta_{I_{\max\{k,l\}}} \subseteq \theta_f$, $\theta_{I_{\max\{i,j\}}} \subseteq \theta_{g}$

We have that $g, z'\in I_{\max\{k,l\}}$, so $(g,z') \in \theta_{I_{\max\{k,l\}}} \subseteq \theta_f$, that is,  $(f,g)\theta (f,z')$.

 As $f,z\in I_{\max\{i,j\}}$, an anolog argument shows that $(f,z)\in \theta_g$. By Lemma \ref{lem:single}, we have $\theta_g\subseteq \theta_{z'}$, and so $(f,z)\in \theta_{z'}$. Thus $(f,z')\theta(z,z')$. Therefore
  $$(f,g)\theta (f,z') \theta (z,z')\in K.$$



{Second case}:   $\theta_{I_{\max\{k,l\}}} \subseteq \theta_f$, $\theta_{I_{\max\{i,j\}}} \subseteq \theta_{g'}$


We have $g,z'\in I_{\max\{k,l\}}$ so $(g,z')\in \theta_f$ and similarly $(f',z)\in \theta_{g'}$. Thus
$$(f,z') \th (f,g) \th (f',g') \th (z,g').$$

Let $h \in S_m$ be such that $fhf=f$ (such $h$ clearly exists). We then have
\begin{eqnarray*}
 (f,g) \th (f,z')&=&(fhf,z')\\
 &=&(f,z')(h,z')(f,z')\\& \th &(f,z')(h,z')(z,g') \\
 &=&(z,z') \;\in I_{0*}\times I_{0*} \subseteq K.
\end{eqnarray*}
Hence in both cases $(f,g)\in K$. As $K$ is a class of $\th$, then $(f',g') \in K$ as well. Now $I\subseteq K$, as $I$ is the smallest ideal containing $\{(f,g),(f',g')\}$.
It follows that $\th_I \subseteq \th_K \subseteq \th$. Conversely $\th \subseteq \th_I$  as $\th$ is generated by $((f,g),(f',g'))$ and $\th_I$ is a congruence that contains  $((f,g),(f',g')$. Hence $\th=\th_I.$
\end{proof}

\begin{cor}
Under the conditions of Theorem \ref{theorem2.5}, if $i \le j$ and $k \le l$ then $\theta=\theta_{I_j \times I_l}$.
\end{cor}

\begin{thm}\label{th:1h}
Let $\theta$ be the congruence on $Q_m \times Q_n$ generated by $((f,g), (f',g'))$, and let $\th_2$ be the congruence on $Q_n$ generated by $(g,g')$.
If $g\ne g'$, $(g,g')\in \Hg$ and $(f,f')\not\in \Hg$; let $j=\mbox{max}\{|f|,|f'|\}$ and $k=|g|=|g'|$.
 Then  $(a,b) \th (c,d)$
 if and only if $(a,b)=(c,d)$ or
$|a|, |c| \le j$, $|b|, |d| \le k$, $b \th_2 d$.
\end{thm}
\begin{proof}
We start with some considerations having in mind the initial conditions.
By the dual of  Lemma \ref{lem:single}, we have $\theta_g=\theta _{g'} \subseteq \theta_h$ for all $h\in Q_n$ with $|h| \le k$. By  Lemma \ref{lem:Icont}, we get $\th_{I_j} \subseteq \th_g=\th_{g'}$, hence that $\theta_{I_j}\subseteq \theta_h$ for each such $h$.  Assume w.l.o.g. that
$|f'|=j$. Now $f,f'\in I_j$ and $\theta_{I_j}\subseteq \theta_g$, so $f\theta_g f'$. Thus
$$(f',g') \th (f,g) \th (f',g).$$
Then $(g,g')\in \theta_{f'}$ and therefore $\theta_2 \subseteq \theta_{f'}$, as $\theta_2$ is the congruence generated by $(g,g')$. By Lemma \ref{lem:single}, $\theta_{f'}\subseteq \theta_u$, for all $u\in Q_m$ such that $|u|\le |f'|=j$.        Thus  $\theta_2 \subseteq \theta_{u}$ for all $u \in Q_m$ with $|u| \le j$.

Let $\th'$ be the relation on $Q_m \times Q_n$ defined by  $(a,b) \th' (c,d)$
 if and only if $(a,b)=(c,d)$ or
$|a|, |c| \le j$, $|b|, |d| \le k$, $b \th_2 d$. We want to show that  $\th=\th'$.

Assume  that $|a|, |c| \le j$, $|b|, |d| \le k$ and  $b \th_2 d$. Taking $u=a$ we obtain  $\theta_2\subseteq \theta_a$.  As $b\theta_2 d$, we get $(a,b)\theta (a,d)$.
 Now $|d| \le k$ which, as mentioned at the beginning of the proof, implies that $\th_{I_j}\subseteq\th_d$. It follows that $(a,c)\in \theta_d$, that is $(a,d)\theta(c,d)$. Therefore $(a,b)\th (a,d) \theta(c,d)$, and so
 $\th' \subseteq \th$.

For the reverse inclusion, it suffices to check that $\th'$ is a congruence containing $((f,g),(f',g'))$. We leave this straightforward verification to the reader.
\end{proof}
Notice that we can once again give a more explicit description of $\th$ by incorporating the classification of $\th_2$ given by Theorem \ref{th:congrQn}.
\begin{cor}\label{c:1h}
Let $(f,g),(f',g') \in Q_m \times Q_n$, such that $g\ne g'$, $(g,g')\in \Hg$ and $(f,f')\not\in \Hg$.
Let $j=\mbox{max}\{|f|,|f'|\}$ and $k=|g|=|g'|$. Let $g'=g \cdot \sigma$ for $\sigma \in S_{k}$ with regard to some ordering
 associated with the $\Hg$-class of $g$, and let $N$ be the normal subgroup of $S_k$ generated by $\sigma$.
If $\theta$ is the congruence on $Q_m \times Q_n$ generated by $((f,g), (f',g'))$, then $(a,b) \th (c,d)$
 if and only if one of the following holds:
\begin{enumerate}
  \item $(a,b)=(c,d)$ for $|a| >j$ or $|b|>k$;
  \item $|a|,|c| \le j$, $|b|=k$, $b \Hg d$ and $d=b \cdot \omega$ for some $\omega \in N$, and with regard to some ordering associated with the $\Hg$-class of $d$;
  \item $|a|, |c| \le j$ and $|b|, |d| < k$.
\end{enumerate}
\end{cor}

We remark that there are obvious dual versions of  Lemma \ref{lm:equal} and Theorem \ref{th:1h} obtained by switching the roles of the coordinates. Apart from the trivial case that $(f,g)=(f',g')$, it remains to determine the principal congruence $\th$ when $f \ne f'$, $g \ne g'$,
$f \mathcal{H}  f'$, $g \mathcal{H}  g'$.

We will first extend the actions $\cdot$ of $S_i$ on $Q_m$ and of $S_j$ on $Q_n$ to a partial action of $S_i \times S_j$ on $Q_m \times Q_n$.
We  define
the action $\cdot$ of $S_i \times S_j$ on the set
$$D_{i,j}=\{(f,g) \in Q_m \times Q_n \mbox{ such that }|f|=i, |g|=j\}$$ by setting $(f,g)\cdot (\omega_1, \omega_2)= (f \cdot \omega_1, g \cdot \omega_2),$ for all $(f,g) \in D_{i,j}$
and $(\omega_1, \omega_2) \in S_i \times S_j$,
where in the first component $\cdot$ is applied with respect to the ordering of the $\Hg$-class of $f$ within $Q_m$, and correspondingly in the second component.

As $\Hg$-classes of $Q_m \times Q_n$ are products of $\Hg$-classes of $Q_m$ and of $Q_n$, it follows that the action $\cdot$ preserves $\Hg$-classes. In addition, the action  $\cdot$ is
transitive on each $\Hg$-class. If $H_f$ and $H_g$ stand for the $\Hg$-classes of $f$ in $Q_m$ and of $g$ in $Q_n$,
we have  $(f,g) \cdot (\omega_1, \omega_2)=(f\bar{\omega}_1^{H_f}, f\bar{\omega}_2^{H_g})$, where $\bar{\omega}_1^{H_f} \in S_{\im f}$ and $ \bar{\omega}_2^{H_g} \in S_{\im g}$ are as defined  before Theorem \ref{th:congrQn}.
In this context, we will always consider $S_{\im f} \times S_{\im g}$ to be a subgroup of $S_m \times S_n$ in the natural way.


\begin{thm}\label{th:2h}
Let $\th$ be the principal congruence on $Q_m \times Q_n$ generated by $((f,g),(f',g'))$.
If  $f \ne f'$, $g \ne g'$,
$f \mathcal{H}  f'$, $g \mathcal{H}  g'$, $|f|=i =|f'|, |g|=k = |g'|$, let $\sigma_1 \in S_i$ and $\sigma_2 \in S_k$ be such that $f\cdot\sigma_1=f'$ and $g\cdot\sigma_2=g'$.
 Let $N$ be the normal subgroup of $S_i \times S_k$ generated by the pair $(\sigma_1, \sigma_2)$.

 Then  $(a,b) \th (c,d)$
 if and only $(a,b)=(c,d) $ or one of the following hold:
 \begin{enumerate}
 \item $|a|, |c| \le i-1$, $|b|, |d| \le k -1$; \label{casea}
 \item $|a|= |c| =i$, $|b|, |d| \le k-1 $, and $a \theta_1 c$;\label{caseb}
 \item $|a|, |c| \le i-1 $, $|b|, |d| = k$, and $b \theta_2 d$;\label{casec}
 \item $|a|=|c|=i, |b|=|d|=k$, $a \mathcal{H} c$, $b  \mathcal{H} d$ and there exist $(\tau_1, \tau_2) \in N$ such that $a \cdot\tau_1 =c$, $b\cdot \tau_2= d$. \label{cased}
 \end{enumerate}

\end{thm}
\begin{proof}
First note that if $(a,b) \th (c,d)$, $(a,b) \ne (c,d)$, and $\th$ is generated by $((f,g),(f',g'))$, we must have $|a|\le |f|=i$, $|b|\le |g|=k$, $|c|\le |f'|=i$, $|d|\le |g'|=k$

As $f \mathcal{H} f'$, then $f$ and $f'$ have the same kernel and  image. Together with $f\ne f'$, this implies $|f|=i \ge 2$. Similarly $k\ge 2$. Let $\ker f= \ker f'=\{K_1, K_2, \dots, K_i\}$.

First assume that $Q \in \{\mathcal{T}, \mathcal{PT}\}$.

Suppose that $i \ge 3$.
As $f \ne f'$ there exists $x \in \im f=\im f'$  such that the associated kernel classes are different, i.e. $f^{-1}(x) \ne f'^{-1}(x)$.
W.l.o.g. assume $K_1= f^{-1}(x)$ and $K_2= f'^{-1}(x)$. Let $\{y\}=K_3f$ so that $f^{-1}(y)=K_3$.

 Let $h \in T_n$ be such that $yh=x$ and it is identical otherwise. Then $(fh, g)\, \th\, (f'h,g')$, where $|fh|=|f'h|=i-1$ and $fh$ and $f'h$ have different kernel, since the preimages of $x$ are
 $K_1 \dot\cup  K_3$ and $K_2 \dot\cup K_3$, respectively.


 Let $\beta $ be the congruence  generated by the pair $((fh, g),(f'h,g'))$. As $(fh,f'h) \notin \Hg$, $g \ne g'$, $g \Hg g'$,
 Theorem \ref{th:1h} is applicable to $\beta$. Thus $(a,b) \beta (c,d)$ for $(a,b) \ne (c,d)$ if and only if
 $|a|,|c| \le i-1$, $|b|,|d| \le k$, and $b \th_{2,\beta} d$, where  $\th_{2, \beta}$ is the $Q_n$-congruence generated by $(g,g')$ and hence is equal to $\th_2$.
 By Theorem \ref{th:congrQn}, the relation $\th_{2}$ restricted to $I_{k-1}$ is the universal relation.
  Therefore the pairs $((a,b),(c,d))$ that  satisfy condition (\ref{casea}) or (\ref{casec}) are in $\beta$, but  $\beta$'s generating pair is in $ \th$, and so they are in $\th$, as well.

Suppose now that $i=2$. Let $(K_1)f=x_1$ and $ (K_2)f= x_2$. Then $ (K_1)f'=x_2,$ and $(K_2)f'= x_1$. Let $h$ be a total map with image contained in $K_1$. Then
$(hf,g) \th (hf', g')$ where $|hf|=|hf'|=1$, but $\im  hf \ne \im hf'$. Thus $(hf, hf') \notin \Hg$, and by Theorem \ref{th:1h}, as before, we conclude that $\th $ must contain all pairs satisfying conditions
(\ref{casea}) or~(\ref{casec}).

Next suppose that $Q= \mathcal{I}$. We have $i\ge2$. In this case, there exists $x \in \im f= \im f'$ such that $f^{-1}(x) \ne f'^{-1}(x)$ (notice that these sets are now singletons).
Let $h \in \mathcal{I}_m$ be the identity map with domain $\{1, \dots,m\}\setminus \{x\}$.
Then $(fh,g)\th(f'h,g')$ where $|fh|=|f'h|=i-1$ and $\dom fh \ne \dom fh'$. Once again applying Theorem \ref{th:1h}, we conclude that $\theta$ must contain
all pairs satisfying conditions (\ref{casea}) or (\ref{casec}).

By symmetrically applying the above considerations to the  second argument,
we also show that $\theta$ contains the pairs that satisfy condition~(\ref{caseb}).

The next step is to prove that, for any $Q \in\{\mathcal{T}, \mathcal{PT}, \mathcal{I}\}$ the pairs that satisfy condition (\ref{cased}) are also in $\th$.

Note that the group of units of $Q_m \times Q_n$ is $S_m \times S_n$. We can choose $(u, v) \in S_m \times S_n$ such that
both $uf$ and $vg$ are idempotent transformations. Now $(uf,vg)= (u,v)(f,g) \th (u,v)(f',g')=(uf',vg')$, and clearly $(uf)\cdot \sigma_1= uf'$ and  $(vg)\cdot \sigma_2= vg'$.
 Hence we may assume w.l.o.g. that $f$, $g$ are idempotents.

Let $H$ be the $\Hg$-class of $(f,g)$. Then $(f',g') \in H$. As $H$ contains an idempotent, $H$ is a group.
Moreover, it is easy to see that $\phi$ given by $(\omega_1, \omega_2) \phi= (f \cdot \omega_1, g\cdot \omega_2)$ is an isomorphism from $S_i \times S_k$  to $H$.

Let $\th'$ be the restriction of $\th$ to $H$, then $\th'$ is a congruence on a group. Let $K'$ be the  normal subgroup of $H$ corresponding to $\th'$, and $K'=\bar K\phi^{-1}$.
As an idempotent, $(f,g)$ is the identity of $H$, so $(f,g)\in K'$ and hence $(f',g')\in K'$, as $(f,g) \th' (f',g')$. Applying $\phi^{-1}$, we get that
$(\sigma_1, \sigma_2) \in K$. As $K$ is a normal subgroup of $S_i \times S_k$, we obtain $N \subseteq K$.

Now let us take  a pair $((a,b), (c,d))$ that satisfies (\ref{cased}), that is, $|a|=|c|=i, |b|=|d|=k$, $a \mathcal{H} c$, $b  \mathcal{H} d$ such that $a \cdot\tau_1 =c$, $b\cdot \tau_2= d$ for some $(\tau_1, \tau_2) \in N$.
 It remains to show that $(a,b) \th (c,d)$.

 As $|a|=|f|$ and $|b|=|g|$, we have $a \mathcal{J} f$, and
$b \mathcal{J} g$, so there exists $h_1,h_2 \in Q_m$ and $h_3,h_4 \in Q_n$ such that
$a=h_1 f h_2 $, $b=h_3 g h_4 $. Once again as $|a|=|f|$ and $|b|=|g|$, $h_2 |_{\im f}$ is an injection and so is $h_4|_{\im f}$. Hence w.l.o.g. we may assume that $h_2 \in S_m$, $h_4 \in S_n$.

Note that as $(\tau_1, \tau_2) \in N$, we have $(\tau_1, \tau_2)\in K$, since $N \subseteq K$.
Recall that $\bar \tau_1^{\pi_1 (H)} \in S_{\im \pi_1(H)}$ denotes the function such that $h  \bar \tau_1^{\pi_1 (H)}= h \cdot \tau_s$ for all $h \in \pi_1(H)$. We will write $\bar \tau_1$ for the extension of
$ \bar \tau_s^{\pi_1 (H)}$ to $S_m$ that is the identity on $\{1,\dots,m\} \setminus \im H$, and use corresponding notation if $\tau_1$ is replaced by other elements of $S_i$ or $S_k$.

Consider
$y=(f h_2 \bar \tau_1 h_2^{-1}, g h_4 \bar \tau_2 h_4^{-1})$. It is straightforward to check that
$y \in H$, whence $y=(f \cdot \omega_1, g \cdot \omega_2)$ for some $(\omega_1, \omega_2) \in S_i \times S_k$.
In fact, $\bar \omega_1= h_2 \bar \tau_1 h_2^{-1}$, as these elements agree on $\im f$
and are the identity otherwise. Hence   $\bar \omega_1$ and $\bar \tau_1$ are conjugate in $S_m$ and so have the same cycle structure. The cycle structure of $\bar \omega_1$ is obtained from $ \omega_1$
 by the addition of $m-i$ trivial cycles. The same holds for $\bar \tau_1$ and $\tau_1$. It follows that
$\tau_1$ and $\omega_1$ are conjugates in $S_i$. Analogously, $\bar \omega_2= h_4 \bar \tau_2 h_4^{-1}$, and
 $\tau_2$ and $\omega_2$ are conjugates in $S_k$. Therefore, $(\omega_1,\omega_2)$ is a conjugate of $(\tau_1,\tau_2)$ in $S_i \times S_k$,
and hence $(\omega_1,\omega_2) \in K$.

We obtain that $(f\cdot \omega_1, f \cdot \omega_2) \in K'$, and hence $(f,g)\th' (f\cdot \omega_1, f \cdot \omega_2)$, and so $(f,g)\th (f\cdot \omega_1, f \cdot \omega_2)$.
Now $f \cdot \omega_1=f\bar \omega_1$, $g \cdot \omega_2=g\bar \omega_2$ and so
\begin{eqnarray*}
(a,b)&=&(h_1 f h_2, h_3 g h_4)\\
&\th& (h_1 f\bar \omega_1 h_2, h_3 g\bar \omega_2 h_4)\\
&=&(h_1 f (h_2\bar \tau_1 h_2^{-1})h_2, h_3 g (h_4 \bar \tau_2 h_4^{-1}) h_4)\\
&=&(h_1 f h_2 \bar \tau_1, h_3 g h_4\bar \tau_2) \\
&=&(a\bar \tau_1, b\bar \tau_2)\\
&=&(a \cdot \tau_1, b \cdot \tau_2)\\
&=&(c,d)
\end{eqnarray*}
as required.

It follows that all pairs that satisfy  one of the conditions (1) to (4) are in $\th$.

Conversely, let $\rho$ be defined on $Q_m \times Q_n$ by, for all $(a,b),(c,d) \in Q_m \times Q_n$,
$$ (a,b) \rho (c,d) \mbox{ iff }\left\{
                                  \begin{array}{ll}
                                    (a,b)=(c,d), & \hbox{or} \\
                                    (a,b) \ne (c,d), & \hbox{and one of (1) to (4) holds.}
                                  \end{array}
                                \right.$$
We can routinely verify that $\rho$ is a congruence. As $((f,g),(f',g')) \in \rho$, $\th \subseteq \rho$. We have shown that $\rho \subseteq \th$, therefore $\th=\rho$, as required.

\end{proof}
We can get a more direct description of the congruence classes by using the following
group theoretic result. Its proof is in the appendix.
\begin{thm}
All normal subgroups of $S_i \times S_k$ are either products of normal subgroups of $S_i$ and $S_k$ or the group of all pairs $(\sigma_1, \sigma_2)$ where $\sigma_1$ and $\sigma_2$ are permutations
with the same signature.
\end{thm}

\begin{cor}
Under the conditions of  Theorem \ref{th:2h}, if one of $\sigma_1$ or $\sigma_2 $ is an even permutation,
then $\theta$ agrees on $(I_i \times I_j)^2$ with the product congruence $\theta_1 \times \theta_2$, and
is trivial elsewhere. Concretely, in this case, let $N_1, N_2$ be the normal subgroups of $S_i$ and $S_k$ generated by $\sigma_1$, $\sigma_2 $, respectively. Then for $(a,b) \ne (c,d)$, 

$(a,b) \th (c,d)$
 if and only if one of the following holds:
 \begin{enumerate}
 \item $|a|, |c| \le i-1$, $|b|, |d| \le k-1 $;
\item $|a|, |c| \le i-1$, $|b| =|d|=k$, $b \mathcal{H} d$, $d=b \cdot \tau_2$ for some $\tau_2 \in N_2$;
 \item $|a|=|c|= i$, $|b| ,|d| \le k-1$, $a \mathcal{H} c$, $a=c \cdot \tau_1$ for some $\tau_1 \in N_1$;
 \item $|a|=|c|=i, |b|=|d|=k$, $a \mathcal{H} c$, $b  \mathcal{H} d$,  $a=c \cdot \tau_1$ for some $\tau_1 \in N_1$, and $d=b \cdot \tau_2$ for some $\tau_2 \in N_2$.
\end{enumerate}

Under the conditions of Theorem \ref{th:2h},  if both $\sigma_1, \sigma_2$ are odd permutations, then for $(a,b) \ne (c,d)$, 

$(a,b) \th (c,d)$
 if and only if one of the following holds:
 \begin{enumerate}
 \item $|a|, |c| \le i-1$, $|b|, |d| \le k-1 $;
\item $|a|, |c| \le i-1$, $|b| =|d|=k$, $b \mathcal{H} d$;
 \item $|a|=|c|= i$, $|b| ,|d| \le k-1$, $a \mathcal{H} c$;
 \item $|a|=|c|=i, |b|=|d|=k$, $a \mathcal{H} c$, $b  \mathcal{H} d$ and there exist
$\tau_1 \in S_i, \tau_2 \in S_k $ of the same signature such
 that $a \cdot \tau_1 =c$, $b \cdot \tau_2= d$.
\end{enumerate}

\end{cor}

\section{The structure of all congruences on $Q_m \times Q_n$}\label{s:allcong}
We now look at our main aim: to determine the structure of all congruences on $Q_m \times Q_n$.
When studying a congruence $\th$ on $Q_m \times Q_n$, as in the case of $Q_n$ (Theorem \ref{th:congrQn}), we realize that the $\th$-classes are intrinsically related to the $\Dg$-classes of $Q_m \times Q_n$.
We shall show that $\th$ is determined by some minimal blocks of $\th$-classes, called here $\th$-dlocks, which are also unions of $\Dg$-classes. The strategy will be to determine the possible types of $\th$-dlocks and
to describe $\th$ within such blocks.

Throughout this section, $\th$ denotes a congruence on $Q_m \times Q_n$.
To avoid some minor technicalities, we assume that   $Q_m$ and $ Q_n$ are non-trivial, i.e.
we exclude the factors $\T_0,\mathcal{T}_1$, $\mathcal{PT}_0$, $\mathcal{I}_0$.

\begin{df}
A $\th$-\emph{dlock} $X$ is a non-empty subset of $Q_m \times Q_n$ such that
\begin{enumerate}
  \item $X$ is a union of $\th$-classes as well as  a union of $\Dg$-classes;\label{n:dlock}
  \item No proper non-empty subset of $X$ satisfies (\ref{n:dlock}).
\end{enumerate}
\end{df}

In other words, the $\th$-dlocks are the classes of the equivalence relation generated by $\Dg \cup \th$. We will just write \emph{dlock} if $\th$ is understood by context, and will describe dlocks by
listing  their $\Dg$-classes. Concretely, let $D_{i,j}$ be the $\Dg$-class of all pairs $(f,g)$ such that $|f|=i, |g|=j$. For
$P \subseteq \{0, \dots, m\} \times \{0, \dots, n\}$, we
set $D_P= \cup_{(i,j) \in P}D_{i,j}$.

 First we describe the various configurations of dlocks with respect to the $\Hg$-classes they contain. To this end,
we will divide the dlocks into $9$ different types.

\begin{df}Let $X$ be a $\th$-dlock. We say that $X$ has first component type
\begin{description}
  \item[$\varepsilon$] if for all  $(a,b),(c,d) \in X$ with $(a,b)\, \th \, (c,d)$ we have $a=c$;
  \item[$\Hg$] if there exist  $(a,b),(c,d) \in X$ such that  $a \ne c$ and $(a,c)\,\th\,(b,d)$, and for all $ (a',b'), (c',d') \in X$ such that $(a',b') \,\th \,(c',d')$ we have $a' \Hg c'$;
  \item[$F$]if there exist  $(a,b),(c,d) \in X$ such that  $(a, c) \notin \Hg$ and $(a,c)\,\th\,$ $(b,d)$.
\end{description}
We define the second component type of $X$ dually. Finally we say that $X$ has type $VW$ for $V, W \in \{\varepsilon, \Hg, F\}$ if it has first component type $V$ and second component type $W$.
\end{df}
Clearly, we obtain all possible types of dlocks.
Next, we will describe the congruence $\th$ by means of its  restriction to each of its dlocks.

\begin{lem} \label{lm:t1} Let $X=D_P=\cup_{(i,j) \in P}D_{i,j}$ be a $\th$-dlock of type $FF$.  Then
$P$ is a downward-closed subset of
$(\{0, \dots, m\} \times \{0, \dots, n\}, \le \times \le)$, and $X$ is a single
$\th$-class that is an ideal of $Q_m \times Q_n$. Conversely, let $\th$ have an ideal class $I$,
 and let $P\subseteq \{0, \dots, m\} \times \{0, \dots, n\}$ be the set of pairs $(i,j)$ for which $D_{i,j} \cap I \ne \emptyset$. Then $I$ is a dlock of Type $FF$
unless one of the following holds:
\begin{enumerate}
\item$P=\{0\} \times \{0,\dots,j\}$ for some $j\in \{0, \dots, n\}$;
\item $P=\{0,\dots, i\}\times \{0\}$ for some $i \in \{0, \dots, m\}$.
\end{enumerate}
\end{lem}
\begin{proof}
Let $X$ be a $\th$-dlock of type $FF$ and assume that its first component type $F$ is witnessed by $((f,g),(f',g'))$, i.e.
$((f,g),(f',g')) \in \th \cap X^2$ and
 $(f,f') \not\in \Hg$.

 Let $\th'$ be the congruence generated by
$((f,g),(f',g'))$. Let $i=\max\{|f|,|f'|\}$ and $k = \max\{|g|,|g'|\}$. Consider the ideals $I_i \subseteq  Q_m$, $I_k \subseteq Q_n$.
 By either Theorem \ref{th:0h} (if $g \Hg g', g \ne g$), Theorem \ref{th:1h} (if $(g,g') \notin \Hg$) or Lemma \ref{lm:equal} (if $g=g'$),
all sets of the form $I_i \times \{b\}$, where $|b|\le k$, are contained in congruence classes of $\th'$. By a dual argument $\{a\} \times I_k$
is contained in a $\th'$ class for $|a| \le i$. By choosing $|a|=0*=|b|$, we see that these sets intersect. It follows that $\th'$ and hence $\th$ have a class that contains
$D_{0*,0*}$. It is straightforward to check that such a $\th$-congruence class $Y$ is an ideal of $Q_m \times Q_n$.  By Lemma \ref{lem:ideals1}, all ideals are unions of $\Dg$-classes.
Therefore $Y$ is a dlock that contains $X$. As dlocks are disjoint $X=Y$, and so
$X$ is a single $\th$-class that is an ideal.

Conversely, suppose that $I$ is an ideal class of $\th$, and that $P$ indexes the $\mathcal{D}$-classes intersecting
$I$. By Lemma \ref{lem:ideals1}, the ideal $I$ is a union of  $\Dg$-classes, so that $D_P=I$, and $I$ is a dlock. It is a dlock of
type $FF$  unless either  $\pi_1(I)$ or $\pi_2(I)$ consists of a single $\Hg$-class. The listed exceptions are the only
way this can happen, as we assumed that $Q_n$, $Q_m$ are non-trivial.
\end{proof}
In particular, by Lemma \ref{lem:ideals}, the congruence $\th$ has at most one dlock of type $FF$. If it exists, the  unique dlock of type $FF$
is the $\th$-class that contains $D_{0*,0*}$. By Lemma \ref{lem:ideals1}, we can visualize this dlock as
 a ``landscape" (see Figure 1).
\begin{figure}[h]
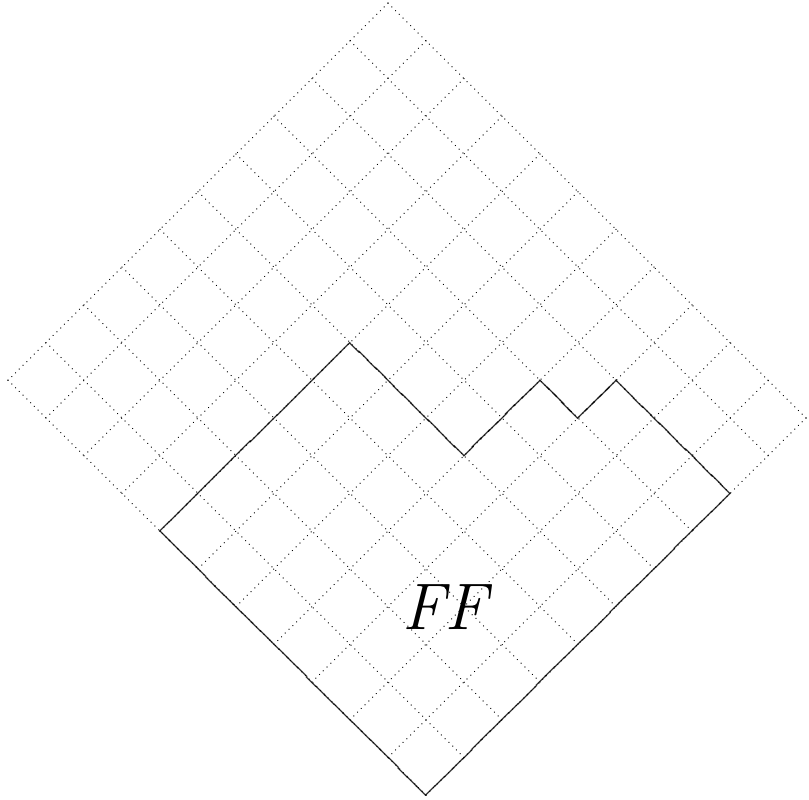

\[
\xy
(20,30)*{}="1";
(15,35)*{}="2";
(10,40)*{}="3";
(5,45)*{}="4";
(0,50)*{}="5";
(10,60)*{}="6";
(15,55)*{}="7";
(15,65)*{}="8";
(20,60)*{}="9";
(25,65)*{}="10";
(30,60)*{}="11";
(35,55)*{}="12";
(40,50)*{}="13";
(35,45)*{}="14";
(30,50)*{}="15";
(25,55)*{}="16";
(50,50)*{}="17";
(45,55)*{}="18";
(55,55)*{}="19";
(60,40)*{}="20";
(55,45)*{}="21";
(60,50)*{}="22";
(55,-5)*{}="O";
(65,55)*{}="23";
(65,45)*{}="24";
(70,50)*{}="25";
(75,55)*{}="26";
(85,55)*{}="29";
(75,45)*{}="27";
(80,50)*{}="28";
(85,45)*{}="30";
(90,50)*{}="31";
(95,55)*{}="32";
(95,35)*{}="33";
(100,40)*{}="34";
(105,45)*{}="35";
"O";"1"**\dir{-};
"1";"18"**\dir{-};
"18";"20"**\dir{-};
"20";"25"**\dir{-};
"25";"27"**\dir{-};
"27";"28"**\dir{-};
"28";"33"**\dir{-};
"O";"33"**\dir{-};
(58,20)*{\mbox{\textit{{\Huge FF}}}};
(55,-5)*{}="aO";
(0,50)*{}="a1";
(5,45)*{}="a2";
(10,40)*{}="a3";
(15,35)*{}="a4";
(20,30)*{}="a5";
(25,25)*{}="a6";
(30,20)*{}="a7";
(35,15)*{}="a8";
(40,10)*{}="a9";
(45,5)*{}="a10";
(50,0)*{}="a11";
(60,0)*{}="b3";
(65,5)*{}="b4";
(70,10)*{}="b5";
(75,15)*{}="b6";
(80,20)*{}="b7";
(85,25)*{}="b8";
(90,30)*{}="b9";
(95,35)*{}="b10";
(100,40)*{}="b11";
(105,45)*{}="b12";
(0,50)*{}="c1";
(5,55)*{}="c2";
(10,60)*{}="c3";
(15,65)*{}="c4";
(20,70)*{}="c5";
(25,75)*{}="c6";
(30,80)*{}="c7";
(35,85)*{}="c8";
(40,90)*{}="c9";
(45,95)*{}="c10";
(50,100)*{}="c11";
(50,100)*{}="d1";
(55,95)*{}="d2";
(60,90)*{}="d3";
(65,85)*{}="d4";
(70,80)*{}="d5";
(75,75)*{}="d6";
(80,70)*{}="d7";
(85,65)*{}="d8";
(90,60)*{}="d9";
(95,55)*{}="d10";
(100,50)*{}="d11";
"a1";"d1"**\dir{.};
"a2";"d2"**\dir{.};
"a3";"d3"**\dir{.};
"a4";"d4"**\dir{.};
"a5";"d5"**\dir{.};
"a6";"d6"**\dir{.};
"a7";"d7"**\dir{.};
"a8";"d8"**\dir{.};
"a9";"d9"**\dir{.};
"a10";"d10"**\dir{.};
"a11";"d11"**\dir{.};
"aO";"c1"**\dir{.};
"b3";"c2"**\dir{.};
"b4";"c3"**\dir{.};
"b5";"c4"**\dir{.};
"b6";"c5"**\dir{.};
"b7";"c6"**\dir{.};
"b8";"c7"**\dir{.};
"b9";"c8"**\dir{.};
"b10";"c9"**\dir{.};
"b11";"c10"**\dir{.};
"aO";"b12"**\dir{.};
"c11";"b12"**\dir{.};
\endxy
\]
\caption{A possible configuration for a dlock of type $FF$}\label{FF}
\end{figure}

\begin{lem}\label{lm:t2l3l}
Let $X$ be a $\th$-dlock. Then $X$ is of type $\varepsilon F$ or $\Hg F$ if and only if
there exist $0* \le i\le m$, $1 \le j \le n$, and $N \unlhd S_i $,
such that $X=D_P$ with $P=\{i\} \times \{0*, \dots, j\}$, and
 for every $(f,g) \in X$,
\begin{equation} [(f,g)]_\th = \{(f',g')\in X| f \Hg f' \mbox{ and } f'=f\cdot \sigma \mbox{ for some } \sigma \in N\}.
\label{dlock2}\end{equation}


If $X$ satisfies these requirements, then $X$ is of type $\Hg F$ exactly when $N \ne \varepsilon_i$.
\end{lem}
\begin{proof}
Let $X=D_P$ be a $\th$-dlock. Suppose $X$ is of  of type $\varepsilon F$ or $\Hg F$. Then for every $\th$-class $C$ contained in $X$, we have $\pi_1(C)$ contained in an $\Hg$-class of $Q_m$ by the definitions of first component types $\varepsilon$ and $\Hg$. Hence $\pi_1(C)$ is contained in a $\Dg$-class of $Q_m$ for all such $C$.
Since the elements of $\pi_1(X)$ that are $\th$-related are pairwise either equal or $\Hg$-related, $\pi_1(X)$ must be a single $\Dg$-class of $Q_m$, say $D_i$, by the definition of $\th$-dlock.
Therefore
$P=\{i\} \times K$ for some non-empty $K\subseteq \{0*,\dots, n\}$. Let $j$ be the largest element in $K$.
We claim that there exists $((f,g), (f',g')) \in \th$ with $(g,g') \not\in \Hg$ and $(f,g) \in D_{i,j}$. For otherwise
$[(\bar f,\bar g)]_{\th}$ would be contained in an $\Hg$-class for every $(\bar f,\bar g) \in D_{i,j}$, and then
$\{(i,j)\}$ would index a dlock contained in $X$, and $X$ being  minimal would imply $D_{i,j}=X$.
However, $X$ is  a dlock of type $\epsilon F$ or $\Hg F$, and we have a contradiction.

So there exist $((f,g),(f',g')) \in \th$, such that $(g,g') \notin \Hg$ and $(f,g) \in D_{i,j}$. Now, there exists $\sigma \in S_i$ such that $f'=f \cdot \sigma$. Let $N_\sigma$ be the normal subgroup of $S_i$  generated by ${\sigma}$.
Let $N$ be maximal between all such $N_\sigma$.
Assume that  $((f,g),(f',g'))$ witnesses $N$.


Let $\th'$ be the principal congruence generated by $((f,g),(f',g'))$.
Either  Corollary \ref{c:equal} or the dual of Theorem \ref{th:1h} is applicable to  $\th'$  - the first one if $f=f'$, and hence $N=\varepsilon_i$, and the second one otherwise.

Assume first that $N \ne \varepsilon_i$, then by the dual of Theorem \ref{th:1h},
we get that for all $(\hat f, \hat g) \in D_{\{i\} \times \{0*,\dots,j\}}$,
\begin{equation} \label{dlock0}[(\hat f, \hat g)]_{\th'}=\{(\bar f, \bar g)\,|\, |\bar g|\le j,\bar f=\hat f \cdot \bar\sigma \mbox{ for some } \bar\sigma \in N\}.\end{equation}
If $N = \varepsilon_i$ then by  Corollary \ref{c:equal},
we get that for all $(\hat f, \hat g) \in D_{\{i\} \times \{0*,\dots,j\}}$,
\begin{equation}\label{dlock0a}
    [(\hat f, \hat g)]_{\th'}=\{(\hat f, \bar g)\,|\,|\bar g|\le j\} =\{(\bar f, \bar g) \,|\, |\bar g|\le j, \bar f = \hat f \cdot \bar\sigma \mbox{ for some }  \bar\sigma \in N\}.
\end{equation}
As the rightmost expressions in (\ref{dlock0}) and (\ref{dlock0a}) are identical, we may treat both cases simultaneously.

We claim that the sets from (\ref{dlock0}) or (\ref{dlock0a}) are also congruence classes of $\th$. As $\th' \subseteq \th$ the sets in (\ref{dlock0}), (\ref{dlock0a})  are contained in classes of $\th$. We
 have already established that $P\subseteq \{i\} \times \{0*,\dots,j\}$,
it now follows that
$P=\{i\} \times \{0*,\dots,j\}$, since the sets in (\ref{dlock0}), (\ref{dlock0a}) intersect all $\Dg$-classes indexed by $\{i\} \times \{0*,\dots,j\}$.

Now let $E$ be a congruence class of $\th$ that is contained in $X$. Then $E$ must be a union of sets from (\ref{dlock0}) or (\ref{dlock0a}), and in particular, must intersect
 $D_{i,j}$, say $(\hat f, \hat g) \in E \cap D_{i,j}$. Let $(\bar f,\bar g) \th (\hat f,\hat g)$. Then $|\bar g| \le j$, since $ E \subseteq X=D_P$. Moreover
$(\hat f,\bar f) \in \Hg$ as $X$ is a dlock of type $\varepsilon F$ or $\Hg F$. Now, if $\beta$ is given by $\bar f=\hat f\cdot \beta$, then $\beta \in N$, by the maximality of $N$.
It follows that $E$ is one of the sets from (\ref{dlock0}), (\ref{dlock0a}), and thus contained in $[(\hat f,\hat g)]_{\th'}$. Therefore $[(\hat f,\hat g)]_{\th}=[(\hat f,\hat g)]_{\th'}$, for $(\hat f, \hat g)\in X$.

Notice that $j$ cannot be $0$. In fact,
if $j=0$, then $\pi_2(X)$  only contains one element, which contradicts the definition of type $\epsilon F$ or $\Hg F$. We have concluded the proof of the ``if" direction
of the first statement of the lemma.

The ``only if" direction  now follows directly from the description (\ref{dlock2}), provided that there exists $(g,g') \not\in \Hg$,
$(g,g') \in \pi_2(X)$. This holds as $Q_n$ is non-trivial and $j\ge 1$.

Finally, the last statement  follows directly from (\ref{dlock2}) and the definitions of type $\varepsilon F$ or $\Hg F$.
\end{proof}

If $X$ is a dlock of type $\Hg F$ or $\varepsilon F$, we call the group $N \unlhd S_{i} $ from Lemma \ref{lm:t2l3l} the \emph{normal subgroup associated with} $X$. Clearly, a dual version of  Lemma \ref{lm:t2l3l} holds for
dlocks $X$ of type $F\varepsilon$ or $F \Hg$.

\begin{lem}\label{lm:t2l}
Let $X=D_P$ be a dlock of type  $\Hg F$ with $P=\{i\} \times \{0*, \dots, j)\}$ and associated normal subgroup $N \unlhd S_i $.  Then $i\ge 2$, and
$\{0*,\dots,i-1\}\times \{0*, \dots, j\}$ is contained in the index set of a $\th$-dlock of type $FF$.

Moreover, every  congruence $\th$ has at most one dlock of type $\Hg F$.
\end{lem}
\begin{proof}
By Lemma \ref{lm:t2l3l}, we have $\varepsilon_i \ne N \unlhd S_i$, which implies that $i \ge 2$. By
the description (\ref{dlock2}), we may find $((f,g),(f',g')) \in \th \cap X^2$, such that
$(f,g) \in D_{i,j}$, $f \ne f'$,  $(f,f') \in \Hg$, and $(g,g') \not\in \Hg$. As $X=D_P$, $|g'| \le j$, $|f| =i$.

Let $\th'$ be the principal congruence generated by $((f,g),(f',g'))$. Then $\th'$ is described in the dual of Corollary \ref{c:1h}.
By this corollary, if $C:=\{0*,\dots,i-1\}\times \{0*, \dots, j\}$ then $D_C$ is one equivalence class of $\th'$, and hence contained in
an equivalence class of $\th$.  As $j \ge 1, i-1 \ge 1$, and we excluded the case that $Q_n$ is trivial, $\pi_1(D_C)$ and
$\pi_2(D_C)$ both contain more then one $\Hg$-class. Hence $C$ is contained in the index set of a
$\th$-dlock $\bar X$ of type $FF$.

Now assume that $\th$ has a potentially different dlock $X'=D_{P'}$ of type $\Hg F$, where $P'=\{i'\} \times \{0*, \dots, j')\}$.
By applying our previous results to $X'$, we get that
$D_{i'-1,0*}$ must also lie in a $\th$-dlock  $\bar X$ of type $FF$. As noted after Lemma \ref{lm:t1}, dlocks of type $FF$ are unique, and so $D_{i'-1,0*} \subseteq \bar X$.
So both $D_{i-1,0*}\subseteq \bar X$ and $D_{i'-1,0*}\subseteq \bar X$, but $D_{i,0*}\not\subseteq \bar X$ and $D_{i',0*}\not\subseteq \bar X$, as these $\Dg$-classes lie in the $\Hg F$-dlocks $X$ and $X'$. By Lemma \ref{lm:t1},
the index set of $\bar X$ is downwards closed, hence there is a unique $\bar i$ such that $D_{\bar i, 0*} \subseteq \bar X$, $D_{\bar i+1,0*}\not\subseteq \bar X$. It follows that $i=\bar i +1 = i'$.
Thus both $X$ and $X'$ contain $D_{i,0*}$ and therefore $X=X'$.
\end{proof}

We may visualize the statement of Lemma \ref{lm:t2l} by saying that a $\th$-dlock of type $\Hg F$ must lies on the  ``most eastern
slope" of the dlock of type $FF$.
Once again, a dual result holds for dlocks of type $F \Hg$.
Figure 2 shows the possible positions for dlocks of type $\Hg F$ and $F \Hg$
 in relation to a dlock of type $FF$.

\begin{figure}[h]
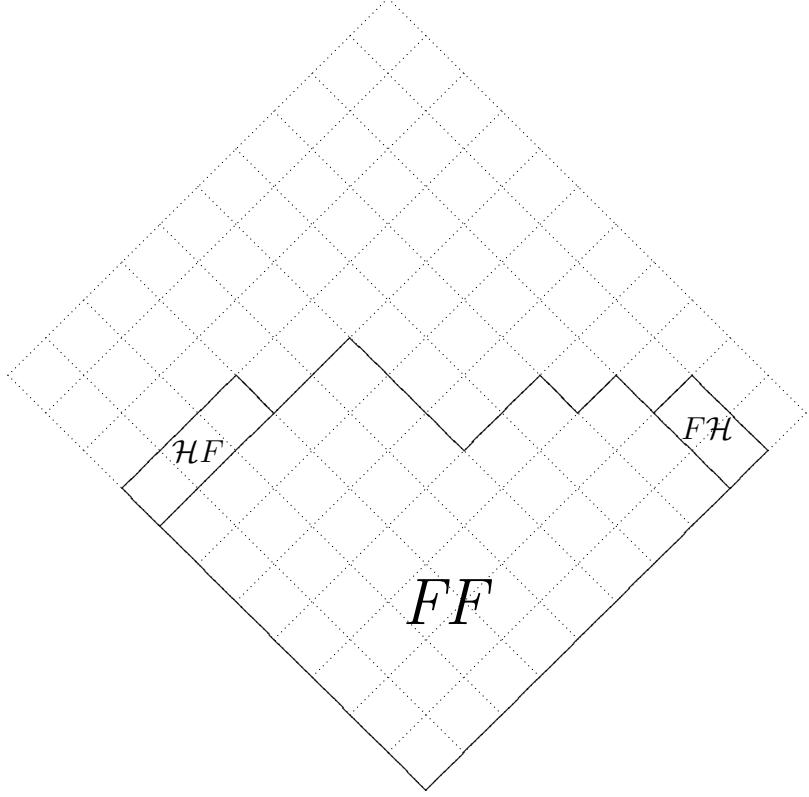

\[
\xy
(20,30)*{}="1";
(15,35)*{}="2";
(10,40)*{}="3";
(5,45)*{}="4";
(0,50)*{}="5";
(10,60)*{}="6";
(15,55)*{}="7";
(15,65)*{}="8";
(20,60)*{}="9";
(25,65)*{}="10";
(30,60)*{}="11";
(35,55)*{}="12";
(40,50)*{}="13";
(35,45)*{}="14";
(30,50)*{}="15";
(25,55)*{}="16";
(50,50)*{}="17";
(45,55)*{}="18";
(55,55)*{}="19";
(60,40)*{}="20";
(55,45)*{}="21";
(60,50)*{}="22";
(55,-5)*{}="O";
(65,55)*{}="23";
(65,45)*{}="24";
(70,50)*{}="25";
(75,55)*{}="26";
(85,55)*{}="29";
(75,45)*{}="27";
(80,50)*{}="28";
(85,45)*{}="30";
(90,50)*{}="31";
(95,55)*{}="32";
(95,35)*{}="33";
(100,40)*{}="34";
(105,45)*{}="35";
"O";"2"**\dir{-};
"1";"18"**\dir{-};
"18";"20"**\dir{-};
"20";"25"**\dir{-};
"25";"27"**\dir{-};
"27";"28"**\dir{-};
"28";"33"**\dir{-};
"O";"34"**\dir{-};
"2";"15"**\dir{-};
"15";"14"**\dir{-};
"30";"31"**\dir{-};
"31";"34"**\dir{-};
(58,20)*{\mbox{\textit{{\Huge FF}}}};
(25,40)*{{\mathcal H}  F};
(92,43)*{F{\mathcal H} };
(55,-5)*{}="aO";
(0,50)*{}="a1";
(5,45)*{}="a2";
(10,40)*{}="a3";
(15,35)*{}="a4";
(20,30)*{}="a5";
(25,25)*{}="a6";
(30,20)*{}="a7";
(35,15)*{}="a8";
(40,10)*{}="a9";
(45,5)*{}="a10";
(50,0)*{}="a11";
(60,0)*{}="b3";
(65,5)*{}="b4";
(70,10)*{}="b5";
(75,15)*{}="b6";
(80,20)*{}="b7";
(85,25)*{}="b8";
(90,30)*{}="b9";
(95,35)*{}="b10";
(100,40)*{}="b11";
(105,45)*{}="b12";
(0,50)*{}="c1";
(5,55)*{}="c2";
(10,60)*{}="c3";
(15,65)*{}="c4";
(20,70)*{}="c5";
(25,75)*{}="c6";
(30,80)*{}="c7";
(35,85)*{}="c8";
(40,90)*{}="c9";
(45,95)*{}="c10";
(50,100)*{}="c11";
(50,100)*{}="d1";
(55,95)*{}="d2";
(60,90)*{}="d3";
(65,85)*{}="d4";
(70,80)*{}="d5";
(75,75)*{}="d6";
(80,70)*{}="d7";
(85,65)*{}="d8";
(90,60)*{}="d9";
(95,55)*{}="d10";
(100,50)*{}="d11";
"a1";"d1"**\dir{.};
"a2";"d2"**\dir{.};
"a3";"d3"**\dir{.};
"a4";"d4"**\dir{.};
"a5";"d5"**\dir{.};
"a6";"d6"**\dir{.};
"a7";"d7"**\dir{.};
"a8";"d8"**\dir{.};
"a9";"d9"**\dir{.};
"a10";"d10"**\dir{.};
"a11";"d11"**\dir{.};
"aO";"c1"**\dir{.};
"b3";"c2"**\dir{.};
"b4";"c3"**\dir{.};
"b5";"c4"**\dir{.};
"b6";"c5"**\dir{.};
"b7";"c6"**\dir{.};
"b8";"c7"**\dir{.};
"b9";"c8"**\dir{.};
"b10";"c9"**\dir{.};
"b11";"c10"**\dir{.};
"aO";"b12"**\dir{.};
"c11";"b12"**\dir{.};
\endxy
\]
\caption{A possible configuration for three dlocks of type $FF$, $\Hg F$ and $F \Hg$}\label{grid1}
\end{figure}

\begin{lem}\label{lm:t3l}
Let $X=D_P$ be a $\th$-dlock of type $\varepsilon F$ with  $P=\{i\} \times \{0*, \dots, j)\}$. Then for each $k=0*,\dots, i-1$, the set
$\{k\}\times\{0*,\dots,j\}$ is contained in the index set of a $\th$-dlock of type $FF$, $\Hg F$, or $\varepsilon F$.

\end{lem}
\begin{proof}
If $i=0*$, the statement quantifies  over the empty set, so assume this is not the case. It suffices to show
the statement for the case that $k=i-1$, the remaining values follow by induction applying: the same result if the resulting  $\th$-dlock $\bar X$ is of type $\varepsilon F$; Lemma \ref{lm:t1}
if $\bar X$ is of type $FF$; or Lemma \ref{lm:t2l} if $\bar X$ is of type $\Hg F$.

By Lemma \ref{lm:t2l3l} we may find $((f,g),(f,g')) \in \th \cap X^2$, such that
$(f,g) \in D_{i,j}$, $(g,g') \not\in \Hg$. Note that this implies that $|g'| \le j$.

Let
$\th'$ be the principal congruence generated by $((f,g),(f,g'))$. Then $\th'$ is described in the dual of Corollary \ref{c:1h}.
By this corollary, if $|\bar f| ={i-1}$, then
$$[(\bar f, g)]_{\th'}= \{(\bar f, \hat g)|\, |\hat g| \le j\}.$$
It follows that
 $\{i-1\}\times\{0*,\dots,j\}$ is the index set of a $\th'$-dlock and hence contained in the index set of a $\th$-dlock $X'$. As
$((\bar f, g),(\bar f,g')) \in \th'\subseteq \th$, and $(g,g') \not\in \Hg$, it follows that
$X'$ is of type $FF$, $\Hg F$, or $\varepsilon F$, as these three options cover all cases where $\pi_2(D_{P'}) \not\subseteq \Hg$. The result follows.\end{proof}

\begin{lem}\label{lm:t3lb}
For a given congruence $\th$, let $J$ be the set of values $(i,j_i)$ such that $\{i\} \times \{0*, \dots,j_i\}$ is the index set of a $\th$-dlock of type $\varepsilon F$.
Then $\pi_1(J)$ is a
set of consecutive integers (possibly empty), and the values $j_i$ are non-increasing in $i$.
\end{lem}
\begin{proof}
Assume that $i_1 < i_2 <i_3$ with $i_1,i_3 \in \pi_1(J)$. We want to show that $i_2 \in \pi_1(J)$, as well. As $i_1, i_3 \in \pi_1(J)$, $X_1= D_{\{i_1\} \times \{0*, \dots,j_{i_1}\}}$ and
$X_3= D_{\{i_3\} \times \{0*, \dots,j_{i_3}\}}$ are $\th$-dlocks of type $\varepsilon F$, for some $j_{i_1}, j_{i_3}$. Applying Lemma \ref{lm:t3l} to $X_3$, we get that
$D_{\{i_2\} \times \{0*, \dots,j_{i_3}\}}$ is contained in $\th$-dlock $X'$ of type $FF$, $\Hg F$, or $\varepsilon F$.

If $X'$ is of type $\Hg F$ or $FF$, then by Lemma \ref{lm:t2l} or Lemma \ref{lm:t1}, respectively, $D_{\{i_1\} \times \{0*, \dots,j_{i_3}\}}$ would be contained in a dlock of type $FF$, which in the latter case
would be the dlock $X'$ itself. In particular, $D_{i_1, 0*}$ would be contained in a dlock of type $FF$. However, as $i_1 \in J$, $D_{i_1, 0*}$ is contained in the dlock $X_1$ of type $\varepsilon F$, a contradiction.
Hence $X'$ is of type  $\varepsilon F$, and $i_2 \in \pi_1(J)$. It follows that $\pi_1(J)$, if not empty, is a set of consecutive integers.

Now let $(i,j_i), (i-1,j_{i-1})\in J$. As above, we have that $D_{\{i-1\} \times \{0*, \dots,j_i\}}$ is contained in a $\th$-dlock $X'$ of type $FF$, $\Hg F$, or $\varepsilon F$. As $(i-1,j_{i-1}) \in J$, this must necessarily be in
the dlock $D_{\{i-1\} \times \{0*, \dots,j_{i-1}\}}$ of type $\varepsilon F$. Hence $j_i \le j_{i-1}$, and so the values $j_i$ are non-increasing in $i$.
\end{proof}


Lemma \ref{lm:t3l} and Lemma \ref{lm:t3lb} show that, if there are any $\th$-dlocks of type $\varepsilon F$, they are layered on  the top of each
other without ``overhanging", with the lowest one either starting at $i=0*$, or lying on the top
of the eastern most slope of the dlock of type $FF$, or lying on the top of the dlock of type $\Hg F$,
in the last two cases without overhanging the dlocks below them.
An example is depicted in Figure 3.
Once again, the dual versions of Lemmas \ref{lm:t3l} and \ref{lm:t3lb} hold as well.
\begin{figure}[h]
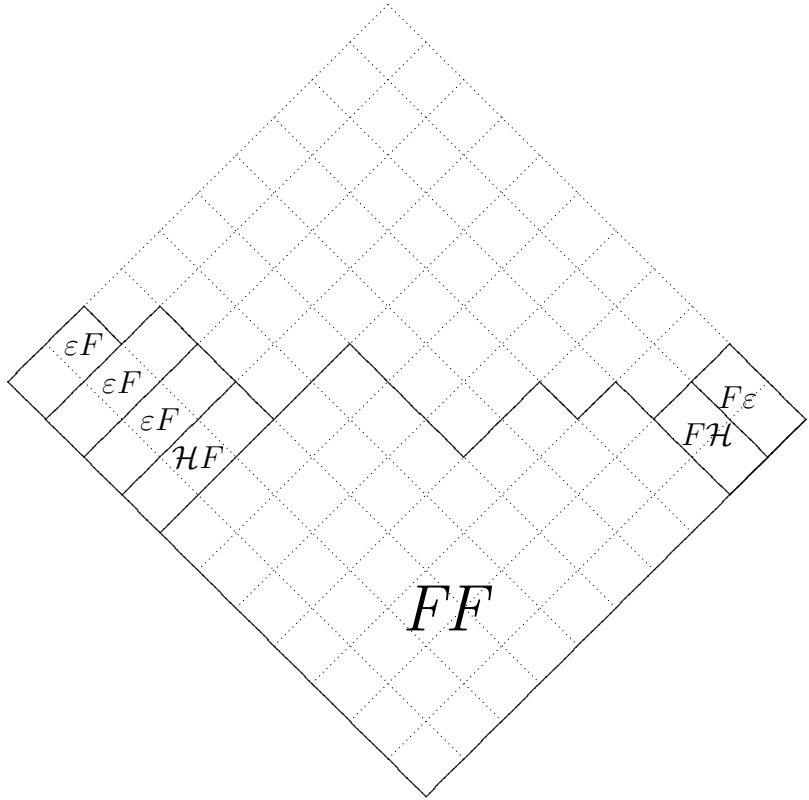

\[
\xy
(20,30)*{}="1";
(15,35)*{}="2";
(10,40)*{}="3";
(5,45)*{}="4";
(0,50)*{}="5";
(10,60)*{}="6";
(15,55)*{}="7";
(15,65)*{}="8";
(20,60)*{}="9";
(25,65)*{}="10";
(30,60)*{}="11";
(35,55)*{}="12";
(40,50)*{}="13";
(35,45)*{}="14";
(30,50)*{}="15";
(25,55)*{}="16";
(50,50)*{}="17";
(45,55)*{}="18";
(55,55)*{}="19";
(60,40)*{}="20";
(55,45)*{}="21";
(60,50)*{}="22";
(55,-5)*{}="O";
(65,55)*{}="23";
(65,45)*{}="24";
(70,50)*{}="25";
(75,55)*{}="26";
(85,55)*{}="29";
(75,45)*{}="27";
(80,50)*{}="28";
(85,45)*{}="30";
(90,50)*{}="31";
(95,55)*{}="32";
(95,35)*{}="33";
(100,40)*{}="34";
(105,45)*{}="35";
"O";"5"**\dir{-};
"3";"16"**\dir{-};
"16";"14"**\dir{-};
"2";"15"**\dir{-};
"18";"1"**\dir{-};
"18";"20"**\dir{-};
"20";"25"**\dir{-};
"25";"27"**\dir{-};
"28";"27"**\dir{-};
"28";"33"**\dir{-};
"33";"35"**\dir{-};
"32";"35"**\dir{-};
"31";"34"**\dir{-};
"32";"30"**\dir{-};
"O";"35"**\dir{-};
"4";"9"**\dir{-};
"9";"16"**\dir{-};
"5";"6"**\dir{-};
"7";"6"**\dir{-};
(58,20)*{\mbox{\textit{{\Huge FF}}}};
(10,55)*{\varepsilon  F};
(15,50)*{\varepsilon  F};
(20,45)*{\varepsilon  F};
(25,40)*{{\mathcal H}  F};
(92,43)*{F{\mathcal H} };
(96,48)*{F \varepsilon};
(0,50)*{};
(20,70)*{};
(25,65)*{};
(30,60)*{};
(55,-5)*{}="aO";
(0,50)*{}="a1";
(5,45)*{}="a2";
(10,40)*{}="a3";
(15,35)*{}="a4";
(20,30)*{}="a5";
(25,25)*{}="a6";
(30,20)*{}="a7";
(35,15)*{}="a8";
(40,10)*{}="a9";
(45,5)*{}="a10";
(50,0)*{}="a11";
(60,0)*{}="b3";
(65,5)*{}="b4";
(70,10)*{}="b5";
(75,15)*{}="b6";
(80,20)*{}="b7";
(85,25)*{}="b8";
(90,30)*{}="b9";
(95,35)*{}="b10";
(100,40)*{}="b11";
(105,45)*{}="b12";
(0,50)*{}="c1";
(5,55)*{}="c2";
(10,60)*{}="c3";
(15,65)*{}="c4";
(20,70)*{}="c5";
(25,75)*{}="c6";
(30,80)*{}="c7";
(35,85)*{}="c8";
(40,90)*{}="c9";
(45,95)*{}="c10";
(50,100)*{}="c11";
(50,100)*{}="d1";
(55,95)*{}="d2";
(60,90)*{}="d3";
(65,85)*{}="d4";
(70,80)*{}="d5";
(75,75)*{}="d6";
(80,70)*{}="d7";
(85,65)*{}="d8";
(90,60)*{}="d9";
(95,55)*{}="d10";
(100,50)*{}="d11";
"a1";"d1"**\dir{.};
"a2";"d2"**\dir{.};
"a3";"d3"**\dir{.};
"a4";"d4"**\dir{.};
"a5";"d5"**\dir{.};
"a6";"d6"**\dir{.};
"a7";"d7"**\dir{.};
"a8";"d8"**\dir{.};
"a9";"d9"**\dir{.};
"a10";"d10"**\dir{.};
"a11";"d11"**\dir{.};
"aO";"c1"**\dir{.};
"b3";"c2"**\dir{.};
"b4";"c3"**\dir{.};
"b5";"c4"**\dir{.};
"b6";"c5"**\dir{.};
"b7";"c6"**\dir{.};
"b8";"c7"**\dir{.};
"b9";"c8"**\dir{.};
"b10";"c9"**\dir{.};
"b11";"c10"**\dir{.};
"aO";"b12"**\dir{.};
"c11";"b12"**\dir{.};
\endxy
\]
\caption{A possible configuration for dlocks of type $FF$, $\Hg F$,  $F \Hg$, $\varepsilon F$, and $F \varepsilon$}\label{Fe}
\end{figure}

\begin{lem}\label{lm:t456}
Let $X$ be a dlock. Then $X$ is of type $\Hg \Hg$, $\varepsilon \Hg$, $\Hg \varepsilon$, or $\varepsilon \varepsilon$ if and only if $X=D_{i,j}$ for some $i,j$ and there exists  $N \unlhd S_i \times S_j$
such that
 for every $(f,g) \in X$,
 \begin{eqnarray}
 \nonumber  [(f,g)]_\th &=& \{(f',g')\in D_{i,j}\: | \: f \Hg f', g \Hg g',  \\
    &&\,\mbox{ and } (f',g')=(f\cdot \sigma,g \cdot \tau) \mbox{ for some } (\sigma,\tau) \in N\}.\label{dlock456}
 \end{eqnarray}
Moreover, in this situation,
\begin{enumerate}
\item[(a)] $X$ is of type $\Hg \Hg$ $\Longleftrightarrow \pi_1(N) \ne \varepsilon_i \mbox{ and } \pi_2(N) \ne \varepsilon_j$;
\item[(b)]  $X$ is of type $\varepsilon \Hg$ $\Longleftrightarrow N= \varepsilon_i \times N'$ for some $N' \ne \varepsilon_j$;
\item[(c)]  $X$ is of type $\Hg \varepsilon$ $\Longleftrightarrow N= N' \times\varepsilon_j$ for some $N' \ne \varepsilon_i$;
\item[(d)]  $X$ is of type $\varepsilon \varepsilon$ $\Longleftrightarrow N= \varepsilon_i \times \varepsilon_j$.
\end{enumerate}
\end{lem}
\begin{proof}
Let $X$ be of type $\Hg \Hg$, $\varepsilon \Hg$, $\Hg \varepsilon$, or $\varepsilon \varepsilon$. Then $\pi_1(\th \cap X^2) \subseteq \Hg$ and $\pi_2(\th \cap X^2)\subseteq \Hg$. It follows that
each $\Hg$-class in $X$ is a union of $\th$-classes. Therefore every $\Dg$-class in $X$ is a union of
$\th$-classes as well, and by the minimality property of a dlock, there is only one $\Dg$-class in $X$. It follows
that $X=D_{i,j}$ for some $(i,j)$.

Now as $\th \cap X^2 \subseteq \Hg \times \Hg$, having $((f,g),(f',g')) \in \th \cap X^2$ implies that $f'=f \cdot \sigma, g'=g \cdot \tau$ for some $\sigma \in S_i$ and $ \tau \in S_j$.
Let $N \subseteq S_i \times S_j$ be the set of all $(\sigma, \tau)$ that correspond to some   $((f,g),(f',g')) \in \th \cap X^2$.

If $N=\{(\id_{S_i}, \id_{S_j})\}$ then $\th$ is the identity on $X$, and $\th\cap X^2$ is given
by (\ref{dlock456}) with the choice $ N= \varepsilon_i \times \varepsilon_j$.

Otherwise, let $(\id_{S_i}, \id_{S_j}) \ne (\sigma, \tau) \in N$, as witnessed by $((f,g),(f \cdot \sigma, g \cdot \tau)) \in \th \cap X^2$. Let $\th'$ be the principal congruence generated by $((f,g),(f \cdot \sigma, g \cdot \tau))$,
and let $(\bar f, \bar g) \in X$. By either Theorem \ref{th:2h} (if $\sigma \ne \id_{S_i}$ and $\tau \ne \id_{S_j}$) or Lemma \ref{lm:equal} and its dual (otherwise), $((\bar f , \bar g),(\bar f \cdot \sigma, \bar g \cdot g)) \in \th' \subseteq \th$.
By repeated use of this argument, we get that the equivalence relation given by (\ref{dlock456}) is contained in $\th$. As the reverse inclusion follows from the definition of $N$, the expression (\ref{dlock456}) describes
$\th \cap X^2$.

It remains to show that $N$ is a normal subgroup of $S_i \times S_j$. Let $(\sigma, \tau), (\sigma', \tau') \in N$, and $(f,g)\in X$ arbitrary, then
$$(f \cdot \sigma, g \cdot \tau) \in \th \Rightarrow (((f \cdot \sigma) \cdot \sigma'), ((g \cdot \tau) \cdot \tau')))\in \th \Rightarrow ((f \cdot (\sigma \sigma')), (g \cdot (\tau \tau'))\in \th,$$
where the first implication follows as (\ref{dlock456}) describes
$\th$ on $X$, and the second implication as $\cdot$ is a group action. Thus $N$ is a subgroup of $S_i \times S_j$. Now if $(\sigma, \tau) \in N$, then by either Theorem \ref{th:2h}, Lemma \ref{lm:equal}, or the dual of Lemma \ref{lm:equal}
(applied to any $((f,g),(f \cdot \sigma, g \cdot \tau)) \in \th \cap X^2$), $N$ contains the normal subgroup generated by $(\sigma, \tau)$. Therefore $N$ is a  subgroup  generated by a union of normal subgroups,
and hence it is  itself normal.

The converse statement is immediate, and the characterization of the various types follows directly from the definition of the types and from (\ref{dlock456}).
\end{proof}
If $X=D_{i,j}$ is a dlock of type $\Hg \Hg$, $\varepsilon \Hg$, $\Hg \varepsilon$, or $\varepsilon \varepsilon$, we will call $N \unlhd S_i \times S_j$ from Lemma \ref{lm:t456} the \emph{normal subgroup associated with} $X$.
\begin{lem}\label{lm:t4}
Let $X=D_{i,j}$ be a $\th$-dlock of type $\Hg\Hg$ with normal subgroup $N \unlhd S_i \times S_j$.
Then $i,j \ge 2$, and $D_{i,j-1}$ is contained in either a $\th$-dlock of type $FF$, or in a $\th$-dlock of type $\Hg F$ with normal subgroup $N'\unlhd S_i$, where
$\pi_1(N) \subseteq N'$.

Symmetrically, $D_{i-1,j}$ is contained in either a $\th$-dlock of type $FF$, or in a $\th$-dlock of type $F \Hg$ with normal subgroup $N'\unlhd S_j$, where
$\pi_2(N) \subseteq N'$.
\end{lem}
\begin{proof}
By Lemma \ref{lm:t456}, we have $\pi_1(N) \ne \varepsilon_i, \pi_2(N) \ne \varepsilon_j$, and so $i,j \ge 2$.
 Also by Lemma \ref{lm:t456}, there are $((f,g),(f',g')) \in \th \cap X^2$ with $f \ne f', g\ne g'$. Let us fix such a pair, and
 let $\th'$ be the principal congruence generated by it. Then $\th' \subseteq \th$ and
$\th'$ is described in Theorem \ref{th:2h}.

Let $\bar g, \bar g' \in Q_m$  be transformations of rank $j-1$ that are in different $\Hg$-classes. Such elements clearly exist. By Theorem \ref{th:2h}(2), we get
$((f, \bar g),(f', \bar g')) \in \th'\subseteq \th$. As $f \ne f'$, $(\bar g,\bar g') \not\in\Hg$, the $\th$-dlock $X'$ containing $D_{i,j-1}$ is either of type $FF$ or $\Hg F$, depending on the existence or not of a pair
$((\hat f, \hat g),(\hat f', \hat g')) \in \th \cap D_{i,j}$ with $(\hat f, \hat f') \not\in \Hg$.

In the first case, we are done, so assume that $X'$ is of type $\Hg F$. Then the restriction of $\th$ to
 ${X'}^2$ is given in Lemma \ref{lm:t2l3l}. Let $N'\unlhd S_i$ be the normal subgroup  of $X'$. We now wish to proof that $\pi_1(N) \subseteq N'$.

 Let $\sigma \in \pi_1(N)$,
 so $(\sigma,\tau) \in N$ for some $\tau$. Then $((f,g),(f \cdot \sigma, g\cdot \tau))\in \th$ by Lemma \ref{lm:t456}, and so $((f,g \bar g), (f\cdot \sigma, (g \cdot\tau) \bar g))\in \th$. As
$(f,g \bar g) $ and  $(f\cdot \sigma, (g \cdot\tau) \bar g)$ lie in the $\Hg F$-dlock $X'$, we get that $\pi_1(\sigma,\tau)=\sigma \in N'$ by Lemma  \ref{lm:t2l3l}.


The last statement follows dually.
\end{proof}
The result means that the dlocks of type $\Hg\Hg$ can only occupy the ``valleys'' in the landscape formed by the dlocks of $FF$, $\Hg F$, and $F \Hg$ (see Figure 4).
\begin{figure}[h]
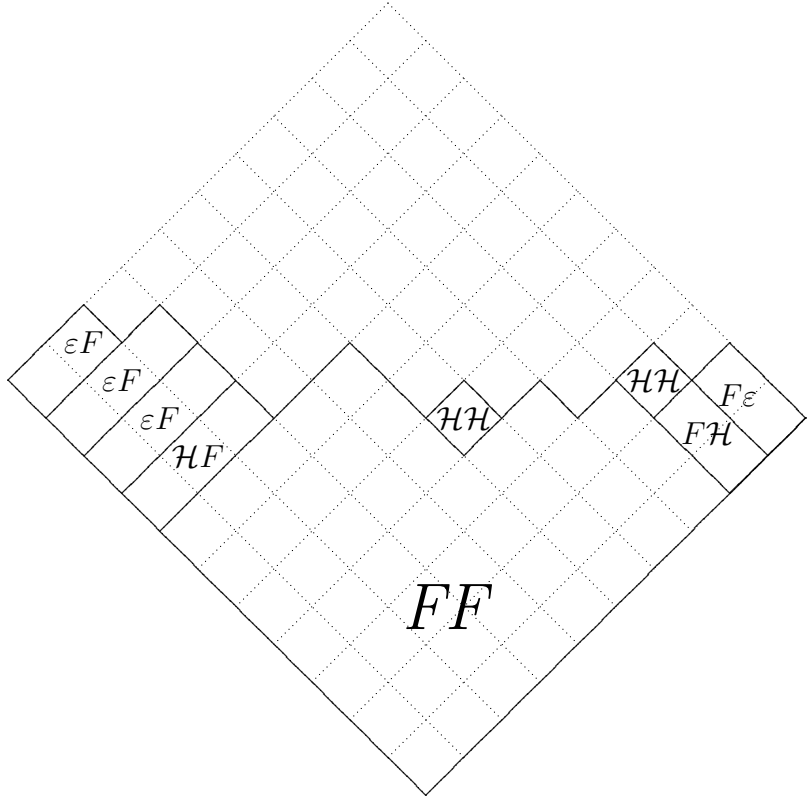

\[
\xy
(20,30)*{}="1";
(15,35)*{}="2";
(10,40)*{}="3";
(5,45)*{}="4";
(0,50)*{}="5";
(10,60)*{}="6";
(15,55)*{}="7";
(15,65)*{}="8";
(20,60)*{}="9";
(25,65)*{}="10";
(30,60)*{}="11";
(35,55)*{}="12";
(40,50)*{}="13";
(35,45)*{}="14";
(30,50)*{}="15";
(25,55)*{}="16";
(50,50)*{}="17";
(45,55)*{}="18";
(55,55)*{}="19";
(60,40)*{}="20";
(55,45)*{}="21";
(60,50)*{}="22";
(55,-5)*{}="O";
(65,55)*{}="23";
(65,45)*{}="24";
(70,50)*{}="25";
(75,55)*{}="26";
(85,55)*{}="29";
(75,45)*{}="27";
(80,50)*{}="28";
(85,45)*{}="30";
(90,50)*{}="31";
(95,55)*{}="32";
(95,35)*{}="33";
(100,40)*{}="34";
(105,45)*{}="35";
"O";"5"**\dir{-};
"3";"16"**\dir{-};
"16";"14"**\dir{-};
"2";"15"**\dir{-};
"18";"1"**\dir{-};
"18";"20"**\dir{-};
"20";"25"**\dir{-};
"25";"27"**\dir{-};
"28";"27"**\dir{-};
"28";"33"**\dir{-};
"33";"35"**\dir{-};
"32";"35"**\dir{-};
"31";"34"**\dir{-};
"32";"30"**\dir{-};
"O";"35"**\dir{-};
"22";"21"**\dir{-};
"22";"24"**\dir{-};
"28";"29"**\dir{-};
"31";"29"**\dir{-};
"9";"16"**\dir{-};
"5";"6"**\dir{-};
"7";"6"**\dir{-};
"7";"9"**\dir{-};
"7";"4"**\dir{-};
(58,20)*{\mbox{\textit{{\Huge FF}}}};
(10,55)*{\varepsilon  F};
(15,50)*{\varepsilon  F};
(20,45)*{\varepsilon  F};
(25,40)*{{\mathcal H}  F};
(60,45)*{{\mathcal H}{\mathcal H}};
(85,50)*{{\mathcal H}{\mathcal H}};
(92,43)*{F{\mathcal H} };
(96,48)*{F \varepsilon};
(0,50)*{};
(20,70)*{};
(25,65)*{};
(30,60)*{};
(55,-5)*{}="aO";
(0,50)*{}="a1";
(5,45)*{}="a2";
(10,40)*{}="a3";
(15,35)*{}="a4";
(20,30)*{}="a5";
(25,25)*{}="a6";
(30,20)*{}="a7";
(35,15)*{}="a8";
(40,10)*{}="a9";
(45,5)*{}="a10";
(50,0)*{}="a11";
(60,0)*{}="b3";
(65,5)*{}="b4";
(70,10)*{}="b5";
(75,15)*{}="b6";
(80,20)*{}="b7";
(85,25)*{}="b8";
(90,30)*{}="b9";
(95,35)*{}="b10";
(100,40)*{}="b11";
(105,45)*{}="b12";
(0,50)*{}="c1";
(5,55)*{}="c2";
(10,60)*{}="c3";
(15,65)*{}="c4";
(20,70)*{}="c5";
(25,75)*{}="c6";
(30,80)*{}="c7";
(35,85)*{}="c8";
(40,90)*{}="c9";
(45,95)*{}="c10";
(50,100)*{}="c11";
(50,100)*{}="d1";
(55,95)*{}="d2";
(60,90)*{}="d3";
(65,85)*{}="d4";
(70,80)*{}="d5";
(75,75)*{}="d6";
(80,70)*{}="d7";
(85,65)*{}="d8";
(90,60)*{}="d9";
(95,55)*{}="d10";
(100,50)*{}="d11";
"a1";"d1"**\dir{.};
"a2";"d2"**\dir{.};
"a3";"d3"**\dir{.};
"a4";"d4"**\dir{.};
"a5";"d5"**\dir{.};
"a6";"d6"**\dir{.};
"a7";"d7"**\dir{.};
"a8";"d8"**\dir{.};
"a9";"d9"**\dir{.};
"a10";"d10"**\dir{.};
"a11";"d11"**\dir{.};
"aO";"c1"**\dir{.};
"b3";"c2"**\dir{.};
"b4";"c3"**\dir{.};
"b5";"c4"**\dir{.};
"b6";"c5"**\dir{.};
"b7";"c6"**\dir{.};
"b8";"c7"**\dir{.};
"b9";"c8"**\dir{.};
"b10";"c9"**\dir{.};
"b11";"c10"**\dir{.};
"aO";"b12"**\dir{.};
"c11";"b12"**\dir{.};
\endxy
\]
\caption{A possible configuration for dlocks of type $FF$, $\Hg F$,  $F \Hg$, $\varepsilon F$, $F \varepsilon$, and $\Hg \Hg$}\label{HH}
\end{figure}

\begin{lem}\label{th:5l}
Let $X=D_{i,j}$ be a $\th$-dlock of type $\varepsilon \Hg$ with normal subgroup $N=\varepsilon_i \times   N' \unlhd S_i \times S_j$. Then $j \ge 2$, and $D_{i,j-1}$ is contained in a dlock of type $FF$, $\Hg F$, or $\varepsilon F$.

Moreover, if $i >0*$, then $D_{i-1,j}$ is contained in one of the following:
\begin{enumerate}
\item a dlock of type $FF$;
\item a dlock of type $\Hg F$;
\item a dlock of type $\varepsilon F$;
\item a dlock of type $F \Hg$  with normal subgroup $\bar N \unlhd S_j$ such that $N' \subseteq \bar N$;
\item a dlock of type $\Hg \Hg$ with normal subgroup $\bar N \unlhd S_{i-1} \times S_j$ such that $\varepsilon_{i-1} \times N' \subseteq \bar N$;
\item a dlock of type $\varepsilon \Hg$ with normal subgroup $\bar N \unlhd S_{i-1} \times S_j$ such that $\varepsilon_{i-1} \times N' \subseteq \bar N$.
\end{enumerate}
\end{lem}
\begin{proof} As $N'$ is non-trivial, $j \ge 2$. Let $\sigma$ generate $N'$ as a normal subgroup in $S_j$. Let $(f,g) \in X$  and set $g'=g \cdot \sigma$. Then $((f,g),(f,g')) \in \th$. Let $\th'$ be the principal congruence generated
by this pair. Then $\th' \subseteq \th$ and
$\th'$ is described in Lemma \ref{lm:equal}.

Let $\bar g, \bar g' \in Q_m$ both be transformations of rank $j-1$ that are in different $\Hg$-classes. Such elements clearly exist. By Lemma \ref{lm:equal},
$((f, \bar g),(f,\bar g'))\in \th'\subseteq \th$. As $(\bar g,\bar g') \not\in\Hg$, the $\th$-dlock $X'$ containing $D_{i,j-1}$ is either of type $FF$, $\Hg F$, or $\varepsilon F$, depending on if there exists
$((\hat f, \hat g),(\hat f', \hat g')) \in \th \cap X^2$ with $(\hat f, \hat f') \not\in \Hg$ or $\hat f \ne \hat f'$.

Now assume that $i > 0*$, and let $\hat f \in Q_n$ be a transformation of rank $i-1$. Once again by Lemma \ref{lm:equal}, it follows that $\th'$, and therefore $\th$,
contains $((\hat f, g),(\hat f, g'))$. Let $X'$ be the dlock
containing $D_{i-1,j}$, then  $((\hat f, g),(\hat f, g')) \in \th \cap {X'}^2$. It follows that $X'$ must be of a type for which $\pi_2(\th \cap {X'}^2)$ is not the identity. The six
listed types in the statement of the theorem
are exactly those for which
 this condition is satisfied.

Now  $((\hat f, g),(\hat f, g'))= ((\hat f, g), (\hat f \cdot \id_{S_{i-1}}, g \cdot \sigma))\in \th$. In the fourth case, i.e.\  when $X'$ is a dlock of type $F \Hg$ with normal subgroup $\bar N$, 
we have that
$ \sigma \in \bar N$. Similarly in the fifth and sixth cases, we get that $(\id_{S_{i-1}},\sigma) \in \bar N$. As $\sigma $ generates $N'$ as a normal subgroup, the statements in the last three cases follow.
\end{proof}
We conclude that the dlocks of type $\varepsilon \Hg$ can be placed onto the ``west-facing'' slopes of the landscape made up of the dlocks of type  $FF$, $\Hg F$, or $\varepsilon F$. For any such slope the dlocks of type $\varepsilon \Hg$, must be ``staked on the top of each other'',
with the  initial $\varepsilon \Hg$-dlock being placed on either a ``step" of the $FF$-$\Hg F$-$\varepsilon F$-landscape or on a dlock of type $\Hg \Hg$ or $F\Hg$ (see Figure 5).
Symmetric statements hold for dlocks of type $\Hg \varepsilon$.

\begin{figure}[h]
\[
\xy
(20,30)*{}="1";
(15,35)*{}="2";
(10,40)*{}="3";
(5,45)*{}="4";
(0,50)*{}="5";
(10,60)*{}="6";
(15,55)*{}="7";
(15,65)*{}="8";
(20,60)*{}="9";
(25,65)*{}="10";
(30,60)*{}="11";
(35,55)*{}="12";
(40,50)*{}="13";
(35,45)*{}="14";
(30,50)*{}="15";
(25,55)*{}="16";
(50,50)*{}="17";
(45,55)*{}="18";
(55,55)*{}="19";
(60,40)*{}="20";
(55,45)*{}="21";
(60,50)*{}="22";
(55,-5)*{}="O";
(65,55)*{}="23";
(65,45)*{}="24";
(70,50)*{}="25";
(75,55)*{}="26";
(85,55)*{}="29";
(75,45)*{}="27";
(80,50)*{}="28";
(85,45)*{}="30";
(90,50)*{}="31";
(95,55)*{}="32";
(95,35)*{}="33";
(100,40)*{}="34";
(105,45)*{}="35";
"O";"5"**\dir{-};
"5";"8"**\dir{-};
"8";"14"**\dir{-};
"1";"18"**\dir{-};
"18";"20"**\dir{-};
"17";"19"**\dir{-};
"19";"24"**\dir{-};
"21";"23"**\dir{-};
"20";"26"**\dir{-};
"28";"29"**\dir{-};
"31";"29"**\dir{-};
"26";"33"**\dir{-};
"31";"34"**\dir{-};
"32";"35"**\dir{-};
"30";"33"**\dir{-};
"30";"32"**\dir{-};
"O";"35"**\dir{-};
"25";"27"**\dir{-};
"27";"28"**\dir{-};
"23";"25"**\dir{-};
"15";"12"**\dir{-};
"12";"13"**\dir{-};
"16";"11"**\dir{-};
"11";"12"**\dir{-};
"9";"10"**\dir{-};
"11";"10"**\dir{-};
"6";"7"**\dir{-};
"9";"7"**\dir{-};
"4";"10"**\dir{-};
"3";"11"**\dir{-};
"2";"12"**\dir{-};
(58,20)*{\mbox{\textit{{\Huge FF}}}};
(10,55)*{\varepsilon  F};
(15,50)*{\varepsilon  F};
(20,45)*{\varepsilon  F};
(15,60)*{{ \varepsilon}{\mathcal H}};
(25,60)*{{ \varepsilon}{\mathcal H}};
(30,55)*{{ \varepsilon}{\mathcal H}};
(35,50)*{{ \varepsilon}{\mathcal H}};
(25,40)*{{\mathcal H}  F};
(60,45)*{{\mathcal H}{\mathcal H}};
(75,50)*{{\mathcal H}{ \varepsilon}};
(85,50)*{{\mathcal H}{\mathcal H}};
(55,50)*{{ \varepsilon}{\mathcal H}};
(65,50)*{{\mathcal H}{ \varepsilon}};
(92,43)*{F{\mathcal H} };
(96,48)*{F \varepsilon};
(0,50)*{};
(20,70)*{};
(25,65)*{};
(30,60)*{};
(55,-5)*{}="aO";
(0,50)*{}="a1";
(5,45)*{}="a2";
(10,40)*{}="a3";
(15,35)*{}="a4";
(20,30)*{}="a5";
(25,25)*{}="a6";
(30,20)*{}="a7";
(35,15)*{}="a8";
(40,10)*{}="a9";
(45,5)*{}="a10";
(50,0)*{}="a11";
(60,0)*{}="b3";
(65,5)*{}="b4";
(70,10)*{}="b5";
(75,15)*{}="b6";
(80,20)*{}="b7";
(85,25)*{}="b8";
(90,30)*{}="b9";
(95,35)*{}="b10";
(100,40)*{}="b11";
(105,45)*{}="b12";
(0,50)*{}="c1";
(5,55)*{}="c2";
(10,60)*{}="c3";
(15,65)*{}="c4";
(20,70)*{}="c5";
(25,75)*{}="c6";
(30,80)*{}="c7";
(35,85)*{}="c8";
(40,90)*{}="c9";
(45,95)*{}="c10";
(50,100)*{}="c11";
(50,100)*{}="d1";
(55,95)*{}="d2";
(60,90)*{}="d3";
(65,85)*{}="d4";
(70,80)*{}="d5";
(75,75)*{}="d6";
(80,70)*{}="d7";
(85,65)*{}="d8";
(90,60)*{}="d9";
(95,55)*{}="d10";
(100,50)*{}="d11";
"a1";"d1"**\dir{.};
"a2";"d2"**\dir{.};
"a3";"d3"**\dir{.};
"a4";"d4"**\dir{.};
"a5";"d5"**\dir{.};
"a6";"d6"**\dir{.};
"a7";"d7"**\dir{.};
"a8";"d8"**\dir{.};
"a9";"d9"**\dir{.};
"a10";"d10"**\dir{.};
"a11";"d11"**\dir{.};
"aO";"c1"**\dir{.};
"b3";"c2"**\dir{.};
"b4";"c3"**\dir{.};
"b5";"c4"**\dir{.};
"b6";"c5"**\dir{.};
"b7";"c6"**\dir{.};
"b8";"c7"**\dir{.};
"b9";"c8"**\dir{.};
"b10";"c9"**\dir{.};
"b11";"c10"**\dir{.};
"aO";"b12"**\dir{.};
"c11";"b12"**\dir{.};
\endxy
\]
\caption{A possible configuration for dlocks of type $FF$, $\Hg F$,  $F \Hg$, $\varepsilon F$, $F \varepsilon$, $\Hg \Hg$, $\varepsilon \Hg$, and $\Hg \varepsilon$}\label{He}
\end{figure}

We will not derive any additional conditions for dlocks of type $\varepsilon \varepsilon$, so we may use them to fill out the remaining ``spaces'' in our landscape without violating any conclusion achieved so far.

The results of this section give us tight constraint about the structure of any congruence $\th$ on $Q_m \times Q_n$. In our next theorem we will state that all the conditions we have derived so far are
in fact sufficient to define a congruence.

\begin{thm}\label{mainQ} Let $Q \in\{\T, \PT,\In\}$ and assume that $Q_m, Q_n$ are non-trivial. Suppose that we are given a partition $\mathcal{P}$ of $Q_m \times Q_n$ that preserves $\Dg$-classes and that to each part $B$ of $\mathcal{P}$, we associate a type from $\{F,\Hg,\varepsilon\}^2$ and, if the type of $B$ differs from $FF$,
 a group $N_B$. Suppose further that the following conditions are met:
   \begin{enumerate}
   \item The partition $\mathcal{P}$ has at most one part $B$ of type $FF$, and if this is the case, then $B=D_P$, where $P$ is a downward-closed subset of
   $\{0*,\dots, m\} \times \{0*, \dots,n\}$ such that
    \begin{enumerate}\item$P\ne \{0\} \times \{0,\dots,j\}$ for all $j\in \{0, \dots, n\}$;
   \item $P\ne \{0,\dots, i\}\times \{0\}$ for all $i \in \{0, \dots, m\}$.
   \end{enumerate}
\item If $B$ is a part of type $\Hg F$ or $\varepsilon F$, then $B=D_P$ where $P$ is of the form  $\{i\} \times \{0*, \dots, j\}$ for some $0* \le i\le m$, $1 \le j \le n$, and
$N_B \unlhd S_i$. Moreover \label{c:HFepsF shape}
          \begin{enumerate}
\item If $B$ has type $\Hg F$, then $i \ge 2$, and $N_B \ne \varepsilon_i$;
\item If $B$ has type $\varepsilon F$, then $N_B = \varepsilon_i$.
\end{enumerate}
\item The dual of condition (\ref{c:HFepsF shape}) holds for $B$ of type $F \Hg$ and $F \varepsilon$.
\item $\mathcal{P}$ has at most one part of type $\Hg F$. If $D_P$, with $P= \{i\} \times \{0*, \dots, j\}$, is such a part, then $\mP$ has a part $B'$ of type $FF$, such that
$D_{P'} \subseteq B'$, where $P'= \{0*,\dots,i-1\} \times \{0*, \dots, j\}$. \label{c:HF loc}
\item The dual of condition (\ref{c:HF loc}) holds for $B$ of type $F \Hg$.
\item Let $J$ be the set of values $(i,j_i)$ such that $\{i\} \times \{0*, \dots,j_i\}$ is the index set of the parts of $\mP$ having type $\varepsilon F$.
Then $\pi_1(J)$ is a
set of consecutive integers (possibly empty), and the values $j_i$ are non-increasing in $i$.

Moreover, if $\pi(J)$ is non-empty and the smallest value $i'$ of $\pi_1(J)$ is larger then $0*$, then $\mP$ has a part $B'$ of type $\Hg F$ or $FF$, such that
$D_{P'} \subseteq B'$, where $P'= \{i'-1\} \times \{0*, \dots, j_{i'}\}$. \label{c:epsF loc}
\item The dual of condition (\ref{c:epsF loc}) holds for the set of $\mP$-parts of type $F \varepsilon$.
\item If $B$ is a part of type $\Hg \Hg$, $\varepsilon \Hg$, $\Hg \varepsilon$, or $\varepsilon \varepsilon$ then $B=D_{i,j}$ for some $0* \le i \le m, 0* \le j \le n$, and
$N_B \unlhd S_i \times S_j$.
Moreover,
\begin{enumerate}
\item[(a)] if $B$ is of type $\Hg \Hg$, then $i \ge 2$, $j\ge 2$, and $\pi_1(N_B) \ne \varepsilon_i, \pi_2(N_B) \ne \varepsilon_j$;
\item[(b)]  if $B$ is of type $\varepsilon \Hg$, then $j \ge 2$, and  $N_B= \varepsilon_i \times N'$ for some $N' \ne \varepsilon_j$;
\item[(c)]  if $B$ is of type $\Hg \varepsilon$,  then $i \ge 2$, and  $N_B= N' \times\varepsilon_j$ for some $N' \ne \varepsilon_i$;
\item[(d)]  if $B$ is of type $\varepsilon \varepsilon$, then $ N_B= \varepsilon_i \times \varepsilon_j$.
\end{enumerate}
\item  If $B=D_{i,j}$ is a part of type $\Hg \Hg$, let $B'$ be the part containing $D_{i,j-1}$. Then $B'$ is either of type $FF$ or of  type $\Hg F$ and
$\pi_1(N_B) \subseteq N_{B'}$. The dual condition holds for the part $B''$ containing $D_{i-1,j}$.
\item Let $B=D_{i,j}$ be a part of type $\varepsilon \Hg$ with $N_B=\varepsilon_i \times N'$. Let $B'$ be the part containing $D_{i,j-1}$. Then $B'$ is of type $FF$, $\Hg F$, or $\varepsilon F$. \label{c:epsH loc}

Moreover, if $i >0*$, the part $B''$ containing $D_{i-1,j}$ satisfies one of the following conditions:
\begin{enumerate}
\item $B''$ is of type $FF$;
\item  $B''$ is of type $\Hg F$;
\item  $B''$ is of type $\varepsilon F$;
\item  $B''$ is of type $F \Hg$  and  $N' \subseteq N_{B''}$;
\item $B''$ is of type $\Hg \Hg$ and $\varepsilon_{i-1} \times N' \subseteq  N_{B''}$;
\item $B''$ is of type $\varepsilon \Hg$ and $\varepsilon_{i-1} \times N' \subseteq  N_{B''}$.
\end{enumerate}
\item The dual of condition (\ref{c:epsH loc}) holds for the $\mP$-parts of type $H \varepsilon$.
\end{enumerate}
Suppose that on each $\mP$-part $B$ we define a binary relation $\th_B$ as follows:
\begin{enumerate}
\item[(i)] If $B$ has type $FF$ let $\th_B=B^2$;
\item[(ii)] If $B$ has type $\Hg F$ or $\varepsilon F$, let $(f,g) \th_B (f',g')$ if and only if $f \Hg f'$ and $ f'=f\cdot \sigma$ for some $ \sigma \in N_B$;
\item[(iii)] If $B$ has type $F \Hg$ or $F \varepsilon$, let $(f,g) \th_B (f',g')$ if and only if $g \Hg g'$ and $ g'=g\cdot \sigma$ for some $ \sigma \in N_B$;
\item[(iv)] If $B$ has type $\Hg \Hg$, $\varepsilon \Hg$, $\Hg \varepsilon$, or  $\varepsilon \varepsilon$, let $(f,g) \th_B (f',g')$ if and only if $f \Hg f'$, $g \Hg g'$ and
$ (f',g')=(f \cdot \sigma, g\cdot \tau)$ for some $ (\sigma,\tau) \in N_B$.
\end{enumerate}
Let $\th=\cup_{B \in \mP}\, \th_B$. Then $\th$ is a congruence on $Q_m \times Q_n$. 

Conversely, every congruence  on $Q_m \times Q_n$ can be obtained in this way.
\end{thm}

%

\begin{proof} The ``converse" part of this last theorem follows from Lemmas \ref{lm:t1} to \ref{th:5l} and, where applicable, their dual versions.
To show that $\th$ is a congruence involves checking for each $((f,g),(f',g'))\in \th$, that the principal congruence generated by  $((f,g),(f',g'))$ is contained in $\th$, using our results on principal
congruences from Section \ref{s:principal}. The proof
is straightforward, but tedious due to the many different cases to be considered, and is left to the reader.
\end{proof}
\begin{obs}\label{obs}
We remark that Theorem \ref{mainQ} also holds for the more general case of semigroups of the  form $Q_m\times P_n$, where $Q,P\in \{\PT,\T,\In\}$, provided that the expression $0*$ is interpreted in the context of the relevant factor. In fact, nearly all our results and proofs carry over to the case of $Q_m\times P_n$ without any other adjustments. The exceptions are  Lemma \ref{lem:ideals}, Theorem \ref{th:0h}, and Theorem \ref{th:2h}, which require simple and straightforward modifications. 
\end{obs}

\section{Products of three transformation semigroups}\label{triples}

As said above, the results of this paper essentially solve the problem of describing the congruences of $Q_{n_1}\times Q_{n_2}\times Q_{n_3}\ldots\times Q_{n_k}$, the product of finitely many transformation semigroups of the types considered, although the resulting description of the congruences would require heavy statements and notation, but not much added value. To illustrate our point, we have included a series of figures that give an idea of how the dlock-structure of a triple product looks like. 

In the following figures, each $\Dg$-class is represented by a cube, and $\Dg$-classes belonging to the same dlock are combined into a colour-coded polytope.  
The figure is orientated so that the 
cube representing the $\Dg$-class of $D_{0*,0*,0*}$ is furthest away from the observer and obstructed from view.

To reduce the number of required colors, types that are obtained by a permutation of the coordinates have the same colour. Each figure adds the dlocks from one such colour group to the previous figure. For example, Figure \ref{3dFFF} contains one grey dlock of type $FFF$, while Figure \ref{3dFFH} adds three  red dlocks of types $FF\Hg$, $F\Hg F$ and $\Hg FF$. 

The following pairs of figures, from the dual and triple product case, can be considered to be in correspondence with each other: (Figure \ref{FF}, Figure \ref{3dFFF}), (Figure \ref{Fe}, Figure \ref{3dFee}), (Figure \ref{HH}, Figure \ref{3dHHH}), and (Figure \ref{He}, Figure \ref{3dHee}).  In the remaining cases, no such direct correspondence exist due to the extra dlock-types present in the three product case.

To obtain a final configuration from Figure \ref{3dHee}, one needs to fill out all remaining spaces with cubes that represent dlock-type $\varepsilon \varepsilon \varepsilon$. Put together, the figures demonstrate a large number, but not all, of the possible configuration of dlock-types.

\begin{figure}[!htb]
\begin{center}
\includegraphics[height=2in,width=3in,angle=0]{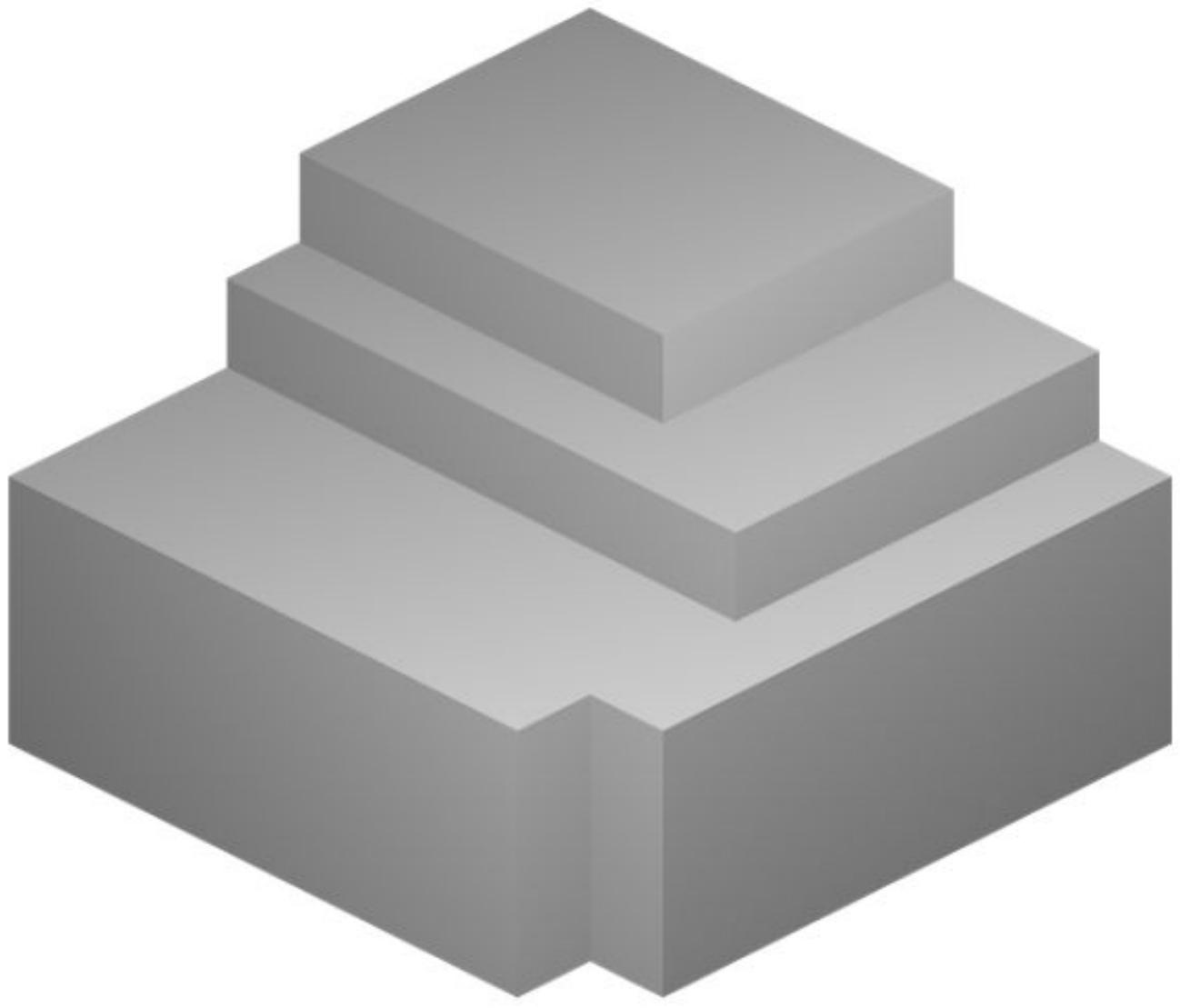}
\caption{$F,F,F$}\label{3dFFF}
\end{center}
\end{figure}

\begin{figure}[!htb]
\begin{center}
\includegraphics[height=2in,width=2in,angle=0]{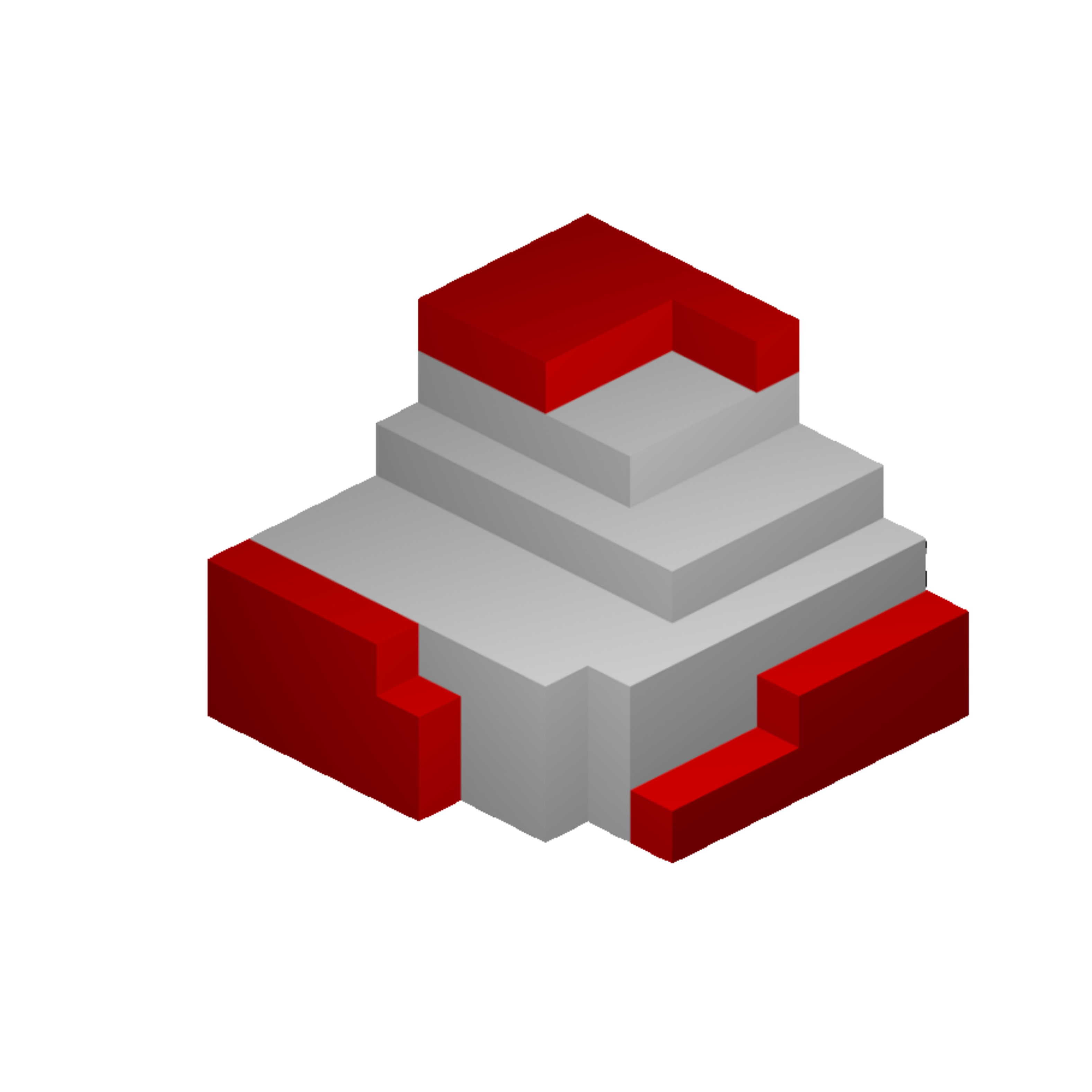}
\caption{$F,F,\Hg$}\label{3dFFH}
\end{center}
\end{figure}

\begin{figure}[!htb]
\begin{center}
\includegraphics[height=2in,width=2in,angle=0]{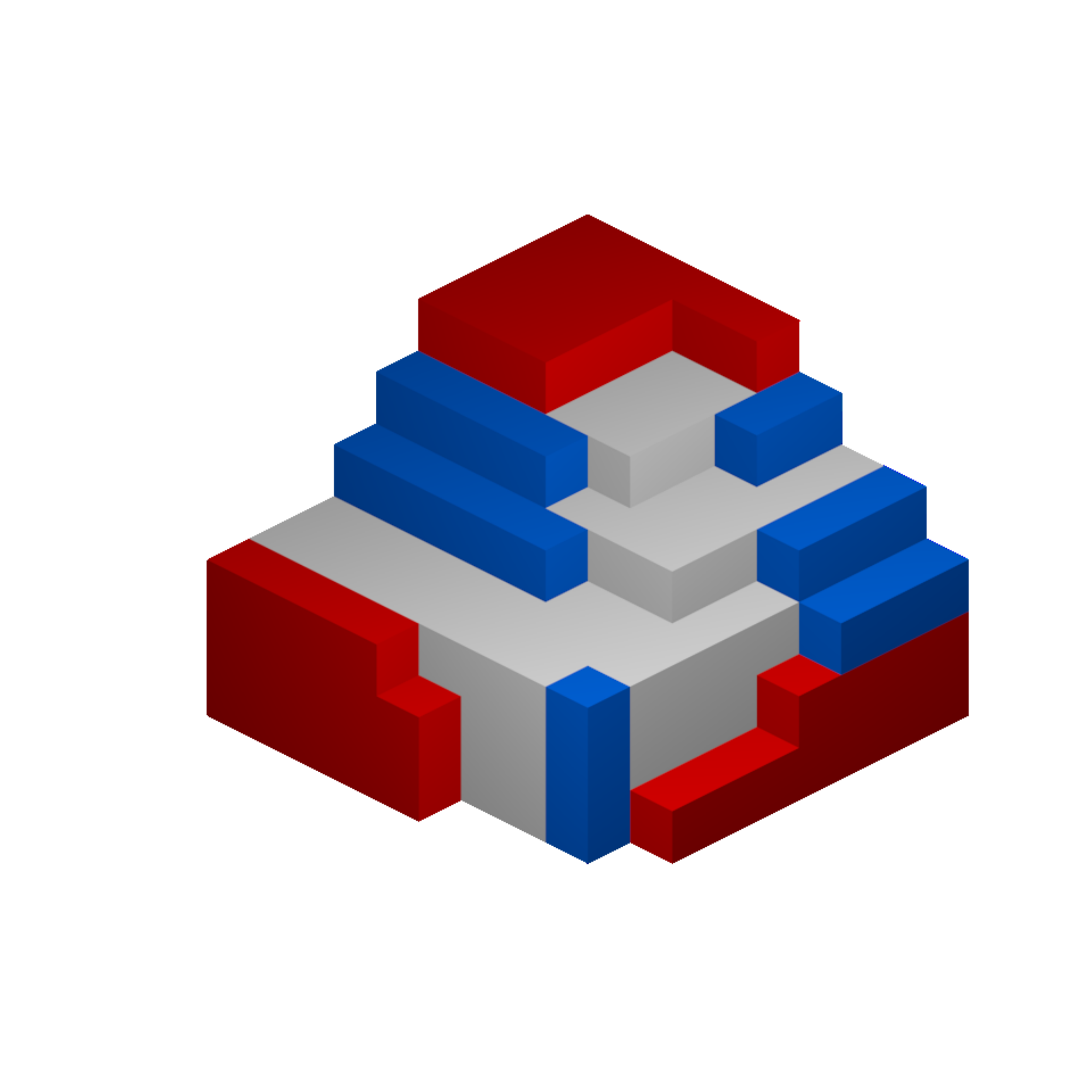}
\caption{$F,\Hg,\Hg$}\label{3dFHH}
\end{center}
\end{figure}

\begin{figure}[!htb]
\begin{center}
\includegraphics[height=2in,width=2in,angle=0]{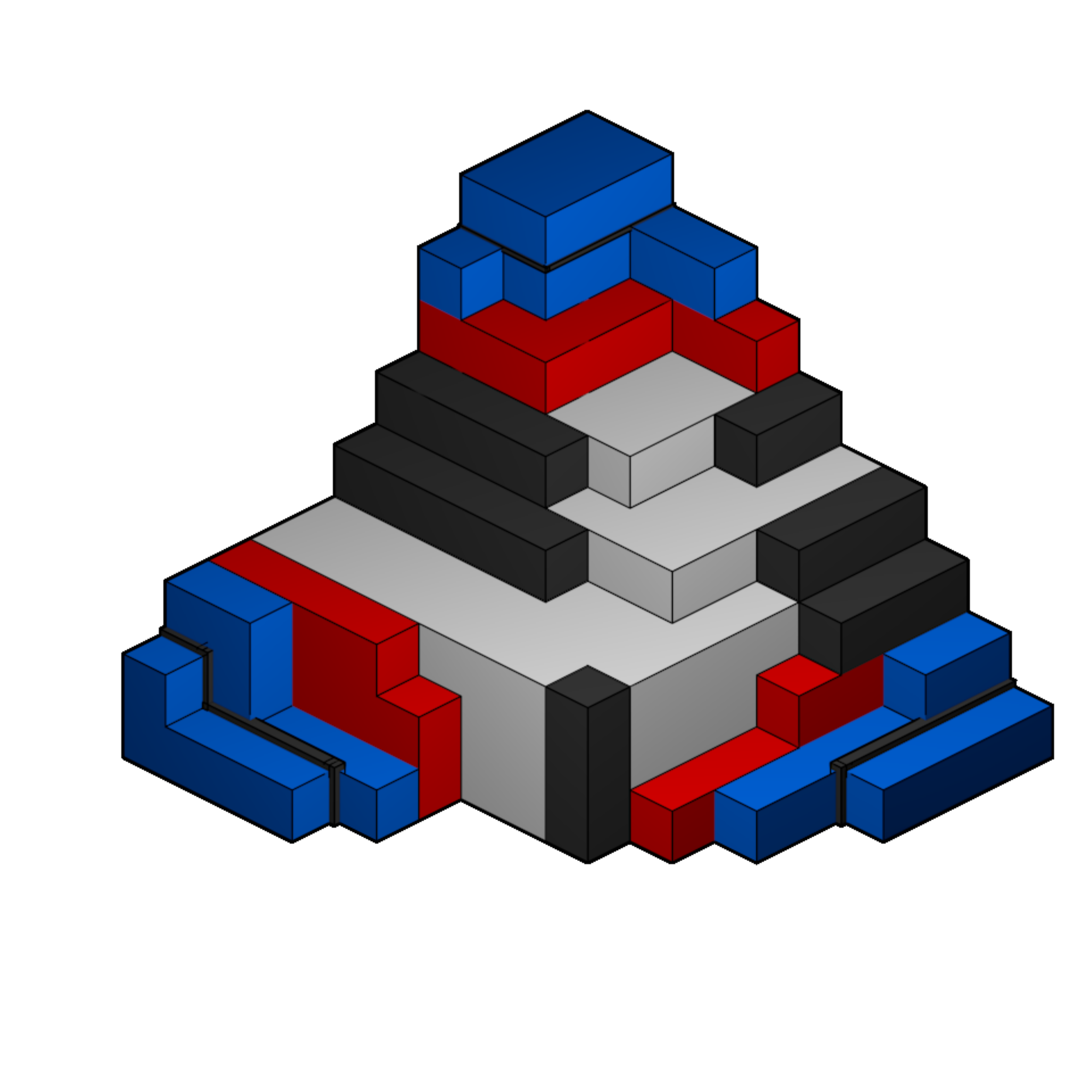}
\caption{$F,F,\varepsilon$}\label{3dFFe}
\end{center}
\end{figure}

\begin{figure}[!htb]
\begin{center}
\includegraphics[height=2in,width=2in,angle=0]{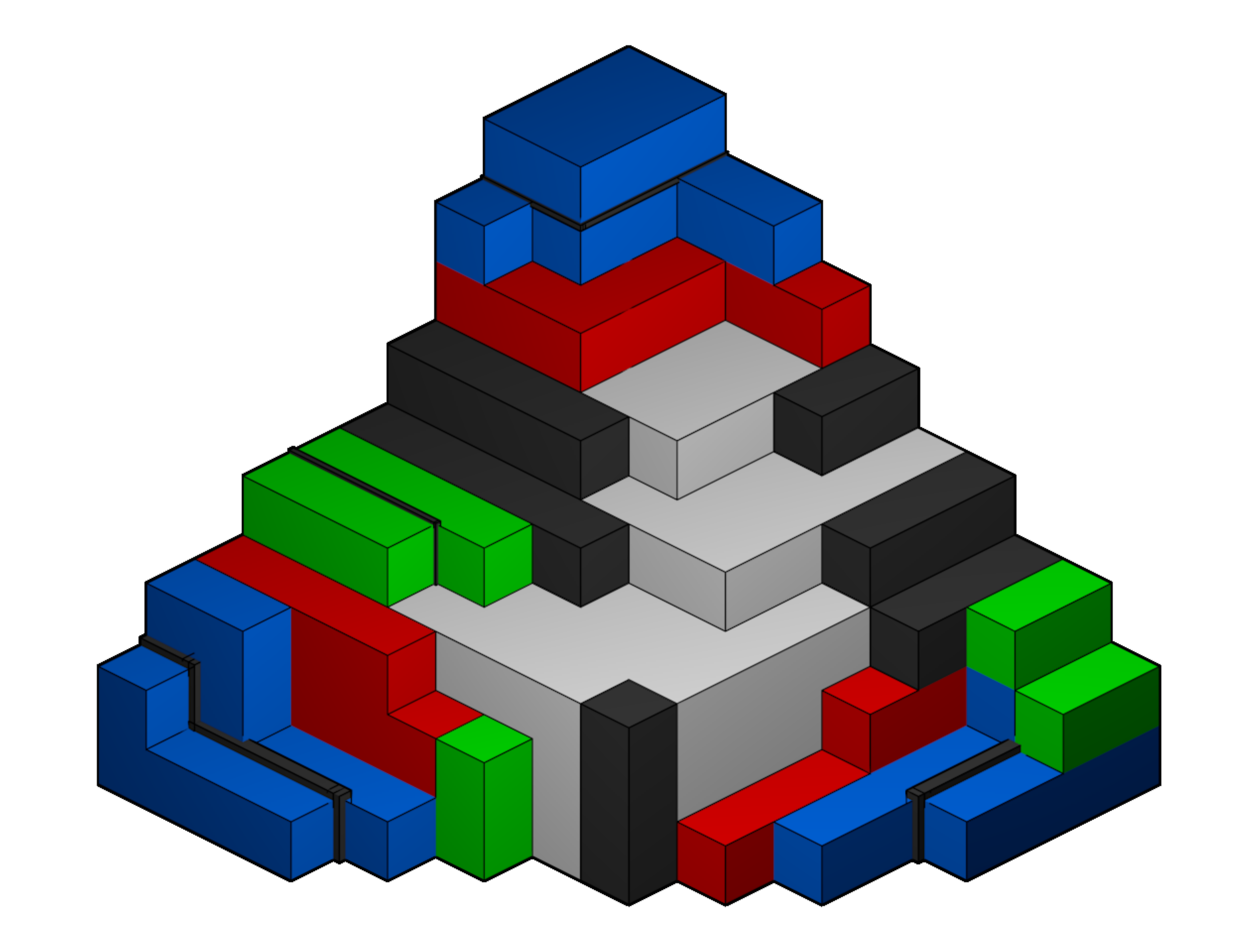}
\caption{$F,\Hg,\varepsilon$}\label{3dFHe}
\end{center}
\end{figure}

\begin{figure}[!htb]
\begin{center}
\includegraphics[height=2in,width=2in,angle=0]{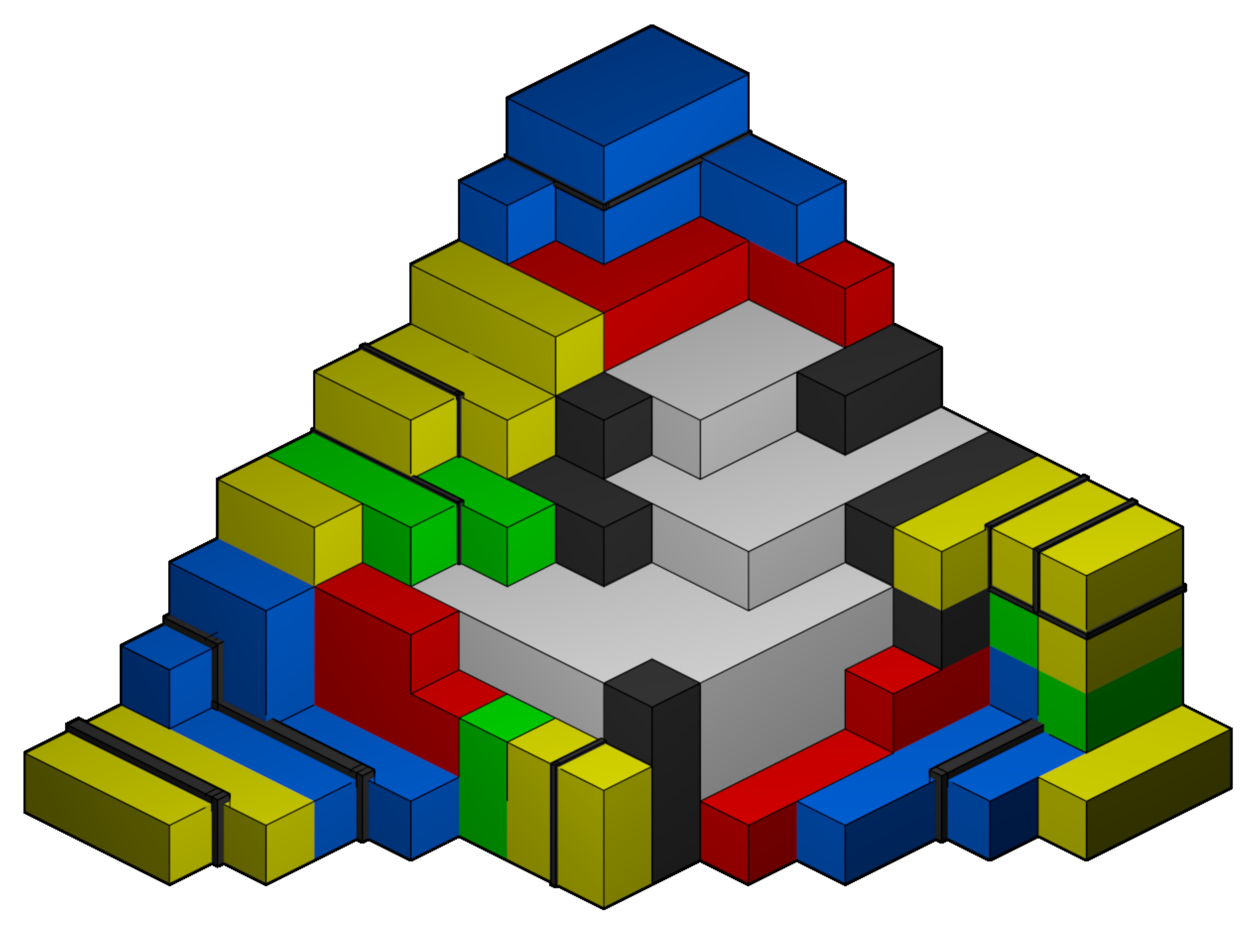}
\caption{$F,\varepsilon,\varepsilon$}\label{3dFee}
\end{center}
\end{figure}

\begin{figure}[!htb]
\begin{center}
\includegraphics[height=2in,width=2in,angle=0]{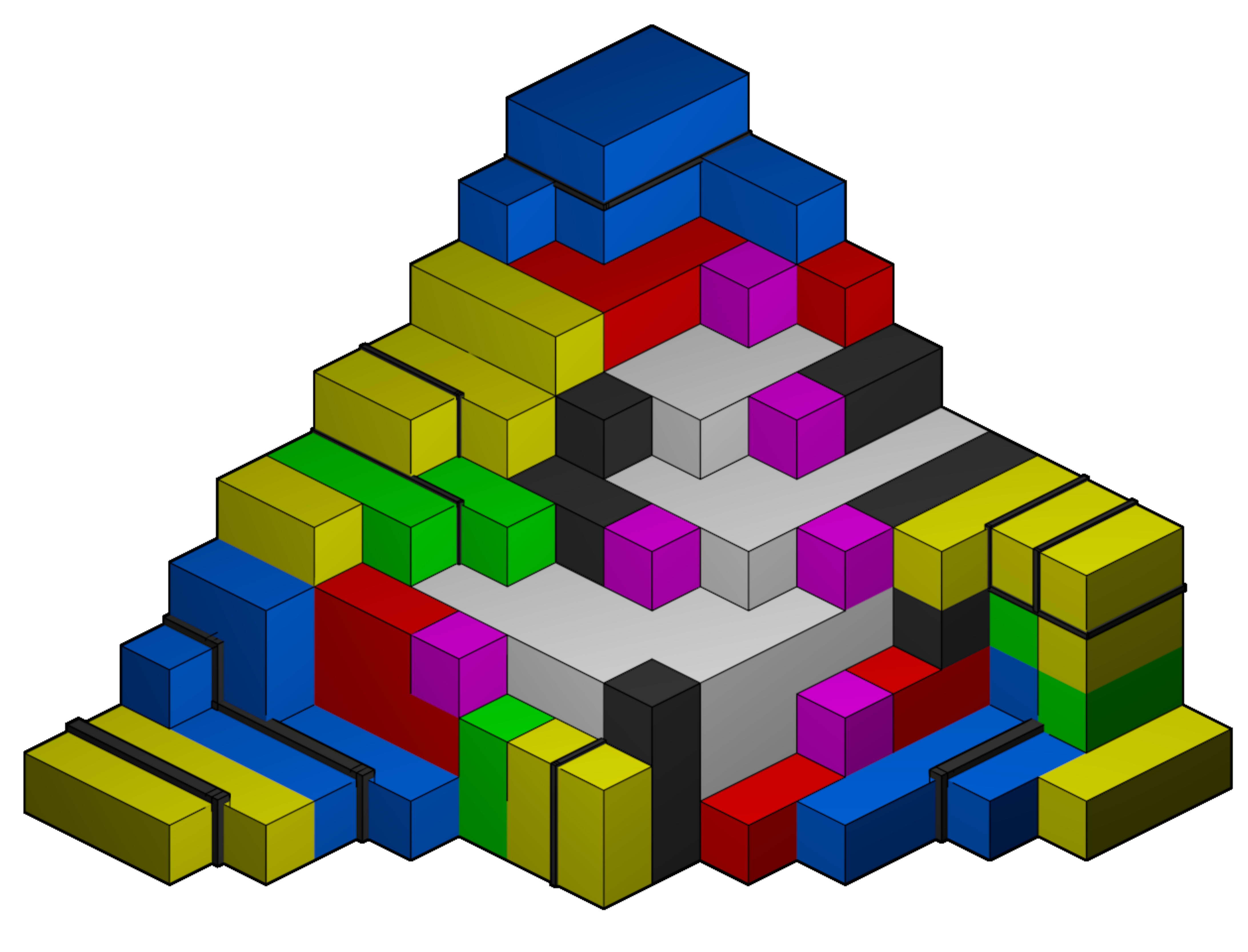}
\caption{$\Hg,\Hg,\Hg$}\label{3dHHH}
\end{center}
\end{figure}

\begin{figure}[!htb]
\begin{center}
\includegraphics[height=2in,width=2in,angle=0]{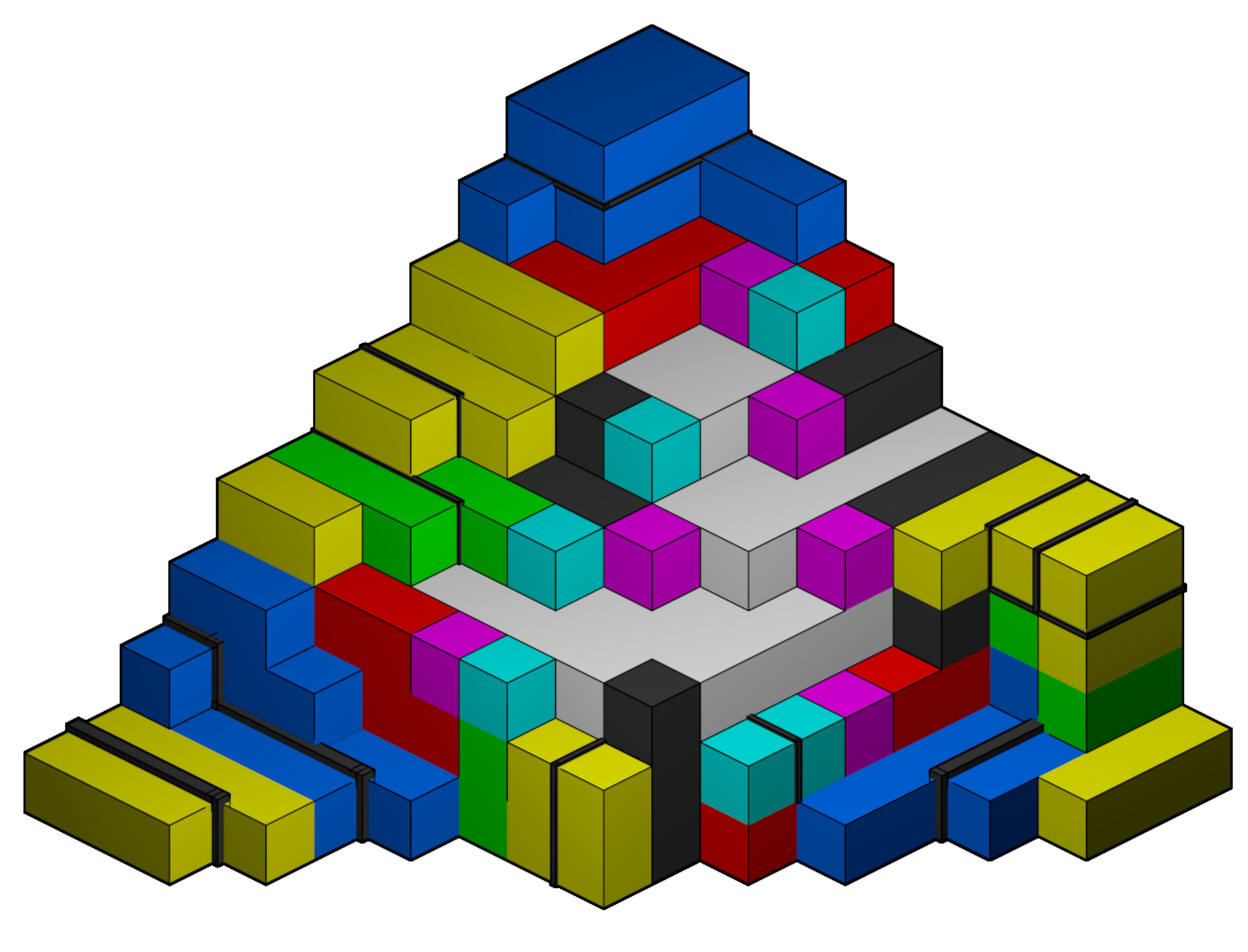}
\caption{$\Hg,\Hg,\varepsilon$}\label{3dHHe}
\end{center}
\end{figure}

\begin{figure}[!htb]
\begin{center}
\includegraphics[height=2in,width=2in,angle=0]{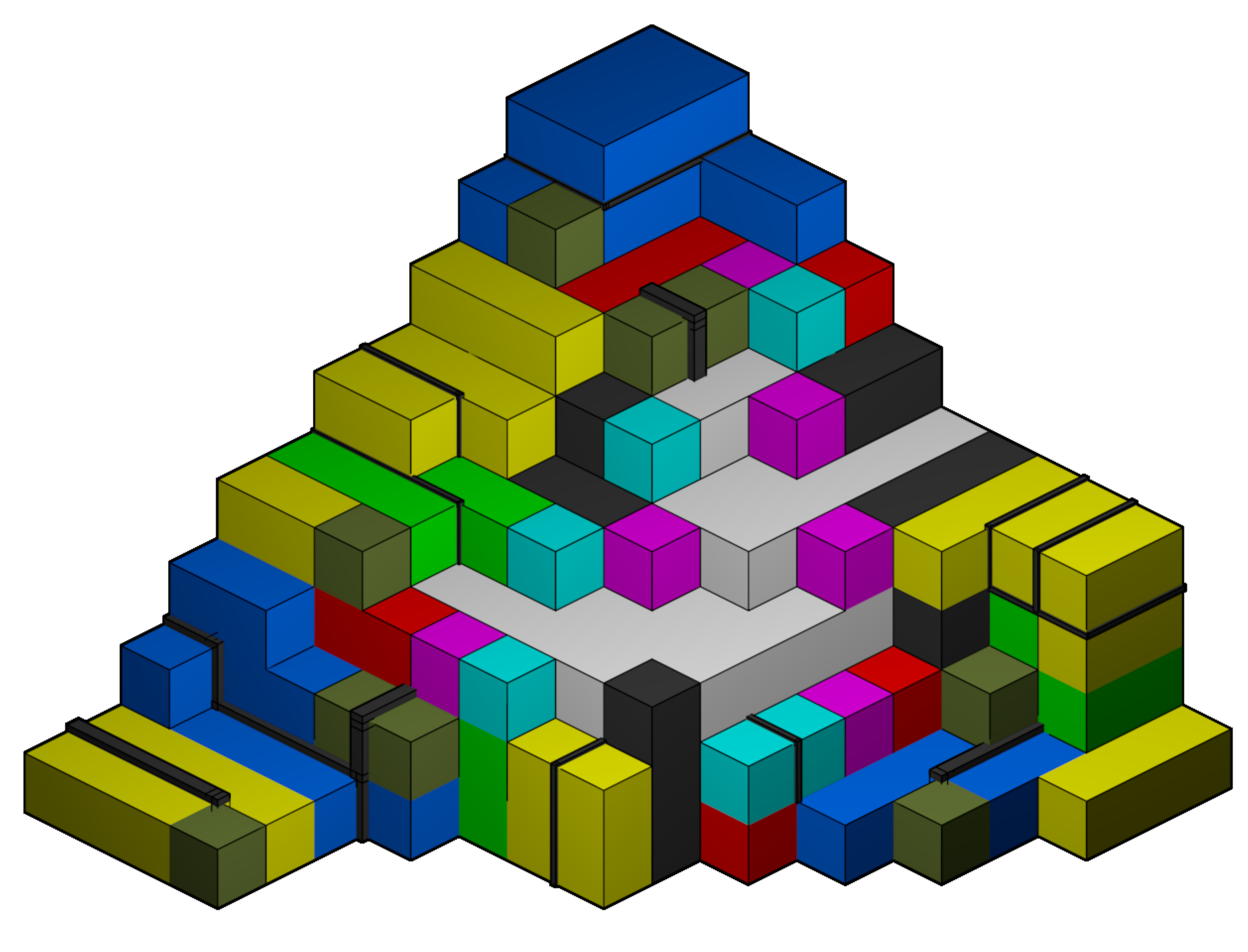}
\caption{$\Hg,\varepsilon,\varepsilon$}\label{3dHee}
\end{center}
\end{figure}

\section{Matrix Monoids}\label{s:matrix}
Let $F$ be a field with multiplicative unit group $F^*$. Consider the multiplicative monoid $F_n$ of all $n \times n$-matrices over $F$. We will, throughout, identify matrices with their induced (left) linear transformation on the vector space $F^n$. The rank, kernel, and image of a matrix are now defined with regard to their usual  meanings from linear algebra. Note that in particular the definition of kernel is now different from the notation of kernel used in the 
section on transformation semigroups. In addition, matrix multiplication corresponds to a composition of linear transformations that is left-to right, and hence inverse from the situation for transformation monoids.

 In this section, we will determine the principal congruences on the monoid $F_m \times F_n$. As it turns out, this case closely mirrors the situation of the semigroup $\PT_m \times \PT_n$. In many cases,  transferring 
the proofs of
the previous sections to our new setting requires only an adaptation of notation. In those cases, we will leave it to the reader to make the relevant changes. 

Other than notional changes, the main difference from the situation on $\PT_m \times \PT_n$ corresponds to the description of the congruence generated by  a pair
 of the form $((f,g),(f,g'))$.
For matrix monoids, this congruence properly relates to the congruence generated by some $((A,B),(\lambda A, B'))$, where $\lambda \in F^*$. Hence our description needs to be adapted to take care of the extra parameter $\lambda$.

We will start with by recalling several facts about the monoids $F_n$.

Recall that two matrices are $\Rg$-related if they have the same image, $\Lg$-related if they have the same kernel, $\Hg$-related if they have the same image and kernel, and
$\Dg$-related if they have the same rank.

We let $e_{i,j}$, for $ 1 \le i,j \le n$, be the elements of the standard linear basis of $F_n$ and set $E_i=e_{1,1}+e_{2,2}+\dots +e_{i,i}$. We identify the linear group $\gl(i,F)$ with the maximal subgroup of $F_n$ that 
contains $E_i$. Denote the identity matrix on $F_n$ by $1$ and the zero matrix by $0$. We have $1=E_n$, and we set $E_0=0$.

The description of the congruences of $F_n$ can be found in \cite{Malcev2}.
We will however use the following slightly different description from \cite{MatWeb}. While this is an unpublished source, the  two characterization only differ on condition 
(b) of the following description, and it is an easy exercise to check that they are indeed equivalent.
\begin{thm}\label{th:congrFnmat}
A binary non-universal  relation $R$ on $F_n$ is a
congruence if and only if, there exists $\mu \in \{0, \dots, n\}$ and

\begin{enumerate}	
\item   there exists a normal subgroup $\bar G_\mu$  of $\gl(\mu,F)$;
\item  if $ \mu \le n-1$ there exist subgroups $G_{\mu+1}, G_{\mu+2}, \dots,
G_{n}$ of $F^*$
 such that $G_n \subseteq G_{n-1} \subseteq \dots \subseteq
G_{\mu+1}$ and
$G_{\mu+1} E_{\mu} \subseteq \bar G_\mu$;
\item
two matrices $A$ and $B$ are in $R$ if and only if one of the following
conditions holds:
\begin{enumerate}
       \item  $\rank A < \mu$ and $\rank B < \mu$;
       \item  $\rank A = \rank B= \mu$, and there exist $s_1,s_2 \in \gl(n,F)$ such that $s_1 A s_2$ and $s_1 B s_2$ are both in $\gl(\mu,F)$, and belong to the 
       same coset of $\gl(\mu,F)$ modulo $\bar G_\mu$;\label{cond:differ}
      \item  $\rank A= \rank B =i$, for some $\mu < i\le n$, and
$A=\lambda B$ for some $\lambda \in G_i$.

\end{enumerate}
\end{enumerate}	
\end{thm}
In addition, we need the following result from \cite{MatWeb}. 
\begin{lem}\label{lm:subspace}
A matrix $A \in F_n$ is a non-zero scalar multiple of the identity matrix if and only if $A$ fixes all subspaces of $F^n$ of dimension $n-1$.
\end{lem}

%
%
%

 The following will be applied later without further reference. Let $A,B \in F_n$. From Theorem \ref{th:congrFnmat}, the principal congruence of $F_n$ generated by $(A,B)$ corresponds to the following parameters
 in the theorem:
  \begin{itemize}
  \item If $A=\lambda B$  for some $\lambda \in F^*$, then $\mu=0$, $G_n=\dots=G_{\rank A+1}=\{1\}$, $G_{\rank A}=\dots=G_1=\langle \lambda\rangle$, and $\bar G_0=\{0\}$;
  \item  If $\rank A= \rank B$, $A \ne \lambda B$ for all $\lambda \in F^*$,  and $A \Hg B$, then $\mu=\rank A$, $G_n=\dots=G_{\mu+1}=\{1\}$, and $\bar G_\mu$ is the normal subgroup of GL$(\mu,F)$ that corresponds to the congruence generated by the pair $(s_1 A s_2, s_1B s_2)$, where $s_1,s_2 \in \gl(n,F)$ are such that $s_1 A s_2, s_1B s_2 \in \gl(\rank A, F)$;
  \item  If $ (A,B)\notin\Hg$, and $\max\{\rank A, \rank B\} \le n-1$, then $\mu=\max\{\rank A, \rank B\}+1$, $G_n=\dots=G_{\mu+1}=\{1\}$, and $\bar G_\mu=\{E_\mu\}$. 
  \item If $(A,B) \notin \Hg$ and $\max\{\rank A, \rank B\} =n$, then $(A,B)$ generates the universal congruence.
\end{itemize}

For $0 \le i \le m$, let $I_i^{(n)}$ stand for the ideal of $F_n$  consisting of all matrices $A$ with $\rank A\le i$.
We will usually just write $I_i$ if $n$ is deducible from the context. Let $\theta_{I_i}$ stand for the Rees congruence on $F_n$ defined by the ideal $J_i$.


\section{Principal congruences on $F_m\times F_n$} \label{s:principalF}


For $A \in F_m \cup F_n$, let $|A| = \rank A$, and for $(A,B) \in F_m\times F_n$ let $|(A,B)| = (|A|,|B|)$, where we order these pairs according to the partial order $\le \times \le$.
Throughout, $\pi_1$ and $\pi_2$ denote the projections from $F_m \times F_n$ to the first and second factor, respectively.

We will start with some general lemmas concerning congruences on $F_m\times F_n$. 

\begin{lem} \label{lem:singleFmat} Let  $\theta$ be a congruence on  $F_m\times F_n$ and fix  $A\in Q_m$; let
$$\theta_A:=\{ (B,B')\in F_n\times F_n\mid (A,B)\theta (A,B')\}.$$

Then
\begin{enumerate}
\item $\theta_A$ is a congruence on $F_n$;
\item if $A'\in F_m$ and $|A'|\le |A|$, then $\theta_A\subseteq \theta_{A'}$;
\item if $|A|=|A'|$, then $\theta_A=\theta_{A'}$.
\end{enumerate}
\end{lem}
\begin{proof}
The proof of this lemma is virtual identical to the proof of Lemma \ref{lem:single}, and is obtained from it by the syntactic substitutions $f\to A$, $g \to B$, $Q \to F$.
\end{proof}
%
%
%
%

In an analogous construction, given a congruence $\th$ on $F_m \times F_n$ and fixed $B \in F_n$, we define $\theta_B:=\{ (A,A')\in F_m\times F_m\mid (A,B)\theta (A',B)\}.$

The next result describes the ideals of $F_n\times F_m$.

\begin{lem}\label{lem:ideals1F}
 The ideals of $F_m\times F_n$ are exactly the unions of sets of the form $I^{(m)}_i\times I^{(n)}_j$, where $I^{(m)}_i$ and $I^{(n)}_j$ are ideals of $F_m$ and $F_n$, respectively.
\end{lem}
Once again, a proof for this result can be obtained from Lemma \ref{lem:ideals1} by making obvious adaptations. The same holds for our next result,  we only have to modify parts (1) and (3) of  Lemma \ref{lem:ideals}, using the fact that
$F_m \times F_n$ contains a zero element as in the case that $Q \in\{\mathcal{PT},\mathcal{I}\}$.
%

\begin{lem} \label{lem:idealsF}
Let $\theta$ be a congruence on $F_m\times F_n$. Then
\begin{enumerate}
\item  $\theta $  contains a  class $I_\theta$ which is an ideal;
\item $\th$ contains at most one  ideal class.
\end{enumerate}
\end{lem}

For the remains of this section we fix the following notation.
Let  $K,K'\in F_m$ and $L,L'\in F_n$. Let $\theta$ be a principal congruence on $F_m\times F_n$ generated by $((K,L),(K',L'))$. Let $\theta_1$ be the principal congruence generated by $(K,K')$ on $F_m$ and $\theta_2$ be the principal congruence generated by $(L,L')$ on $F_n$.

Now suppose that $(K,L), (K',L') \in F_m \times F_n$ are such that $K \Hg K'$ and $L \Hg L'$. Let $|K|=i, |L|=j$. To the pair $((K,L),(K',L'))$ we
associate a normal subgroup $H$ of $\gl(i,F) \times \gl(j,F)$ as follows.  As $|K|=i$ and $|L|=j$, there exist $(s_1,s_2),(s_3,s_4) \in S_m\times S_n$ such that 
$(s_1,s_2)(K,L)(s_3,s_4)=(E_i,E_j)$. Then $(s_1,s_2)(K',L')(s_3,s_4)\in  \gl(i,F) \times \gl(j,F)$, and we take $H$ as the the normal subgroup generated by
 $(s_1,s_2)(K',L')(s_3,s_4)$.
 

 We claim that the definition of $H$ is independent of our choice for $(s_1,s_2)$, $(s_3,s_4)$. For suppose that  $(t_1,t_2),(t_3,t_4) \in S_m\times S_n$ are such that
$$(t_1,t_2)(K,L)(t_3,t_4)=(E_i,E_j).$$ Then 
\begin{equation} \label{e:mat} (E_i,E_j)= (t_1,t_2) (s_1,s_2)^{-1}(E_i,E_j)(s_3,s_4)^{-1}(t_3,t_4).\end{equation}
Let $V, W$ be the subspaces of $F^m$ and $F^n$ generated by the columns of $E_i$ and $E_j$, respectively. It is easy to check that left multiplication by
 $(t_1,t_2) (s_1,s_2)^{-1}$
maps  $V \times W$ onto itself. Hence there is $s \in \gl(i,F) \times \gl(j,F)$ such that for all $v \in V \times W$ we have $sv=  (t_1,t_2) (s_1,s_2)^{-1}v$. It follows that 
$sA=  (t_1,t_2) (s_1,s_2)^{-1}A$ for all $A \in \gl(i,F) \times \gl(j,F)$.

Symmetrically, there exists  $t \in \gl(i,F) \times \gl(j,F)$ such that 
$At=  A(s_3,s_4)^{-1}(t_3,t_4)$, for all $A \in \gl(i,F) \times \gl(j,F)$.
Now, (\ref{e:mat}) shows that in the group $ \gl(i,F) \times \gl(j,F)$, we have $t=s^{-1}$. Thus, in $ \gl(i,F) \times \gl(j,F)$,
$$(t_1,t_2)(K',L')(t_3,t_4)=s(s_1,s_2)(K',L')(s_3,s_4)s^{-1}$$
is a conjugate of $(s_1,s_2)(K',L')(s_3,s_4)$ and thus generates the same normal subgroup.

If $K'=\lambda K$ for some $\lambda\in F^*$, it is easy to see that the normal subgroup $H$ of $ \gl(i,F) \times \gl(j,F)$ 
associated with 
$((K,L),(K',L'))$ is contained in $F^*E_i \times \gl(j,F)$. We then associate a normal subgroup $\hat H$ of $F^* \times \gl(j,F)$ with the pair $((K,L),(K',L'))$,
taking $\hat H$ as the image of $H$ under the canonical map from $F^*E_i \times \gl(j,F)$ to $F^* \times \gl(j,F)$. If $L'=\lambda L$ for some $\lambda \in F^*$,
we dually  associate a normal subgroup of $\gl(i,F)\times F^*$.

\begin{lem}\label{lm:HHmat}
Let $\th$ be a congruence on $F_m \times F_n$. Let $(K,L),(K',L')\in F_m \times F_n$ be such that  $(K,L) \th (K',L')$, $K \Hg K'$, and $L \Hg L'$. Let $|K|=i, |L|=j$, and 
$H$ be the normal subgroup of  $\gl(i, F) \times \gl(j,F)$ associated with the pair $((K,L),(K',L'))$.      

Let $(M,N), (M,N') \in F_m \times F_n$ where $|M|=i$, $|N|=j$, $M \Hg M'$ and  $N \Hg N'$. Let $H'$ be the normal subgroup of  $\gl(i, F) \times \gl(j,F)$ associated with the pair $((M,N),(M',N'))$. If $H' \subseteq H$, then
$(M,N)\th (M',N')$. 
\end{lem}

\begin{proof} Let $(s_1,s_2),(s_3,s_4) \in S_m\times S_n$ be the elements considered in the definition of $H$: $(s_1,s_2)(K,L)(s_3,s_4)=(E_i,E_j)$ 
and $H$ is the normal subgroup of $\gl(i,F)\times\gl(j,F)$ generated by 
$(s_1,s_2)(K',L')(s_3,s_4)$. Let $(t_1,t_2),(t_3,t_4) \in S_m\times S_n$ be the corresponding elements taken to define
$H'$. We have that
\begin{equation} (s_1,s_2)(K',L')(s_3,s_4) \th (s_1,s_2)(K,L)(s_3,s_4) = (E_i,E_j).\label{e:groups}\end{equation}
Let $\hat \th$ be the restriction of $\th$ to the group $\gl(i, F) \times \gl(j,F)$, and $R$ the normal subgroup of $\gl(i,F)\times \gl(j,F)$ corresponding to $\hat \th$. By (\ref{e:groups}),
we have 
$(s_1,s_2)(K',L')(s_3,s_4) \in R$. As $H$ is normal, $H \subseteq R$, and so $H' \subseteq R$, since $H' \subseteq H$. 
In particular, $(t_1,t_2)(M',N')(t_3,t_4) \in R$, hence $(t_1,t_2)(M',N')(t_3,t_4) \th (E_i, E_j)$, and thus
\begin{eqnarray*}
(M,N)&=&(t_1,t_2)^{-1}(E_i, E_j)(t_3,t_4)^{-1}  \\
 &\th& (t_1,t_2)^{-1}\left((t_1,t_2)(M',N')(t_3,t_4)\right)(t_3,t_4)^{-1}\\
 &=&(M',N').
\end{eqnarray*}

\end{proof}

%

\begin{lem}\label{lm:eHmat}
Let $\th$ be a congruence on $F_m \times F_n$.  Let $(K,L),(K',L') \in F_m \times F_n$ be such that
 $(K,L) \th (K',L')$, $L \Hg L'$, and $K'=\lambda K$ for some $\lambda \in F^*$. Let $|K|=i, |L|=j$, and 
$\hat H$ be the normal subgroup of  $F^* \times \gl(j,F)$ associated with the pair $((K,L),(K',L'))$.      

Suppose that $(M,N), (M,N') \in F_m \times F_n$ are such that  $|M|=k\le i$, $|N|=j$, $N \Hg N'$, and $M'=\lambda' M$ for some $\lambda' \in F^*$. Let $\hat H'$ be the normal subgroup of  $F^* \times \gl(j,F)$ associated with the pair $((M,N),(M',N'))$. If $\hat H' \subseteq \hat H$, then
$(M,N)\th (M',N')$. 

\end{lem}
\begin{proof} 
We have $K \Hg K'$and $M \Hg M'$ so,
by Lemma \ref{lm:HHmat}, we may assume that $(K,L)=(E_i,E_j)$. We then have $(E_i, E_j)\th (\lambda E_i, L')$, and hence 
$$(E_k,E_j)=(E_k,E_n)(E_i,E_j) \th (E_k,E_n)(\lambda E_i, L')=(\lambda E_k,L').$$
It is straightforward to check that the normal subgroup of $F^* \times \gl(k,F)$ associated with $((E_k,E_j), (\lambda E_k, L'))$ is $\hat H$. 

Let $\bar H$ be the normal subgroup 
of $\gl(k,F)\times \gl(j,F)$ associated with this pair, so that $\hat H=\phi(\bar H)$, where $\phi$ is the natural isomorphism from $F^*E_k\times \gl(j,F)$ to 
$F^*\times \gl(j,F)$. Let $\bar H'$ be the normal subgroup of $\gl(k,F)\times \gl(j,F)$ associated with the pair $((M,N),(M',N'))$. 
Notice that $\bar H' \subseteq F^*E_k \times \gl(j,F)$  since $M'=\lambda'M$.
Then 
$\bar H' =\phi^{-1} (\hat H') \subseteq \phi^{-1}(\hat H) =\bar H$. The result now follows with Lemma \ref{lm:HHmat}.
\end{proof}
It is clear that a dual version of Lemma \ref{lm:eHmat} holds as well.

Now let us look at the principal congruences on $F_m \times F_n$.

\begin{thm}\label{lm:equalF1} Let $\theta$ be a principal congruence on $F_m\times F_n$ generated by $((K,L),(K',L'))$.
If $K'=\lambda K$ and $L'= \nu L$ for some $\lambda, \nu \in F^*$, we have $(M,N)\theta(M',N')$ if and only if  $(M,N)= (M',N')$ or  $|M|\le |K|$, $|N|\le |L|$, and there exists 
a $(\lambda', \nu') \in \langle (\lambda,\nu) \rangle\subseteq F^* \times F^*$ such that $M'=\lambda' M$ and $N'= \nu' N$.
\end{thm}
\begin{proof} The lemma can be shown by applying Lemma \ref{lm:eHmat} followed by its dual. However, we will give a short direct proof.

Let $\th'$ be the binary relation defined by the statement of the lemma.

If $|M|\le |K|$, $|N|\le |L|$  then $M=UKV$, $N=WLX$ for some $U,V\in F_m$, $W,X\in F_n$, and hence
$$(M,N)=(UKV, WLX) \th (UK'V, WL'X)=(U\lambda KV, W\nu LX)$$
$$=(\lambda UKV, \nu WLX)=(\lambda M, \nu N).$$
As $|M|\le |K|=|K'|$, $|N|\le |L|=|L'|$, we can similarly show that $ (M,N) \th (\lambda^{-1} M, \nu^{-1} N)$. Thus $\th' \subseteq \th$, since $(M',N')=(\lambda'M, \nu'N)$, and
$(\lambda',\nu')$ belongs to the cyclic group generated by $(\lambda,\nu)$.

Conversely, it is straightforward to check that $\th'$ is a congruence containing $((K,L),(K',L'))$, therefore
$\th \subseteq \th'$. The result follows.
\end{proof}

\begin{lem} \label{lem:Icontmat}
Let $\theta$ be the principal congruence on $F_m\times F_n$ generated by $((K,L),(K',L'))$, and let $j = \max \{|K|,|K'|\}$.
If 
$(K,K')\not\in \mathcal{H} $,  then $\th_L$ or $\theta_{L'}$ contains the Rees congruence $\theta_{I_j}$ of $F_m$.
\end{lem}
We remark that the following proof is essentially equivalent to the proof of Lemma \ref{lem:Icont}. As the technical adaptations required to transform one lemma into the other 
are more complex 
than in previous cases,  we have decided to provide a complete proof. 
\begin{proof}
We will show that for $j=|K'|$, $\th_{I_j} \subseteq \th_{L'}$. An anologous result for $j=|K|$ follows symmetrically. So let us assume that $|K| \le |K'|=j$.
As $(K,K')  \not\in  \mathcal{H}$, $K$ and $K'$ must differ in either the image or the kernel. We consider
 two cases.

 {First case}:  $\im K \ne \im K'$.

As $|K'|=j\ge |K|$, then $\im K' \not \subseteq \im K$. As $F_m$ is regular, there exists an idempotent $M \in F_m$ such that $M \Rg K$. Hence $\im M= \im K$, and as $M$ is idempotent,
$MK=K$.

We have that 
\begin{eqnarray*}
(MK',L') & = &(M, E_n)(K',L') \\ 
 &\th & (M, E_n)(K,L)\\
 &=&(MK,L) \\
 &=&(K,L) \\
 &\th& (K',L')
\end{eqnarray*}
 and
so $(MK', K') \in \theta_{L'}$ on $F_m$. As $\im K' \not \subseteq \im K$, and $\im M=\im K$, the transformations  $MK'$ and $K'$ have different images, and 
it follows that $(MK',K') \notin\Hg$. Now, the congruence $\th'$ generated by $(MK',K')$ is contained in $\th_{L'}$ and by Theorem \ref{th:congrFnmat} and the remarks following it, we have $\th'=\th_{I_j}$. We get $\th_{I_j} \subseteq \theta_{L'}$.

{Second case}:  $\ker K \ne \ker K'$.

Now $|K'|=j \ge |K|$, implies that $\ker K \not\subseteq \ker K'$. As above, the regularity of $F_m$ implies that there exists an idempotent $M$ that is $\Lg$-related to $K$; thus $M$ and $K$ have the same kernel and $KM=K$.
 Hence
 \begin{eqnarray*}
(K'M,L')&=& (K',L')(M,E_n) \\
&\theta & (K,L)(M,E_n)\\
&=&(KM,L)\\
&=&(K,L)\\
& \theta &(K',L')
\end{eqnarray*}
and so $(K'M,K')\in \theta_{L'}$.
Now $\ker K = \ker M \subseteq \ker(K'M)$. As $\ker K \not \subseteq \ker K'$, we get $\ker K' \ne \ker (K'M)$ and so
$(K'M,K')\not\in \Hg$.  As above, by Theorem \ref{th:congrFnmat} and the remarks following it, 
we get $\th_{I_j} \subseteq \theta_{L'}$.
\end{proof}

\begin{thm}\label{theorem2.5mat}\label{th:0hmat}
Let $\theta$ be the congruence on $F_m \times F_n$ generated by $((K,L), (K',L'))$. 
Assume that $(K,K') \not\in \mathcal{H} $, $(L,L') \not\in \mathcal{H} $, $|K|=i, |K'|=j, |L|=k$, and $|L'|=l$.
Then $\theta$ is the Rees congruence on $F_m\times F_n$ defined by the ideal $J=I_i\times I_k
\cup  I_j \times I_l$.
\end{thm}
The proof of this theorem is once again essentially the proof of Theorem \ref{th:0h}. 

\begin{cor}
Under the conditions of Theorem \ref{theorem2.5mat}, if $i \le j$ and $k \le l$ then $\theta=\theta_{I_j \times I_l}$.
\end{cor}

The proof of the following Theorem corresponds to the proof of Theorem \ref{th:1h}.
\begin{thm}\label{th:1hmat}
Let $\theta$ be the congruence on $F_m \times F_n$ generated by $((K,L), (K',L'))$, and let $\th_2$ be the congruence on $F_n$ generated by $(L,L')$.
If $(L,L')\in \Hg$ and $(K,K')\not\in \Hg$; let $j=\mbox{max}\{|K|,|K'|\}$ and $k=|L|=|L'|$.
 Then  $(M,N) \th (M',N')$
 if and only if $(M,N)=(M',N')$ or
$|M|, |M'| \le j$, $|N|, |N'| \le k$, $N \th_2 N'$.
\end{thm}
Notice that we can once again give a more explicit description of $\th$ by incorporating the classification of $\th_2$ given by Theorem \ref{th:congrFnmat}.
\begin{cor}\label{c:1hmat}\begin{enumerate} \item
Let $(K,L),(K',L') \in F_m \times F_n$, such that $L'\ne \lambda L$ for any $\lambda \in F^*$, $(L,L')\in \Hg$ and $(K,K')\not\in \Hg$.
Let $j=\mbox{max}\{|K|,|K'|\}$ and $k=|L|=|L'|$. Let $s_1, s_2 \in \gl(n,F)$ be such that $s_1 L s_2$ and $s_1 L' s_2$ are in $\gl(k,F)$, and $H$ the normal subgroup of 
$\gl(k,F)$ generated by $(s_1 L s_2, s_1 L' s_2)$.
If $\theta$ is the congruence on $F_m \times F_n$ generated by $((K,L), (K',L'))$, then $(M,N) \th (M',N')$
 if and only if one of the following holds:
\begin{enumerate}
  \item $(M,N)=(M',N')$ for $|M| >j$ or $|N|>k$,
  \item $|M|,|M'| \le j$, $|N|=k$, $N \Hg N'$ and there exist $t_1, t_2 \in \gl(n,F)$ such that $t_1 N t_2$, $t_1 N' t_2 \in \gl(k,F)$ and lie in the same coset of $H$. 
    \item $|M|, |M'| \le j$ and $|N|, |N'| < k$.
\end{enumerate}
\item Let $(K,L),(K',L') \in F_m \times F_n$, such that $L'= \lambda L$ for some $\lambda \in F^*$,  and $(K,K')\not\in \Hg$.
Let $j=\mbox{max}\{|K|,|K'|\}$ and $k=|L|=|L'|$. Let $H$ be the subgroup of $F^*$ generated by $\lambda$. 
If $\theta$ is the congruence on $F_m \times F_n$ generated by $((K,L), (K',L'))$, then $(M,N) \th (M',N')$
 if and only if one of the following holds:
\begin{enumerate}
  \item $(M,N)=(M',N')$ for $|M| >j$ or $|N|>k$,
  \item $|M|,|M'| \le j$, $|N|\le k$, $N'= \lambda' N$ for some $\lambda' \in H$. 
\end{enumerate}
\end{enumerate}
\end{cor}

We remark that there is an obvious dual version of  Theorem \ref{th:1hmat} obtained by switching the roles of the coordinates.  

 \begin{thm}\label{lm:equalF2}
Let $\theta$ be a principal congruence on $F_m\times F_n$ generated by $((K,L),(K',L'))$, where $K'=\lambda K$ for $\lambda\in F^*$,  and $L \Hg L'$, such that 
$L'$ is not a scalar multiple of $L$.  
Let $|K|=i$, $|L|=j$, $\hat H$ the the normal subgroup of  $F^* \times \gl(j,F)$ associated with the pair $((K,L),(K',L'))$, and $G$ the subgroup of $F^*$ generated by $\lambda$.
Then, for $(M,N),(M',N') \in F_m \times F_n$, $(M,N)\theta(M',N')$ if and only if  one of the following holds:
\begin{enumerate}
  \item $(M,N)= (M',N')$;
  \item $|M|\le i$, $M' \Hg  M$, $|N|=j$, $N' \Hg N$ and the normal subgroup $\hat H'$ of  $F^* \times \gl(j,F)$ associated with the pair $((M,N),(M',N'))$ is contained in $\hat H$;
  \item $|M|\le i$, $M'=\lambda' M$ for some $\lambda' \in G$, $|N|, |N'|<j$.
\end{enumerate}
\end{thm}  
\begin{proof} Let $\th'$ be the relation defined by the statement of the theorem. It is straightforward to check that $\th'$ is a congruence.

Conversely, $\th$ contains the pairs from (1) trivially,   and the ones from (2) by Lemma \ref{lm:eHmat}.

It remains to show that $\th$ contains the pairs from (3). By multiplying with suitable $(E_n,s_1), (E_n,s_2) \in \gl(m,F)\times \gl(n,F)$ on the left and right, we may assume that $L=E_j$, $L' \in \gl(j,F)$, and that $ L'$ is not a scalar multiple of $E_j$. For the following considerations, we will identify  the ideal $F_n E_j F_n$ with $F_j$.

Now, as $L'$ is not a scalar multiple of $E_j$, by Lemma \ref{lm:subspace}, there exists a subspace $V$ of $F^j$ with dimension $j-1$ that is not preserved by $L'$. Let $A \in F_j$ be such that $A$ has rank $j-1$ and is the identity on $V$. Then $(K',L'A)=(K', L')(E_m, A) \th (K,L)(E_m,A)=(K,E_j A)$.

Considering those elements in $F_m \times F_n$ again, we see that $V= \im (E_j A) \ne \im (L'A)$, and so $(E_jA, L'A) \notin \Hg$. By applying Theorem \ref{th:1hmat} to
the pair $((K, E_jA),(K',L'A))$, we see that the pairs in (3) belong to $\th$.
\end{proof}

In view of  Theorem  \ref{lm:equalF1}, Theorem \ref{th:0hmat}, Theorem \ref{th:1hmat},
Theorem \ref{lm:equalF2}, and, where applicable, their dual versions, it remains to determine the principal congruence $\th$ when 
$K \mathcal{H} K$, $L \mathcal{H}  L'$, $K' \ne \lambda K$ for all $\lambda \in F^*$ and $L' \ne \lambda L$ for all $\lambda \in F^*$.

 \begin{thm}\label{lm:HHF}
Let $\theta$ be a principal congruence on $F_m\times F_n$ generated by $((K,L),(K',L'))$, where $K \Hg K'$,  and $L \Hg L'$, such that $K'$ is not a scalar multiple of $K$ and
$L'$ is not a scalar multiple of $L$.  
Let $|K|=i$, $|L|=j$, $H$ the the normal subgroup of  $\gl(i, F) \times \gl(j,F)$ associated with the pair $((K,L),(K',L'))$.
Then for $(M,N),(M',N') \in F_m \times F_n$, $(M,N)\theta(M',N')$ if and only if  one of the following holds:
\begin{enumerate}
  \item $(M,N)= (M',N')$;
  \item $|M|=i$, $M' \Hg  M$, $|N|=j$, $N' \Hg N$ and the normal subgroup $H'$ of  $\gl(i, F) \times \gl(j,F)$ associated with the pair $((M,N),(M',N'))$ is contained in $H$;
  \item $|M|=i$, $M' \Hg M$, $|N|,|N'| < j$, and there exists $s_1,s_2 \in \gl(m,F)$ such that $s_1Ms_2$ and $s_2M's_2$ are both in $\gl(i,F)$ and belong to the same coset of 
  $\gl(i,F)$
  modulo $\pi_1(H)$;
  \item $|N|=j, N' \Hg N$, $|M|,|M'| <i$, and there exists $s_1,s_2 \in \gl(n,F)$ such that $s_1Ns_2$ and $s_2N's_2$ are both in $\gl(j,F)$ and belong to the same coset of 
  $\gl(j,F)$
  modulo $\pi_2(H)$;
  \item $|M|, |M'|<i$, $|N|,|N'|< j$.
\end{enumerate}
\end{thm}  
\begin{proof}
Let $\th'$ be the relation defined by the statement of the theorem. It is straightforward to check that $\th'$ is a congruence.

Conversely, $\th$ contains the pairs from (1) trivially,   and the ones from (2) by Lemma \ref{lm:HHmat}.

It remains to show that $\th$ contains the pairs from (3), (4), and (5). Using an analogous argument to the last part of the proof of Theorem \ref{lm:equalF2}, we can find an 
$A \in F_n$ such that
$(K',L'A)\th(K,E_j A)$, where $L'A$ and $E_jA$ have rank $j-1$ but have different images and hence are not $\Hg$-related. 
Now the congruence $\beta$ generated by $((K',L'A),(K,E_jA))$ is contained in $\th$ and an application of Corollary \ref{c:1hmat}
to $\beta$ shows that $\beta$ contains all pairs from (3) and (5), and then so does $\th$.

Finally, $\th$ contains the pairs in (4) by a symmetric argument.

\end{proof}

\section*{Appendix}

This appendix contains the proof of the following purely group theoretic result.
\begin{thm}
The normal subgroups  $N$ of $ S_i \times S_k$ are exactly the ones of the following form:
\begin{enumerate}
  \item $N=N_1 \times N_2$, where $N_1 \unlhd S_i, N_2 \unlhd S_k$.
  \item $N=\{(a,b) \in S_i \times S_j\,|\, \mbox{$a$ and $b$ have the same parity}\}$.
\end{enumerate}

\end{thm}
\begin{proof}
It is clear that the listed items are normal subgroups of $S_i \times S_k$. Conversely, let $N \unlhd S_i \times S_k$. Note that if $(x,y) \in N$ then by normality we get that
$(gxg^{-1}, y), (x, hyh^{-1}) \in N$ for all $g \in S_i, h \in S_k$.

Let $N_1= \{x \in S_i \,|\, (x,\id_{S_k}) \in N\}, N_2=\{y \in S_k \,|\, (\id_{S_i},y) \in N\}$. Clearly, $N_1, N_2$ are normal subgroups of $S_i, S_k$, respectively, hence they
are either the trivial, alternating, or symmetric groups of the corresponding degree, or the normal subgroup
$V_4$ of $S_4$. It is also clear that $N_1 \times N_2 \subseteq N$. If $N= N_1 \times N_2$ then $N$ is of the form (1).

Assume that $N$ is not of the form $N_1 \times N_2$. Let
$(x,y) \in N \setminus N_1 \times N_2$, say $x \notin N_1$. Then
$y \notin N_2$ either, for otherwise $(x,\id)=(x,y)(\id, y^{-1}) \in N$, contradicting $x \notin N_1$. Similarly $y \notin N_2$ implies $x \notin N_1$. Hence we may conclude in this case that $N_1 \ne S_i$ and $N_2 \ne S_k$.

Let $N_1 \ne A_i$, and note that this implies that $i \ge 3$. Suppose that $N_1= \varepsilon_i$. If $(x',y) \in N$ then $(x'x^{-1}, \id)=(x',y) (x^{-1},y^{-1})\in N$, hence $x'x^{-1} \in N_1=\varepsilon_i$
and therefore $x=x'$. However, if
$(x,y) \in N$, then $(\bar x,y) \in N$ for every conjugate $\bar x$ of $x$. It follows that $x$ is a central element of $S_i$. However $x \notin N_1= \varepsilon_i$. As  the identity is the only central element of $S_i$ for $i \ge 3$, we arrive at a contradiction, and hence $N_1 \ne \varepsilon_i$.

We now must discuss the possibility of having $i=4$ and
 $N_1=V_4$. In this case, $N$ contains exactly $4$ elements of the form $(x',y)$, and as before, for every conjugate $\bar x$ of $x$, $(\bar x, y ) \in N$. Recall that $V_4$ consists of the identity and all products of two disjoint transpositions in $S_4$. As $x \notin N_1=V_4$, then $x$ must be either a transposition, a $3$-cycle or a $4$-cycle. However, the $S_4$-conjugacy classes corresponding to these cycle structures  have sizes $6$, $8$, and $6$, respectively, and so they cannot all be of the form $(x,y)(V_4 \times \varepsilon_k)$. As before, we may conclude that $N_1 \ne V_4$.

Dual results hold for $N_2$, and hence we are left to analyse the case that $N_1=A_i$ and $N_2=A_k$. If $i=1$ or $k=1$, then $A_i=S_1$ or $A_k=S_1$ and these cases have been covered. So assume that $i,k\ge 2$. Then
$x$ and $y$ are odd permutations and $A_i \times A_k \cup \{(x,y)\}\left(A_i \times A_k\right)\subseteq N$. The reverse inclusion holds, as we have already shown that $N$ does not contain elements $(x',y')$ with
$x' \in N_1$ and $y' \notin N_2$, or vice versa. Therefore $N=A_i \times A_k \cup \{(x,y)\}\left(A_i \times A_k\right)$ and hence is of the form (2).
\end{proof}

\section{Problems} 

In this final section we are going to propose a number of problems that arise naturally from the results above.

In \cite{ST}, the description of the congruences of $\T_n$ (provided in \cite{Malcev}) is used to describe the endomorphisms of $\T_n$. This suggests the following problem. 

\begin{prob}
Let $Q\in \{\PT,\T,\In\}$. Describe the endomorphisms of the monoid $Q_n\times Q_m$. 
\end{prob}

Regarding the monoid $F_n$, its congruences are known since $1953$ \cite{Malcev2}, but the description of $\End(F_n)$ is still to be done. 

\begin{prob}
Describe the endomorphisms of the monoid $F_n$, the monoid of all $n\times n$ matrices with entries in the field $F$.
\end{prob}

Once the previous  problem is solved, the following is also natural given the results in this paper. 

\begin{prob}  
 Use the description of the congruences on  $S=F_n\times F_m$ to classify the endomorphisms of $S$. 
\end{prob}

Our original goal was the description of the congruences on the direct product of some classic transformation monoids; of course, there are many more such classes whose congruences have also been classified (see for example \cite{derek,fern1,fern2,fern3,fern4,suzana}). Therefore  the following is a natural question. 

\begin{prob}
Let $Q_n$ be any transformation monoid or semigroup whose congruences have been described (in particular, let $Q_n$ be one of the monoids or semigroups considered in the papers \cite{derek,fern1,fern2,fern3,fern4}). Describe the congruences of $Q_n\times Q_m$. 
\end{prob}

Let $G$ be a permutation group contained in the symmetric group $S_n$ and let $t\in \T_n$; let $\langle G,t\rangle$ denote the monoid generated by $G$ and $t$. These monoids (in which the group of units is prescribed) proved to be a source of very exciting new results, especially because usually they lead to considerations involving several different parts of mathematics such as groups, combinatorics, number theory, semigroups, linear algebra, computational algebra, etc.  (For an illustration of what has just been said, see \cite{andre2,ABC,ABCRS,ArBeMiSc,circulant2,ArCa13,ArCa14,ArCa12,acmn,ArCaSt15,ArDoKo,ArMiSc,ben,Ata15,Le85,Le87,Le96,Le99,levi00,lmm,lm,15,mcalister,neu,symo} and the references therein.)
In particular, in \cite{levi00}, the congruences of the monoids $\langle S_n,t\rangle$ are described. In a more general context one may ask the following. 

\begin{prob}\label{p7.5}
Let $G\le S_n$ be a permutation group and let $t\in \T_n\setminus S_n$ be a transformation. Describe the congruences and the endomorphisms of the monoid $\langle G,t\rangle$.  
\end{prob}

The next problem connects  the Problem \ref{p7.5} with  the main results of this paper. 
\begin{prob}\label{p7.6}
Let $G\le S_n$ be a permutation group and let $t,q\in \T_n\setminus S_n$ be two transformations. Describe the congruences and the endomorphisms of the monoid $\langle G,t\rangle\times \langle G,q\rangle$.  
\end{prob}

In order to solve the previous problems, it seems sensible to start by solving them for especially relevant particular classes of groups. 

\begin{prob}
Let $G\le S_n$ be a transitive [respectively, imprimitive, $3$-transitive, $2$-transitive, $3$-homogeneous, $2$-homogeneous, primitive] group. Solve Problem \ref{p7.5} and Problem \ref{p7.6} assuming that $G$ belongs to one of these classes.  
\end{prob}

Next, notice that the previous problems admit linear analogous. 

\begin{prob}
Let $V$ be a finite dimension vector space and let $G\le \Aut(V)$; let $t\in \End(V)\setminus \Aut(V)$. Describe the congruences and the endomorphisms of the monoid $\langle G,t\rangle$. 
\end{prob}

\begin{prob}
Let $V$ be a finite dimension vector space and let $G\le \Aut(V)$; let $t,p\in \End(V)\setminus \Aut(V)$. Describe the congruences and the endomorphisms on $\langle G,t\rangle\times \langle G,p\rangle$. 
\end{prob}

For some results on linear monoids of the form $\langle G,t\rangle$ we refer the reader to \cite{fcsilva}.

Once a result is proved for the endomorphism monoid of a set with $n$ elements, that is, for $\T_n$, and it is also proved for the endomorphism monoid of a finite dimension vector space, the next natural step is to ask for an analogous theorem for the endomorphism monoid of an independence algebra. (Recall that sets and vector spaces are examples of independence algebras; for definitions and results see \cite{araujo,aeg,arfo,cameron,fou1,fou2,gould,gh,mar1,nark,nar3,z88b,z89}.)		
 \begin{prob}\label{p7.10}
 Let $A$ be a finite dimension independence algebra. Describe the congruences on $\End(A)$. 
 \end{prob}
 
 Possibly the best approach to attack Problem \ref{p7.10} is to rely on the classification theorem of these algebras (\cite{cameron,ur,urbanik}), as  in \cite{abk}.

 Once the previous question has been solved, the next task should be to tackle direct products. 
 
 \begin{prob}
 Let $A$ and $B$ be finite dimension independence algebras. Describe the congruences on $\End(A)\times End(B)$. 
 \end{prob}
 
 Fountain and Gould defined the class of \emph{weak exchange algebras}, which contains independence algebras,  among others \cite{FoGo03,FoGo04}.

\begin{prob}
Solve problems, analogous to the two previous, for weak exchange algebras, or some of its subclasses (such as basis algebras, weak independence algebras, etc.; for definitions see \cite{FoGo03,FoGo04}).
\end{prob}

In a different direction, in \cite{ArWe} a large number of other classes generalizing 
independence algebras were introduced, and again similar questions can be posed.

\begin{prob}
Solve problems analogous to the ones above 
for $MC$-algebras, $\mathit{MS}$-algebras, $\mathit{SC}$-algebras, and $\mathit{SC}$-ranked algebras.
\end{prob}

A first step towards the solution of the previous problem would be to solve them for an
$\mathit{SC}$-ranked free $M$-act \cite[Chapter 9]{ArWe},
and for an $\mathit{SC}$-ranked free module over an $\aleph_1$-Noetherian ring \cite[Chapter 10]{ArWe}.

In the introduction we mentioned  that we were driven to the results in this paper by some considerations on centralizers of idempotents, that we now introduce.  

Let $e^2=e\in \T_n$, and denote by $C_{\T_n}(e)$ the centralizer of the idempotent $e$ in $\T_n$, that is, 
\[
C_{\T_n}(e)=\{f\in \T_n\mid fe=ef\}.
\]
The monoid $C_{\T_n}(e)$ has many very interesting features, starting with the fact that it is a generalization of both $\T_n$ and $\PT_n$ (see \cite{andre,arko2,arko1,ak}). 

\begin{prob}
Describe the congruences on $C_{\T_n}(e)$. 
\end{prob}

This is certainly a very difficult problem. One possibly more feasible challenge is the following. 

\begin{prob}
Describe the congruences on the monoids  $C_{\T_n}(e)$ that are regular. 
\end{prob}

The solution of  this problem requires a complete description of the congruences of $\PT_n\times \PT_m$, and that was what prompted us to write this paper. This description is necessary but it is not sufficient. In fact it is also necessary to solve the following. 

\begin{prob}
Let $\Gamma$ be a finite set of finite  chains and let $S$ be the direct product of the chains in $\Gamma$. This direct product can be seen as a commutative semigroup of idempotents. Describe the congruences on this semilattice.   
\end{prob}

In this paper we consider monoids of the form $Q_n\times Q_m$, where $Q\in \{\PT,\T,\In\}$. However, as pointed out in Observation \ref{obs}, it is not necessary  to restrict the direct product to semigroups of the same \emph{type}. 

\begin{prob}
Let $S$ and $T$ be any two transformation semigroups belonging to classes of semigroups whose congruences have already been described. Find the endomorphisms and congruences on $S\times T$.  
\end{prob}

%

Finally, the partition monoid \cite{chwas,east1,east3,east4,east4b),east5,east5b),east2,fitzLau} has a very rich structure and hence has been attracting increasing attention. The problems mentioned before when considered in the context of partition monoids
are certainly great challenges, but given the growing importance of these monoids they are certainly natural.

\begin{prob}
Describe the congruences and endomorphisms of the partition monoid. Solve similar problems for the direct product of two partition monoids. 
\end{prob}

\section*{Acknowledgements}
This work was developed within FCT projects CAUL (PEst-OE/ MAT/UI0143/2014)  and CEMAT-CI\^{E}NCIAS (UID/Multi/04621/2013).


The second author acknowledges support from the
European Union Seventh Framework Programme (FP7/2007-2013) under
grant agreement no.\ PCOFUND-GA-2009-246542 and from the Foundation for
Science and Technology of Portugal under SFRH/BCC/52684/2014.



\begin{thebibliography}{10}


 \bibitem{chwas}
    C. Ahmed, P. Martin and Volodymyr Mazorchuk. 
   \emph{On the number of principal ideals in d-tonal partition monoids.}
 arXiv:1503.06718  

\bibitem{andre}
J. M. Andr\'{e}, J. Ara\'{u}jo and J. Konieczny, Regular centralizers of idempotent transformations. 
\emph{Semigroup Forum}  {\bf 82} (2)  (2011), 307--318.


\bibitem{andre2}
J. Andr\'{e}, J. Ara\'{u}jo and P.J. Cameron.
The classification of partition homogeneous groups with applications to semigroup theory. \emph{To appear J. Algebra.}

{http://arxiv.org/abs/1304.7391}
%
%

\bibitem{araujo}
J.\ Ara\'{u}jo,
Idempotent generated endomorphisms of an independence algebra,
\emph{Semigroup Forum} \textbf{67} (2003), 464--467.

\bibitem{abk}
J.\ Ara\'{u}jo, W. Bentz and J. Konieczny,
 The largest subsemilattices of the semigroup of endomorphisms of an
independence algebra,
\emph{Linear Algebra and its Applications} \textbf{458} (2014), 50--79.

\bibitem{ABC}
J.~Ara\'ujo, W. Bentz and P. J. Cameron.
\newblock{Groups synchronizing a transformation of non-uniform kernel.
{\em Theoret. Comput. Sci.} {\bf 498} (2013), 1--9. }

\bibitem{ABCRS}
J.~Ara\'ujo, W. Bentz, P. J. Cameron, G. Royle and A. Schaefer,
\newblock{Primitive groups and synchronization}
\newblock{http://arxiv.org/abs/1504.01629}



\bibitem{ArBeMiSc}
J.~Ara\'ujo, W. Bentz, J. D. Mitchell and C. Schneider.
\newblock{The rank of the semigroup of transformations stabilising a partition of a finite set.}
\newblock{ {\em Mathematical Proceedings of the Cambridge Philosophical Society} \textbf{159} (2) (2015), 339--353.}

\bibitem{circulant2}
J.~Ara\'ujo, W. Bentz, E. Dobson, J. Konieczny and J. Morris.
\newblock{Automorphism groups of circulant digraphs
with applications to semigroup theory.}
\newblock{ {\em To appear Combinatorica.}}

\bibitem{ArCa13}
J.~Ara\'ujo and P. J. Cameron.
\newblock{Permutation groups and transformation semigroups: results and problems.}
{\it Groups St Andrews 2013 London Mathematical Society Lecture Note Series}, 128--141. 


\bibitem{ArCa14}
J.~Ara\'ujo and P. J. Cameron.
\newblock{Primitive groups synchronize non-uniform maps of extreme ranks.}
\newblock{\emph{Journal of Combinatorial Theory, Series B}}, {\bf 106} (2014), 98--114. 

\bibitem{ArCa12}
J.~Ara\'ujo and P. J. Cameron.
\newblock{Two Generalizations of Homogeneity in Groups with Applications to Regular Semigroups.} 
\newblock{{\emph Transactions American Mathematical Society} {\bf 368} (2016), 1159--1188.
}


\bibitem{acmn}
J.~Ara\'ujo, P. J. Cameron, J. D. Mitchell, M. Neuhoffer.
\newblock{The classification of normalizing groups.}
\newblock{ {\em Journal of Algebra} \textbf{373} (2013), 481--490.}

\bibitem{ArCaSt15}
J.~Ara\'ujo, P. J. Cameron and B. Steinberg.
\newblock{Between primitive and 2-transitive:
Synchronization and its friends}
{\it to appear}. 


\bibitem{ArDoKo}
J.~Ara\'ujo, E. Dobson, J. Konieczny.
\newblock{Automorphisms of endomorphism semigroups of reflexive digraphs.}
\newblock{ {\em Mathematische Nachrichten} \textbf{283} (7) (2010), 939--964.}


\bibitem{aeg}
J.\ Ara\'{u}jo,  M.\ Edmundo, and S. Givant,
$v^{*}$-Algebras, independence algebras and logic,
\emph{Internat. J. Algebra Comput.} \textbf{21} (2011), 1237--1257.


\bibitem{arfo}
J.\ Ara\'{u}jo and J.\ Fountain,
The origins of independence algebras, in
\emph{Semigroups and languages}, 54--67, World Sci.\ Publ., River Edge, NJ,
2004.


\bibitem{arko2}
J. Ara\'{u}jo and J. Konieczny,
Automorphisms groups of centralizers of idempotents,
\emph{J. Algebra} \textbf{269} (2003), 227--239.

\bibitem{arko1}
J. Ara\'{u}jo and J. Konieczny,
Semigroups of transformations preserving an equivalence relation and a
cross-section,
\emph{Comm. Algebra} \textbf{32} (2004), 1917--1935.

\bibitem{ak}
J. Ara\'{u}jo and J. Konieczny,
Centralizers in the full transformation semigroup,
\emph{Semigroup Forum} \textbf{86} (2013), 1--31.

\bibitem{ArMiSc}
J.~Ara\'ujo, J. D. Mitchell and C. Schneider.
\newblock{Groups that together with any transformation generate regular semigroups or idempotent generated semigroups.}
\newblock{ {\em Journal of Algebra} \textbf{343} (1) (2011), 93--106.}

\bibitem{fcsilva}
J. Ara\'{u}jo, and F.C. Silva, Semigroups of linear endomorphisms closed under conjugation, \emph{Comm. Algebra}  {\bf 28} (8) (2000), 3679--3689.

\bibitem{ArWe}
J.~Ara\'ujo and F. Wehrung,
 Embedding properties of endomorphism semigroups,
\emph{Fund.\ Math.} {\bf 202} (2009), 125--146.



\bibitem{ben}
F. Arnold and B. Steinberg.
\newblock{Synchronizing groups and automata.}
\newblock{ \textit{  Theoret. Comput. Sci.} \textbf{359} (2006), no. 1-3, 101--110.}

\bibitem{Ata15}
V.S. Atabekyan, 
The automorphisms of endomorphism semigroups of free burnside groups.
{\em Internat. J. Algebra Comput.} {\bf 25} (2015), no. 4, 669--674. 


\bibitem{cameron}
P. J. Cameron and C. Szab\'{o},
Independence algebras,
\emph{J. London Math. Soc.} \textbf{61} (2000), 321--334.

\bibitem{derek}
V.D. Derech, On quasi-orders on some inverse semigroups, \emph{Izvestiya VUZ. Matematika} {\bf 3}, (1991), 76--78. (Russian).


\bibitem{east1}
J. East, Infinite partition monoids. Internat. \emph{J. Algebra Comput.} {\bf 24} (2014), no. 4, 429--460. 


\bibitem{east3} 
J.  East, Generators and relations for partition monoids and algebras. \emph{J. Algebra} {\bf 339} (2011), 1--26. 

\bibitem{east4} 
J. East, On the singular part of the partition monoid. \emph{Internat. J. Algebra Comput.} {\bf 21} (2011), no. 1-2, 147--178. 

\bibitem{east4b)}
J. East, I. Dolinka, Idempotent generation in the endomorphism monoid of a uniform partition. \emph{(to appear)}

\bibitem{east5} 
J. East, I. Dolinka, A. Evangelou, D. FitzGerald, N. Ham, J. Hyde and N. Loughlin, Enumeration of idempotents in diagram semigroups and algebras. \emph{J. Combin. Theory Ser. A.}, {\bf 131} (2015), 119--152. 

\bibitem{east5b)} 
J. East, I. Dolinka and James Mitchell, Idempotent rank in the endomorphism monoid of a non-uniform partition. \emph{(to appear)} 

\bibitem{east2} 
J. East and D.G. FitzGerald, The semigroup generated by the idempotents of a partition monoid. \emph{J. Algebra} {\bf 372} (2012), 108--133. 

\bibitem{fitzLau} 
D.G. FitzGerald and K. W. Lau, On the partition monoid and some related semigroups. \emph{Bull. Aust. Math. Soc.} {\bf 83} (2011), no. 2, 273--288. 

\bibitem{fern1}
V.H. Fernandes, The monoid of all injective orientation preserving partial transformations on a finite chain, \emph{Comm. Algebra.} {\bf 28} (2000),  3401--3426.

\bibitem{fern2}
Fernandes, V. H. "The monoid of all injective order preserving partial transformations on a finite chain." Semigroup Forum. 62 (2001): 178-204.


\bibitem{fern3}
V.H. Fernandes, Gracinda M. S. Gomes, and Manuel M. Jesus, Congruences on monoids of order-preserving or order-reversing transformations on a finite chain, \emph{Glasg. Math. J.} {\bf 47} (2005), 413--424.

\bibitem{fern4}
V.H. Fernandes, Gracinda M. S. Gomes, and Manuel M. Jesus, Congruences on monoids of transformations preserving the orientation of a finite chain, \emph{J. Algebra} {\bf 321} (2009), 743--757.

\bibitem{FoGo03}
J.Fountain and V. Gould, Relatively free algebras with weak exchange properties,
\emph{J. Aust.\ Math.\ Soc.} {\bf 75} (2003), 355--384.

\bibitem{FoGo04}
J.Fountain and V. Gould, Endomorphisms of relatively free algebras with weak exchange properties,
\emph{Algebra Universalis} {\bf 51} (2004), 257--285.


\bibitem{fou1}
J.\ Fountain and A.\ Lewin, Products of idempotent
endomorphisms of an independence algebra of finite rank,
\emph{Proc.\ Edinburgh Math.\ Soc.} \textbf{35} (1992), 493--500.

\bibitem{fou2}
J.\ Fountain and A.\ Lewin,
Products of idempotent endomorphisms of an independence algebra of infinite
rank,
\emph{Math.\ Proc.\ Cambridge\ Philos.\ Soc.} \textbf{114} (1993), 303--319.

\bibitem{Marzo}
 O. Ganyushkin and V.  Mazorchuk, Classical finite transformation semigroups. An introduction. Algebra and Applications, 9. Springer-Verlag London, Ltd., London, 2009.



\bibitem{gh}
G.M.S. Gomes and J.M. Howie,
Idempotent endomorphisms of an independence algebra of finite rank.
\emph{Proc. Edinburgh Math. Soc.} (2) {\bf 38} (1995), no. 1, 107--116.
 
 \bibitem{gould}
V. Gould,
Independence algebras,
\emph{Algebra Universalis} \textbf{33} (1995), 294--318.

 
\bibitem{Ho95}
John~M. Howie.
\newblock{{\em Fundamentals of semigroup theory}, volume~12 of {\em London
  Mathematical Society Monographs. New Series}.}
\newblock{The Clarendon Press Oxford University Press, New York, 1995.}
\newblock{Oxford Science Publications.}

\bibitem{MatWeb} A. Kudryavtseva and V. Mazorchuk, Square Matrices as a Semigroup, http://www2.math.uu.se/research/pub/Mazorchuk9.pdf, July 6, 2015.

 \bibitem{Le85}
I. Levi, Automorphisms of normal transformation semigroups.
{\it Proc.\ Edinburgh Math.\ Soc.\ (2)\/} {\bf 28} (1985), 185--205.

\bibitem{Le87}
I. Levi. Automorphisms of normal partial transformation semigroups,
{\it Glasgow Math.\ J.\/} {\bf 29} (1987), 149--157.

\bibitem{Le96}
I. Levi. On the inner automorphisms of finite transformation semigroups,
{\it Proc.\ Edinburgh Math.\ Soc.\ (2)\/} {\bf 39} (1996), 27--30.
\bibitem{Le99}
I. Levi. On groups associated with transformation semigroups,
{\it Semigroup Forum\/} {\bf 59} (1999), 342--353.


 
\bibitem{levi00}
I. Levi.
\newblock{Congruences on normal transformation semigroups.}
\newblock{\textit{Math. Japon. }, \textbf{52} (2) (2000), 247--261.}

\bibitem{lmm}
I.~Levi, D.~B. McAlister, and R.~B. McFadden.
\newblock{Groups associated with finite transformation semigroups.}
\newblock{{\em Semigroup Forum}, \textbf{61} (3) (2000), 453--467.}

\bibitem{lm}
I.~Levi and R.~B. McFadden.
\newblock{ {$S\sb n$}-normal semigroups.}
\newblock{{\em Proc. Edinburgh Math. Soc. (2)}, \textbf{37} (3) (1994), 471--476.}



\bibitem{15}
Inessa Levi and Steve Seif.
\newblock Combinatorial techniques for determining rank and idempotent rank of
  certain finite semigroups.
\newblock {\em Proc. Edinb. Math. Soc. (2)}, 45(3):617--630, 2002.



\bibitem{Liber} A. Liber, On symmetric generalized groups, \emph{Mat. Sbornik N.S.}  {\bf 33} (75), (1953), 531--544.
\bibitem{Malcev} A. Mal\'cev, Symmetric groupoids, \emph{Mat. Sbornik N.S.} {\bf 31} (73), (1952), 136--151.
\bibitem{Malcev2} A. Mal\'cev, Multiplicative congruences of matrices, \emph{Doklady Akad. N.S.} {\bf 90} (1953), 333--335.

\bibitem{mar1}
E.\ Marczewski, {A general scheme of the notions of
independence in mathematics}, \emph{Bull.\ Acad.\ Polon.\ Sci.}\ {\bf
6} (1958), 731--736.

\bibitem{mcalister}
Donald~B. McAlister.
\newblock{Semigroups generated by a group and an idempotent.}
\newblock{{\em Comm. Algebra}, \textbf{26} (2) (1998), 515--547.}


\bibitem{nark}
W.\ Narkiewicz,  {Independence in a certain class of abstract
algebras}, \emph{Fund.\ Math.}\ {\bf 50} (1961/62), 333--340.

%
\bibitem{nar3}
W.\ Narkiewicz, {On a certain class of abstract algebras},
\emph{Fund.\ Math.} {\bf 54} (1964), 115--124.

\bibitem{neu}
P. M. Neumann.
\newblock{Primitive permutation groups and their
section-regular partitions.}
\newblock{ \textit{Michigan Math. J.} {\bf 58} (2009), 309--322.}


\bibitem{ST} B. Schein and B. Teclezghi, Endomorphisms of finite full transformation semigroups. \emph{Proc. Amer. Math. Soc.} {\bf 126} (1998), no.9, 2579--2587.

\bibitem{suzana}
S.  Mendes-Gon�alves and R. P. Sullivan, The ideal structure of semigroups of transformations with restricted range. \emph{Bull. Aust. Math. Soc.} {\bf 83} (2011), no. 2, 289--300. 
 
\bibitem{Sutov} E. Sutov, Homomorphisms of the semigroup of all partial transformations. Izv. Vyssh. Uchebn. Zaved. Mathematika {\bf 22}, no. 3, (1961), 177-184.

\bibitem{symo}
J.S.V. Symons.
\textit{Normal transformation semigroups}.
J. Austral. Math. Soc. Ser. A \textbf{22} (1976), no. 4, 385--390.


\bibitem{ur} K.\ Urbanik,
{A representation theorem for $v^*$-algebras},
\textit{Fund.\ Math.},   \textbf{52} (1963),
291--317.

%
\bibitem{urbanik}
K.\ Urbanik, { Linear independence in abstract algebras},
\textit{Colloq.\ Math.} {\bf 14} (1966),   233--255.

%
\bibitem{z88b}B. \ I.\ Zilber
{Hereditary transitive groups and quasi-Urbanik structures} (Russian), pp. 58--77 in {Model Theory and its Application}, \textit{Proc. Math. Inst. Sib. Branch Ac. Sci. USSR}, (Novosibirsk 1988), Ed. Yu. Ershov. (English translation: Amer. Math. Soc. Transl. {\bf 195} (2) (1999), pp. \ 165--180.)



%
\bibitem{z89}B.\ I.\ Zilber
  { Quasi-Urbanik structures} (Russian), pp. 50--67 in {Model-theoretic algebra},
   \textit{Collect.\ Sci.\ Works},  (Alma-Ata 1989).









%
%
%
%
%
%
%
%
%
%
%
%
%
%
%
%
%
\end{thebibliography}
\end{document}